\documentclass[11pt,a4paper,oneside,DIV=12,headings=small]{scrartcl}

\usepackage{tikz}     % OK
\usepackage{mathptmx}    % OK
\usepackage{graphicx}    % OK
\usepackage{amsmath}     % OK
\usepackage{amssymb}     % OK
\usepackage{amsfonts}   
\usepackage{amsthm} % OK
\usepackage{enumerate} 
\usepackage{arydshln}
\usepackage[shortlabels]{enumitem}
\setlist{noitemsep,topsep=0pt,parsep=0pt,partopsep=0pt}
\usepackage{lipsum}  % OBREISATI

\usepackage{pgf}
\usetikzlibrary{arrows, backgrounds}
\usetikzlibrary{shadows}

\usetikzlibrary{decorations.pathreplacing,angles,quotes}

\usepackage{subfigure}
\usepackage{epstopdf}
\usepackage{caption}
\usepackage{multirow} 
\usepackage{multicol} 
\newcommand{\Rset}{\mathbb{R}}

\newcommand{\cS}{\mathcal{S}}

\newcommand{\cQ}{\mathcal{Q}}

\newcommand{\cN}{\mathcal{N}}

\newcommand{\cM}{\mathcal{M}}

\newcommand{\cR}{\mathcal{R}}

\newcommand{\cA}{\mathcal{A}}

\newcommand{\diag}{\operatorname{diag}}

\newcommand{\col}{\operatorname{col}}
\newcommand{\image}{\operatorname{Im}}

\newcommand{\kernel}{\operatorname{Ker}}

\newtheorem{theorem}{Theorem}[section]
\newtheorem{proposition}[theorem]{Proposition}
\newtheorem{lemma}[theorem]{Lemma}

\newtheorem{corollary}[theorem]{Corollary}

\theoremstyle{plain}
\newtheorem{remark}[theorem]{Remark}

\newtheorem{example}[theorem]{Example}
\title{
On Structured Lyapunov Functions and Dissipativity in Interconnected LTI Systems
}

\author{Andrej Joki\'{c} and Ivica Naki\'{c}% <-this % stops a space
\thanks{A.\ Joki\'{c} is with the Faculty of Mechanical Engineering and Naval Architecture, University of Zagreb, Ivana Lu\v{c}i{\'c}a 5, 10000 Zagreb, Croatia, tel: +385921013815, e-mail: andrej.jokic@fsb.hr.}
\thanks{Ivica Naki\'{c} is with the Department of Mathematics,
        Faculty of Science, University of Zagreb,
        e-mail: nakic@math.hr.}%
\thanks{Correspondence address: A. Joki\'{c}, Faculty of Mechanical Engineering and Naval Architecture, University of Zagreb, Ivana Lu\v{c}i{\'c}a 5, 10000 Zagreb, Croatia.}
\thanks{This research has been supported by the Croatian Science Foundation under the projects 9354 ``Control of Spatially Distributed Systems'', 9345 ``MMACACMS'', IP-2016-06-2468 ``Control of Dynamical Systems''
, and by the European Regional Development Fund under the grant KK.01.1.1.01.0009 (DATACROSS).}
}

\begin{document}

\maketitle

\begin{abstract}
In this paper we study connections between structured storage or Lyapunov functions of a class of interconnected systems (dynamical networks) and dissipativity properties of the individual systems. We prove that if a dynamical network, composed as a set of linear time invariant (LTI) systems interconnected over an acyclic graph, admits an additive quadratic Lyapunov function, then the individual systems in the network are dissipative with respect to a (nonempty) set of interconnection neutral supply functions. Each supply function from this set is defined on a single interconnection link in the network. Specific characterizations of neutral supply functions are presented which imply robustness of network stability/dissiptivity to removal of interconnection links. 
\end{abstract}

%%%%%%%%%%%%%%%%%%%%%%%%%%%%%%%%%%%%%%%%%%%%%%%%%%%%%%%%%%%%%%%%%%%%%%%%%%%%%%%%
\section{INTRODUCTION}
\label{Sec1}

In this paper we consider dissipative dynamical systems as introduced by Jan Willems in \cite{Willems_1}, \cite{Willems_2}. 
The dynamical system $\dot{x}=f(x,d)$, $z=g(x,d)$, is defined to be dissipative with respect to the supply function $s(\cdot,\cdot)$ if there exists a storage $V:X\rightarrow \Rset$ such that the dissipation inequality
\begin{equation}
\label{DssipativityInequality1}
V(x(t_1))\leq V(x(t_0))+ \int_{t_0}^{t_1} s(d(t),z(t))dt 
\end{equation}  
holds for all trajectories of $x, d, z$ satisfying the system's dynamics, for all $t\in [t_0,t_1]$ and for all $t_0<t_1$. 
The supply function quantifies the power supplied to (or extracted from) the system, while a storage function\footnote{In general definition we do not require a storage function to be non-negative, in accordance with the definition in \cite{Willems_Trentleman}. In the paper we explicitly indicate when some particular storage function is assumed/required to be non-negative.} quantifies the energy stored within the system at any given moment. The dissipation inequality implies that the difference in the stored energy over any finite time interval cannot exceed the amount of energy supplied to the system over the same time interval.  

Being an extension of Lyapunov theory to open systems, i.e., systems with exogenous inputs and outputs, with inflow/outflow of power, the dissipativity theory is one of the major tools in both \emph{i)} robust control theory (see, e.g., \cite{DullerudPaganini, SchererWeiland}), where many of the problems can be formulated, solved or interpreted in this framework; and \emph{ii)} stability analysis/control synthesis for large scale systems, see e.g. \cite{MoylanHill} for classical results and, e.g., \cite{DAndreaDullerud, Langbort} for some more recent results.

Loosely speaking, in robust control, supply functions are used to model the way uncertainty in the system processes power. In large scale systems, supply functions are used to capture the power exchanges between the subsystems. In both cases, the notion of interconnection neutral supply functions, introduced in \cite{Willems_1}, often plays a central role. 

%The interconnection neutral supply rates are defined in relation to 

%In case of two interconnected systems $G_1$ and $G_2$, supply rate $s(\cdot,\cdot)$  interconnection neutral supply rate $s_1(\cdot,\cdot)$ and $s_2(\cdot,\cdot)$ are functions of interconnecting signals between the systems.   where    
%
%To illustrate this notion, consider two interconnected systems as presented in Figure~\ref{}. Suppose $G_1$ is disspative with respect to $s_1()$ with storage function $V_1(x_1)$ and $G_2$ is dissipative to $s_2()$ with storage $V_2$. The supply functions $s_1$ and $s_2$ are said to be interconnection neutral supply functions if $s_1+s_2=0$.   

%The main results of this paper give insights into interplay between structured storage/Lyapunov functions of a class of interconnected systems and dissipativitly properties of the individual systems.
One of the results from \cite{Willems_1} (Theorem 5 from \cite{Willems_1}) states that dissipative systems which are interconnected via a neutral interconnection constraint define a new dissipative system where the sum of storage functions of the individual subsystems is a storage function for the overall interconnected system. 
In this paper we are concerned with the converse statement: \emph{if a set of interconnected systems is disspative (stable) with a storage function (Lyapunov function) characterized by an additive structure\footnote{In this context we say that a function has an \emph{additive structure} if it is represented as a sum of local functions, where each local function is defined on a level of an individual system in a network.}, does then also necessarily exist a set of interconnection neutral supply functions with respect to which the individual systems in the network are dissipative?}

%\textcolor{blue}{The main contributions of the paper is positive answer to the above question for a case of acyclic networks\footnote{Precise definition of what we mean by acyclic network is presented in Section~\ref{Sec4}. This definition allows for two adjacent systems ($G_i$, $G_j$) in the network to be interconnected in feedback loops (there are outputs from $G_i$ acting as inputs to $G_j$ and vice-versa), but places restrictions on having more general cycles (loops) in the interconnections} of interconnected linear time invariant systems with quadratic supply functions and quadratic storage functions (Lyapunov) only.}

In this paper we consider linear time invariant (LTI) systems and quadratic supply and storage/Lyapunov functions. The main contributions are summarized as follows:
\begin{itemize}
\item In the case of two open interconnected LTI systems, existence of an additive storage function for an external supply rate with an additive structure implies existence of interconnection neutral supply functions. In the case of an autonomous interconnected LTI system, existence of a quadratic Lyapunov function with an additive structure implies existence of interconnection neutral supply functions.
\item Generalization of the above results to acyclic networks\footnote{Precise definition of what we mean by acyclic network is presented in Section~\ref{Sec4}. This definition allows for two adjacent systems ($G_i$, $G_j$) in the network to be interconnected in feedback loops (there are outputs from $G_i$ acting as inputs to $G_j$ and vice-versa), but places restrictions on having more general cycles (loops) in the interconnections.} of interconnected LTI systems.  
\item  Specific characterizations of neutral supply functions are presented which imply robustness of network stability/dissiptivity to a removal of interconnection links. Based on these results, we present sufficient conditions under which acyclic networks are robustly stable with respect to removal/disconnection of an interconnection link; as well as sufficient conditions under which networks with cycles are robustly stable with respect to removal/disconnection of a system.
\end{itemize}

To the best of our knowledge, and surprisingly, such (and similar) converse statements to \cite[Theorem~5]{Willems_1} have not been presented yet, with exception of our conference paper \cite{JokicNakic_MTNS} which reported preliminary results on this topic. In \cite{JokicNakic_MTNS} we have presented a proof of existence of neutral supply rates for two interconnected systems with no direct feed-through matrices (no algebraic loop) and indicated that the generalization to acyclic networks is possible, but with no detailed proof presented. In contrast, this paper contains results, with detailed proofs, for systems with algebraic loops, generalization to open systems and acyclic networks.  

The remainder of the paper is organized as follows. In Section~\ref{Sec2} we define the notions of dissipativity and interconnection neutral supply functions and present several results from \cite{Willems_1} to set the stage for the main results of the paper. The main results of the paper are presented in Sections~\ref{Sec3}, \ref{Sec4} and \ref{Sec5}. In Section~\ref{Sec3} we consider interconnection of two systems only. These results are further extended in Section~\ref{Sec4} to accommodate arbitrary acyclic networks. Characteristics of interconnection neutral supply rates and certain robustness properties of a network are mutually related in Section~\ref{Sec5}. A numerical example is presented in Section~\ref{Example}. To make the presentation more clear, most of the proofs are collected in a single section, Section~\ref{Sec6}. Conclusions are summarized in Section~\ref{Sec7}. There are two appendices included which contain either well known (Appendix~\ref{Appendix_A}) or novel isolated results (Appendix~\ref{Appendix_B}) which are used in the proofs in Section~\ref{Sec6}.

%%%%%%%%%%%%%%%%%%%%%%%%%%%%%%%%%%%%%%%%%%%%%%%%%%%%%%%%%%%%%
%%%%%%%%%%%%%%%%%%%%%%%%%%%%%%%%%%%%%%%%%%%%%%%%%%%%%%%%%%%%
\section{PRELIMINARIES}
\label{Sec2}

%In this section we define the notation and present some notions and results which set the stage for presentation of the main results of the paper.

\subsection{Notation and terminology}
%\subsection{NOTATION}
Let $\Rset$ denote the field of real numbers and let $\Rset^{m\times n}$ denote $m$ by $n$ matrices with elements in $\Rset$. $I_n$ is the identity matrix with dimension $n$. Index $n$ will be omitted when the dimension is clear from the context.
The transpose of a matrix $A$ is denoted by $A^{\top}$. 
%We use $\Sset^n$ to denote the set of all symmetric matrices of dimension $n \times n$.
$\kernel{A}$ and $\image{A}$ are used to denote the kernel and the image space of $A$, respectively.
The operator $\col(\cdot,\ldots,\cdot)$ stacks its operands into a column vector.
For a set of (not necessarily square) matrices $\{M_1, M_2, \ldots, M_n\}$ we use $\diag(M_1,\ldots,M_n)$ to denote the matrix \begin{footnotesize}$\begin{pmatrix} M_1 & 0 & \cdots & 0 \\ 0 & M_2 & \cdots & 0 \\ \vdots & \vdots & \ddots & \vdots \\ 0 & 0 & \cdots & M_n \end{pmatrix}$. \end{footnotesize}
The matrix inequalities $A\succ B$ ($A \prec B$) and $A\succeq B$ ($A \preceq B$) mean $A$ and $B$ are symmetric and $A-B$ is positive definite (negative definite) and positive semi-definite (negative semi-definite), respectively.
For a transfer matrix $G$ with realization $G(s)=C(sI-A)^{-1}B+D$ we write $G=\left[\begin{array}{c | c} A & B \\ \hline C & D \end{array} \right]$. 
Blocks in matrices that can be inferred by symmetry are sometimes denoted by $\star$ to save space.
For a finite set $\Omega$ we use $|\Omega|$ to denote its cardinality.

Throughout the paper, when we refer to stability of a system, we will mean \emph{asymptotic stability}. The term \emph{stable} should be interpreted in that way.     

%%%%%%%%%%%%%%%%%%%%%%%%%%%%%%%%%%%%%%%%%%%%%%%%%%%%%
\subsection{Dissipativity of LTI systems with quadratic supply functions}
\label{Sec2_B}
%\subsection{DISSIPATIVITY OF LTI SYSTEMS}
Here we recall characterization of dissipativity for linear time invariant (LTI) systems in terms of linear matrix inequalities (LMI). For more details we refer to, e.g., \cite{SchererWeiland}. 

The dynamical system  $\dot{x}=f(x,d)$, $z=g(x,d)$ is said to be strictly dissipative if there exists $\epsilon > 0$ so that the dissipation inequality \eqref{DssipativityInequality1} holds when $s(d(t),z(t))$ is replaced with  $s(d(t),z(t))-\epsilon \|d(t)\|^2$. 
An LTI system $G$, given by
\begin{equation}
\label{LTI_G}
G: \quad \begin{pmatrix}\dot{x} \\ z \end{pmatrix} 
=
\begin{pmatrix} A & B \\ C & D \end{pmatrix}
\begin{pmatrix} x \\ d \end{pmatrix},
\end{equation}
is dissipative with respect to the quadratic supply function
%\begin{equation}
%\label{QuadraticSupply}
%s(d,z)=\left(\begin{smallmatrix}
%  d\\
%  z\\
%\end{smallmatrix}\right)^{\top}
%\left(\begin{smallmatrix}
%  Q & S\\
%  S^{\top} & R\\
%\end{smallmatrix}\right)
%\left(\begin{smallmatrix}
%  d\\
%  z\\
%\end{smallmatrix}\right), 
%\end{equation} 
\begin{equation}
\label{QuadraticSupply}
s(d,z)=\left(\begin{matrix}
  d\\
  z\\
\end{matrix}\right)^{\top}
\left(\begin{matrix}
  Q & S\\
  S^{\top} & R\\
\end{matrix}\right)
\left(\begin{matrix}
  d\\
  z\\
\end{matrix}\right), 
\end{equation}
where $Q$ and $R$ are symmetric matrices of appropriate dimensions and $S$ is a real matrix, if there exists a quadratic storage function 
\begin{equation}
\label{QuadraticStorage}
V(x)=x^\top P x,
\end{equation}
such that the time derivative of $V(x(t))$ along the system's trajectories satisfies the \emph{differential dissipation inequality}
\begin{equation}
\label{DifferentialDissipativity}
\dot{V}(x(t)) \leq s(d(t),z(t))
\end{equation}
at any time $t$ and for all $(x,d,z)$ related via \eqref{LTI_G}. The system \eqref{LTI_G} is strictly dissipative with quadratic supply and quadratic storage if for all $\col(x(t),d(t))\neq 0$ the inequality \eqref{DifferentialDissipativity} holds with $\leq$ replaced by $<$.   
The differential strict dissipativity condition is equivalent to the existence of a symmetric $P$ such that the following linear matrix inequality (LMI) is feasible

\begin{equation}
\label{DisLMI}
\left(\begin{array}{cc} I & 0 \\ A & B \\ \hdashline 0 & I \\ C & D \end{array}\right)^\top
\left(\begin{array}{cc:cc} 0 & P & 0 & 0 \\ 
P & 0 & 0 & 0 \\ \hdashline 
0 & 0 & -Q & -S \\ 
0 & 0 & -S^\top & -R \end{array}\right)
\left(\begin{array}{cc} I & 0 \\ A & B \\ \hdashline 0 & I \\ C & D \end{array}\right) \prec 0. 
\end{equation}

%For strict dissipativity, $\preceq$ is above replaced with $\prec$.
Note that the we do not assume that the pair $(A,B)$ is controllable. Since in this paper we will be exclusively dealing with \emph{strict} dissipativity, no corresponding controllability assumptions are needed for the results presented in the paper, see e.g. Chapter~2 in \cite{SchererWeiland} for details. 

%\begin{remark} 
%There are no additional conditions assumed about the system \eqref{LTI_G} to state equivalence of \emph{i)} feasibility of \eqref{DisLMI} and \emph{ii)} \emph{strict} dissipativity of \eqref{LTI_G} with respect to a supply function \eqref{QuadraticSupply} with storage function \eqref{QuadraticStorage}. On the contrary, to state that \emph{i)} feasibility of \eqref{DisLMI} when $\prec$ is replaced by $\preceq$ is equivalent to \emph{ii)} dissipativity of \eqref{LTI_G} w.r.t. \eqref{QuadraticSupply} with \eqref{QuadraticStorage}, one requires an assumption that the pair $(A,B)$ is controllable; see, e.g., \cite{SchererWeiland} for details. Since in this paper we will be exclusively dealing with strict dissipativity, no suitable controllability assumptions are needed for the results presented in the paper.
%\end{remark}

Several well-known results regarding robust dissipativity and robust stability are further presented in Appendix~\ref{Appendix_A}, as they are not required for presentation of the main results of the paper (Sections~\ref{Sec3}, \ref{Sec4} and \ref{Sec5}), but are however used in the proofs of these results (Section~\ref{Sec6}).  

%%%%%%%%%%%%%%%%%%%%%%%%%%%%%%%%%%%%%%%%%%%%%%%%%%

%%%%%%%%%%%%%%%%%%%%%%%%%%%%%%%%%%%%%%%%%%%%%%%%%%%%
%%%%%%%%%%%%%%%%%%%%%%%%%%%%%%%%%%%%%%%%%%%%%%

\subsection{Interconnection neutral supply rates}
\label{Sec2_C}
%\subsection{INTERCONNECTION NEUTRAL SUPPLY RATES}
Consider two LTI systems $G_i$, $i=1,2$, with inputs $(v_i,d_i)$, outputs $(w_i,z_i)$ and state vectors $x_i$, as presented in Figure~\ref{NeutralSupply}. Suppose the systems are interconnected with the interconnection constraint $v_1=w_2$ and $v_2=w_1$.
%Indeed, we assume that the dimensions of the considered signals are compatible so that such interconnection is possible.

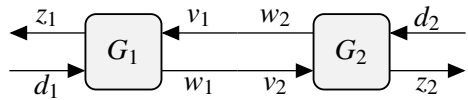
\begin{figure}[h!]
\centering
 \begin{tikzpicture}[scale=1]
 
\draw [line width=0.7pt, rounded corners] (1.5,1) rectangle (2.5,2);
\draw [line width=0.7pt, rounded corners] (4.5,1) rectangle (5.5,2);

\draw [line width=0.5pt, ->, >= triangle 45, rounded corners] (1.5,1.75) -- (0.5,1.75);
\draw [line width=0.5pt, ->, >= triangle 45, rounded corners] (0.5,1.25) -- (1.5,1.25);

\draw [line width=0.5pt, ->, >= triangle 45, rounded corners] (3.5,1.75) -- (2.5,1.75);
\draw [line width=0.5pt,  rounded corners] (2.5,1.25) -- (3.5,1.25);

\draw [line width=0.5pt,  rounded corners] (4.5,1.75) -- (3.5,1.75);
\draw [line width=0.5pt, ->, >= triangle 45, rounded corners] (3.5,1.25) -- (4.5,1.25);

\draw [line width=0.5pt, ->, >= triangle 45, rounded corners] (6.5,1.75) -- (5.5,1.75);
\draw [line width=0.5pt, ->, >= triangle 45, rounded corners] (5.5,1.25) -- (6.5,1.25);

%\draw [red, line width=2.5pt, dashed] (3.5,0.7) -- (3.5,2.3);

\node at (1,1.95) {$z_1$};
\node at (1,1.05) {$d_1$};
\node at (2,1.5) {$G_1$};
\node at (3,1.95) {$v_1$};
\node at (3,1.05) {$w_1$};

\node at (4,1.95) {$w_2$};
\node at (4,1.05) {$v_2$};
\node at (5,1.5) {$G_2$};
\node at (6,1.95) {$d_{2}$};
\node at (6,1.05) {$z_{2}$};

%%%%%%%%%%%%%%% Direct feed-throughs

%%%%%%%%%%%%%%%%

\begin{pgfonlayer}{background}
\filldraw [line width=0.1mm,rounded corners,black!5] (1.5,1) rectangle (2.5,2);
\end{pgfonlayer}

\begin{pgfonlayer}{background}
\filldraw [line width=0.1mm,rounded corners,black!5] (4.5,1) rectangle (5.5,2);
\end{pgfonlayer}

 \end{tikzpicture}
\caption{Open systems with external and interconnection signals.}
\label{NeutralSupply}
\end{figure}
 
Note that throughout the paper we will use symbols $v$ and $w$ (possibly with indexes, e.g., $v_i$, $w_i$) to refer  respectively to input and outputs of a system, which are used to form interconnections with other systems. These signals therefore become internal signals for the connected system, when the interconnections are made.     
On the other side, we will use symbols $d$ and $z$ (possibly with indices, e.g., $d_i$, $z_i$) to refer respectively to exogenous inputs and outputs (exogenous with respect to the set of interconnected systems, that is, with respect to a network).

For $i=1,2$, let the system $G_i$ be dissipative with respect to supply function $s_i(v_i,d_i,w_i,z_i)$ with a storage function $V_i(x_i)$ and suppose that the supply functions have the following additive structure
\begin{equation}
\label{Q1}
s_i(v_i,d_i,w_i,z_i)=s_{i,ext}(d_i,z_i)+s_{i,int}(v_i,w_i), \quad \quad i=1,2.
\end{equation}
The interconnection is said to be \emph{neutral} with respect to the supply functions $s_{1,int}$, $s_{2,int}$ if
\begin{equation}
\label{Q2}
s_{1,int}(v_1,w_1)+s_{2,int}(v_2,w_2)=0,
\end{equation}
for all $v_1$, $w_1$, $v_2$, $w_2$ such that $v_1=w_2$, $v_2=w_1$. We will use the term \emph{interconnection neutral supply functions} to refer to supply functions $s_{1,int}(v_1,w_1)$ and $s_{2,int}(v_2,w_2)$ which satisfy the above property. 

We will use the symbol $G$ to refer to the system obtained by interconnecting $G_1$ and $G_2$.
The following proposition originates from \cite{Willems_1}.  
\begin{proposition}
\label{Prop1}
Let $G_1$ and $G_2$ be strictly dissipative with respect to $s_1(v_1,d_1,w_1,z_1)$ and $s_2(v_2,d_2,w_2,z_2)$, which both have an additive structure as in \eqref{Q1}. Let some corresponding storage functions be $V_1(x_1)$ and $V_2(x_2)$. Suppose the interconnected system $G$ is well-posed and $s_{1,int}(v_1,w_1)$ and $s_{2,int}(v_2,w_2)$ are interconnection neutral supply functions. Then the system $G$ is strictly dissipative with respect to the supply $s_{ext}(d_1,d_2,z_1,z_2):=s_{1,ext}(d_1,z_1)+s_{2,ext}(d_2,z_2)$.      
\end{proposition} 
\begin{proof}Adding the two dissipation inequalities $\dot{V}_1(x_1) < s_1(v_1,d_1,w_1,z_1)$ and $\dot{V}_2(x_2) < s_2(v_2,d_2,w_2,z_2)$, the desired result follows directly from the neutrality condition \eqref{Q2}, as it turns that $V(x_1,x_2)=V_1(x_1)+V_2(x_2)$ acts as a storage function for $G$ with respect to $s_{ext}(d_1,d_2,z_1,z_2)$.
\end{proof}
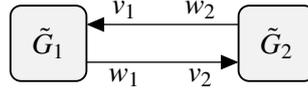
\begin{figure}[h!]
\centering
 \begin{tikzpicture}[scale=1]
 \draw [line width=0.7pt, rounded corners] (1.5,1) rectangle (2.5,2);
\draw [line width=0.7pt, rounded corners] (4.5,1) rectangle (5.5,2);
\draw [line width=0.5pt, ->, >= triangle 45, rounded corners] (3.5,1.75) -- (2.5,1.75);
\draw [line width=0.5pt, rounded corners] (2.5,1.25) -- (3.5,1.25);
\draw [line width=0.5pt, rounded corners] (4.5,1.75) -- (3.5,1.75);
\draw [line width=0.5pt, ->, >= triangle 45, rounded corners] (3.5,1.25) -- (4.5,1.25);
\node at (2,1.5) {$\tilde{G}_1$};
\node at (3,1.95) {$v_1$};
\node at (3,1.05) {$w_1$};
\node at (4,1.95) {$w_2$};
\node at (4,1.05) {$v_2$};
\node at (5,1.5) {$\tilde{G}_2$};
%%%%%%%%%%%%%%%%%%%%%%%%%%%%%%%%%
\begin{pgfonlayer}{background}
\filldraw [line width=0.1mm,rounded corners,black!5] (1.5,1) rectangle (2.5,2);
\end{pgfonlayer}
\begin{pgfonlayer}{background}
\filldraw [line width=0.1mm,rounded corners,black!5] (4.5,1) rectangle (5.5,2);
\end{pgfonlayer}
\end{tikzpicture}
\caption{Interconnected system as an autonomous system.}
\label{StructuredClosed}
\end{figure}
%%%%%%%%%%%%%%%%%%%%%%%%%%%%%%%%%%%%%%%%%%%%%%%%%%%%%%%%%%%%%%%%%%%%%%%%%%%
%%%%%%%%%%%%%%%%%%%%%%%%%%%%%%%%%%%%%%%%%%%%%%%%%%%%%%%%%%%%%%%%%%%%%%%%%%%

Let us now consider an autonomous system $\tilde{G}$ which consists of two interconnected systems $\tilde{G}_1$ and $\tilde{G}_2$, as presented in Figure~\ref{StructuredClosed}. Let $x_1$ and $x_2$ denote state vectors of $\tilde{G}_1$ and $\tilde{G}_2$, respectively.  

\begin{proposition}
\label{Prop2}
Let $\tilde{G}_1$ and $\tilde{G}_2$ be \emph{strictly} dissipative with respect to $s_{1,int}(v_1,w_1)$ and $s_{2,int}(v_2,w_2)$ with some corresponding storage functions $V_1(x_1)$ and $V_2(x_2)$, respectively. Suppose $V_1(\cdot)$ and $V_2(\cdot)$ are positive definite functions, the interconnected system $\tilde{G}$ is well-posed and that $s_{1,int}(v_1,w_1)$ and $s_{2,int}(v_2,w_2)$ are interconnection neutral supply rates. Then the system $\tilde{G}$ is stable. 
\end{proposition} 
\begin{proof} Summing the strict dissipation inequalities $\dot{V}_1(x_1)<s_{1,int}(v_1,w_1)$ and $\dot{V}_2(x_2)<s_{2,int}(v_2,w_2)$, which hold for $\col(x_1,v_1)\neq 0$ and $\col(x_2,v_2)\neq 0$, with \eqref{Q2}, we obtain $\dot{V}_1(x_1)+\dot{V}_2(x_2)<0$ for $x_1\neq 0$ and/or $x_2 \neq 0$. Therefore the positive definite function $V(x_1,x_2):=V_1(x_1)+V_2(x_2)$ is a Lyapunov function for the interconnected system $G$.
\end{proof}
\medskip

Extension of Propositions~\ref{Prop1} and \ref{Prop2} to a larger number of interconnected systems is straightforward, see \cite{Willems_1}.

%%%%%%%%%%%%%%%%%%%%%%%%%%%%%%%%%%%%%%%%%%%%%%%%%%%%%%%%%%%%%%%%%%
%%%%%%%%%%%%%%%%%%%%%%%%%%%%%%%%%%%%%%%%%%%%%%%%%%%%%%%%%%%%%%%%%%
\section{TWO INTERCONNECTED SYSTEMS}
\label{Sec3}
In this section we present the main results of the paper for the case of two interconnected systems. We prove suitably defined converse statements to those of Proposition~\ref{Prop1} and Proposition~\ref{Prop2}. While Propositions~\ref{Prop1} and \ref{Prop2} hold for general dissipative systems with general supply and storage functions, in this section, and in the remainder of the paper, we consider LTI systems with quadratic supply functions, i.e., supply functions in the form of \eqref{QuadraticSupply} and quadratic storage functions, i.e., storage functions in the form of \eqref{QuadraticStorage}.     

\subsection{Interconnection with an algebraic loop}
\label{Sec3_A}
%\subsection{INTERCONNECTION WITH AN ALGEBRAIC LOOP}

Consider two interconnected systems $G_1$ and $G_2$, as depicted in Figure~\ref{StructuredOpen1_53}, where
%\begin{equation}
%\label{Systems_53}
%G_i: \,\,\, \begin{cases} \dot{x}_i=A_i x_i+ B_i v_i + E_i d_i \\ 
%w_i=C_i x_i +D_i v_i\\ 
%z_i=F_i x_i + K_i v_i+L_i d_i  \end{cases}, 
%\quad i=1,2,
%\end{equation}
\begin{equation}
\label{Systems_53}
G_i: \,\,\,
\begin{pmatrix} \dot{x}_i \\ w_i \\ z_i \end{pmatrix}=
\begin{pmatrix} A_i & B_i & E_i \\
C_i & D_i & 0 \\
F_i & K_i & L_i \end{pmatrix}
\begin{pmatrix} x_i \\ v_i \\ d_i \end{pmatrix}, 
\quad i=1,2,
\end{equation}
where $A_1 \in \Rset^{n_1 \times n_1}$, $D_1\in \Rset^{n_{w1}\times n_{v1}}$, $L_1 \in \Rset^{n_{z1}\times n_{d1}}$, $A_2 \in \Rset^{n_2 \times n_2}$, $D_2\in \Rset^{n_{w2}\times n_{v2}}$, $L_2 \in \Rset^{n_{z2}\times n_{d2}}$. 
With $n_{v1}=n_{w2}$, $n_{v2}=n_{w1}$, the systems are interconnected so that
\begin{equation}
\label{Interconnection_53}
v_1=w_2, \quad \quad v_2=w_1.
\end{equation}
Note that the existence of nonzero matrices $D_1$ and $D_2$ implies that there is a direct feed-through of the signal $v_1$ to the signal $w_2$ (see Figure~\ref{StructuredOpen1_53}). This is what we mean by the term that there exists an \emph{algebraic loop} in the system.
Furthermore, note that there is a direct feed-through from $d_1$ to $z_1$ and from $v_1$ to $z_1$, but there is no direct feed-through from $d_1$ to $w_1$. Analogous situation is with the system $G_2$. These feed-through channels are illustrated with dashed lines in Figure~\ref{StructuredOpen1_53}. A distinguishing feature of the interconnected system \eqref{Systems_53}, \eqref{Interconnection_53} is that there is no direct feed-through path that goes from an external input of $G_1$ to an external output of $G_2$, that is, there is no direct feed-through path from $d_1$ to $z_2$. Analogously, there is no direct feed-through path from $d_2$ to $z_1$.  

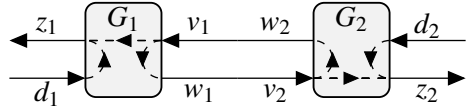
\begin{figure}[h!]
\centering
 \begin{tikzpicture}[scale=1]
 
\draw [line width=0.7pt, rounded corners] (1.5,1) rectangle (2.5,2.25);
\draw [line width=0.7pt, rounded corners] (4.5,1) rectangle (5.5,2.25);

\draw [line width=0.5pt, ->, >= triangle 45, rounded corners] (1.5,1.75) -- (0.5,1.75);
\draw [line width=0.5pt, ->, >= triangle 45, rounded corners] (0.5,1.25) -- (1.5,1.25);

\draw [line width=0.5pt, ->, >= triangle 45, rounded corners] (3.5,1.75) -- (2.5,1.75);
\draw [line width=0.5pt, rounded corners] (2.5,1.25) -- (3.5,1.25);

\draw [line width=0.5pt, rounded corners] (4.5,1.75) -- (3.5,1.75);
\draw [line width=0.5pt, ->, >= triangle 45, rounded corners] (3.5,1.25) -- (4.5,1.25);

\draw [line width=0.5pt, ->, >= triangle 45, rounded corners] (6.5,1.75) -- (5.5,1.75);
\draw [line width=0.5pt, ->, >= triangle 45, rounded corners] (5.5,1.25) -- (6.5,1.25);

%\draw [red, line width=2.5pt, dashed] (3.5,0.7) -- (3.5,2.3);

\node at (1,1.95) {$z_1$};
\node at (1,1.05) {$d_1$};
\node at (2,2.04) {$G_1$};
\node at (3,1.95) {$v_1$};
\node at (3,1.05) {$w_1$};

\node at (4,1.95) {$w_2$};
\node at (4,1.05) {$v_2$};
\node at (5,2.04) {$G_2$};
\node at (6,1.95) {$d_{2}$};
\node at (6,1.05) {$z_{2}$};

%%%%%%%%%%%%%%% Direct feed-throughs
%
%\draw [red, line width=0.5pt, dashed] (1.5,1.25) arc (-90:90:0.25);
%\draw [red, line width=0.5pt, dashed ] (2.5,1.75) arc (90:270:0.25);
%
%\draw [red, line width=0.5pt, dashed] (4.5,1.25) arc (-90:90:0.25);
%\draw [red, line width=0.5pt, dashed ] (5.5,1.75) arc (90:270:0.25);
%
%\draw [red, line width=0.5pt, dashed] (2.5,1.75) rectangle (1.5,1.75);
%\draw [red, line width=0.5pt, dashed] (4.5,1.25) rectangle (5.5,1.25);
%
%\draw [red, line width=.5pt, ->, >= triangle 45] (1.94,1.75) -- (1.9,1.75);
%\draw [red, line width=.5pt, ->, >= triangle 45] (5.04,1.25) -- (5.1,1.25);
%
%\draw [red, line width=.5pt, ->, >= triangle 45] (1.75,1.5) -- (1.75,1.6);
%\draw [red, line width=.5pt, ->, >= triangle 45] (2.25,1.5) -- (2.25,1.4);
%
%\draw [red, line width=.5pt, ->, >= triangle 45] (4.75,1.5) -- (4.75,1.6);
%\draw [red, line width=.5pt, ->, >= triangle 45] (5.25,1.5) -- (5.25,1.4);

%%%%%%%%%%%%%%% Direct feed-throughs

\draw [ line width=0.5pt, dashed] (1.5,1.25) arc (-90:90:0.25);
\draw [line width=0.5pt, dashed ] (2.5,1.75) arc (90:270:0.25);

\draw [line width=0.5pt, dashed] (4.5,1.25) arc (-90:90:0.25);
\draw [line width=0.5pt, dashed ] (5.5,1.75) arc (90:270:0.25);

\draw [line width=0.5pt, dashed] (2.5,1.75) rectangle (1.5,1.75);
\draw [line width=0.5pt, dashed] (4.5,1.25) rectangle (5.5,1.25);

\draw [line width=.3pt, ->, >= triangle 45] (1.94,1.75) -- (1.9,1.75);
\draw [line width=.3pt, ->, >= triangle 45] (5.04,1.25) -- (5.1,1.25);

\draw [line width=.3pt, ->, >= triangle 45] (1.75,1.5) -- (1.75,1.6);
\draw [line width=.3pt, ->, >= triangle 45] (2.25,1.5) -- (2.25,1.4);

\draw [line width=.3pt, ->, >= triangle 45] (4.75,1.5) -- (4.75,1.6);
\draw [line width=.3pt, ->, >= triangle 45] (5.25,1.5) -- (5.25,1.4);

%%%%%%%%%%%%%%%%

\begin{pgfonlayer}{background}
\filldraw [line width=0.1mm,rounded corners,black!5] (1.5,1) rectangle (2.5,2.25);
\end{pgfonlayer}

\begin{pgfonlayer}{background}
\filldraw [line width=0.1mm,rounded corners,black!5] (4.5,1) rectangle (5.5,2.25);
\end{pgfonlayer}

 \end{tikzpicture}
\caption{Two interconnected open systems.}
\label{StructuredOpen1_53}
\end{figure}
%%%%%%%%%%%%%%%%%%%%%%%%%%%%%%%%%%%%%%%%%%%%%%%%%%%%%%%%%%%%%%%%%%%%%%%%%%%
%%%%%%%%%%%%%%%%%%%%%%%%%%%%%%%%%%%%%%%%%%%%%%%%%%%%%%%%%%%%%%%%%%%%%%%%%%%%
%%%%%%%%%%%%%%%%%%%%%%%%%%%%%%%%%%%%%%%%%%%%%%%%%%%%%%%%%%%%%%%%%%%%%%%%%%%%%

We use the symbol $G$ to denote the overall interconnected system presented in Figure~\ref{StructuredOpen1_53} and defined by \eqref{Systems_53} and \eqref{Interconnection_53}. The system $G$ has exogenous inputs $(d_1,d_2)$ and outputs $(z_1,z_2)$.   

\begin{theorem}
\label{OpenSeparationProposition_53}
Let $C_1$ and $C_2$ be full row rank matrices and suppose the interconnected system $G$ is well-posed. Suppose that the system $G$ is strictly dissipative with respect to a quadratic supply function $s_{ext}(d_1,d_2,z_1,z_2)$, which is structured as follows (has an additive structure)
\begin{equation}
\label{AdditiveSupply_53}
s_{\text{ext}}(d_1,d_2,z_1,z_2)=s_{1,ext}(d_1,z_1)+s_{2,ext}(d_2,z_2),
\end{equation}
and suppose there exists a storage function $V(x_1,x_2)$ which is an additive quadratic function as well, i.e., we have
\begin{equation}
\label{AdditiveStorage_53}
V(x_1,x_2)=V_1(x_1)+V_2(x_2),
\end{equation} 
with $V_1$ and $V_2$ quadratic functions.
Then there exist quadratic interconnection neutral supply functions $s_{1,int}(v_1,w_1)$ and $s_{2,int}(v_2,w_2)$ such that 
\begin{equation}
\label{NCondition}
s_{1,int}(v_1,w_1) + s_{2,int}(v_2,w_2) =0 \quad \text{for} \,\, v_1=w_2, v_2=w_1, 
\end{equation}
and 
\begin{itemize}
\item[i)] $G_1$ is strictly dissipative with respect to the supply \[s_1(d_1,v_1,z_1,w_1):=s_{1,ext}(d_1,z_1) + s_{1,int}(v_1,w_1)\] with the storage function $V_1(x_1)$;
\item[ii)] $G_2$ is strictly dissipative with respect to the supply \[s_2(d_2,v_2,z_2,w_2):=s_{2,ext}(d_2,z_2) + s_{2,int}(v_2,w_2)\] with the storage function $V_2(x_2)$.
\end{itemize}
\end{theorem}
We emphasize that in the above theorem the storage functions $V_1$, $V_2$ from the statements (\emph{i}) and (\emph{ii}) are the same $V_1$, $V_2$ from \eqref{AdditiveStorage_53}. 
%Note that the supply functions $s_{1,int}(\cdot,\cdot)$ and $s_{2,int}(\cdot,\cdot)$ are interconnection neutral supply rates. 

Theorem~\ref{OpenSeparationProposition_53} is a converse result to Proposition~\ref{Prop1} and its proof is given in Section~\ref{Sec6}. 
The presented proof indeed exploits the assumption that both $C_1$ and $C_2$ are full row rank matrices and it is still an open question can this assumption be relaxed in this rather general setting. In Section~\ref{Sec3_B} we show that when $D_1=0$ and $D_2=0$ the rank assumption can be omitted. 

%%%%%%%%%%%%%%%%%%%%%%%%%%%%%%%%%%%%%%%%%%%%%%%%%%%%%%%%%%%%%%%%%
Let us now consider an autonomous system which consists of two interconnected systems $\tilde{G}_1$ and $\tilde{G_2}$, as presented in Figure~\ref{StructuredClosed} and given by
\begin{equation}
\label{tildeGi}
\tilde{G}_i: \,\,\,
\begin{pmatrix} \dot{x}_i \\ w_i  \end{pmatrix}=
\begin{pmatrix} A_i & B_i \\
C_i & D_i \end{pmatrix}
\begin{pmatrix} x_i \\ v_i \end{pmatrix}, 
\quad i=1,2.
\end{equation}
Let $\tilde{G}$ denote the interconnected system obtained by imposing interconnection constraints $v_1=w_2$, $v_2=w_1$ on \eqref{tildeGi}. 
The following Corollary follows from Theorem~\ref{OpenSeparationProposition_53}. 
\begin{corollary}
\label{Cor1}
Let $C_1$ and $C_2$ be full row rank matrices. Suppose $\tilde{G}$ is well-posed, stable and admits an additive quadratic Lyapunov function of the form
$V(x_1,x_2)=V_1(x_1)+V_2(x_2)$. Then there exist quadratic interconnection neutral supply functions $s_{1,int}(v_1,w_1)$ and $s_{2,int}(v_2,w_2)$ such that $s_{1,int}(v_1,w_1) + s_{2,int}(v_2,w_2) =0$ for $v_1=w_2$, $v_2=w_1$ and
\begin{enumerate}[(i)] 
\item $G_1$ is strictly dissipative with respect to $s_{1,int}(v_1,w_1)$ with the storage function $V_1(x_1)$, 
\item $G_2$ is strictly dissipative with respect to $s_{2,int}(v_2,w_2)$ with the storage function $V_2(x_2)$.
\end{enumerate}  
\end{corollary}
Proof of Corollary~\ref{Cor1} follows directly from the proof of Theorem~\ref{OpenSeparationProposition_53} presented in Section~\ref{Sec6_A} when in the proof we omit all the terms related to exogenous inputs/outputs $d_1, d_2, z_1, z_2$ (or equivalently, by setting $E_i=0$, $F_i=0$, $K_i=0$, $L_i=0$) and the external supply functions $s_{1,ext}(\cdot,\cdot)$ and $s_{2,ext}(\cdot,\cdot)$.

\subsection{Interconnection without an algebraic loop}
\label{Sec3_B}

In many real-life cases systems are interconnected with no direct feed-through matrices, i.e., matrices $D_1$ and $D_2$ from \eqref{Systems_53} are zero matrices. It turns out that such interconnections have some specific features. Some of them are summarized in the following proposition and in Section~\ref{Sec6}. 

\begin{proposition}
\label{PropositionRank}
Consider LTI systems given by \eqref{Systems_53} and suppose $D_1=0$ and $D_2=0$. Then Theorem~\ref{OpenSeparationProposition_53} and Corollary~\ref{Cor1} remain to hold even when $C_1$ and $C_2$ are not necessarily full row rank matrices. 
\end{proposition}

Note that with $D_1=0$ and $D_2=0$, the systems $G$ and $\tilde{G}$ are necessarily well-posed.

Proof of Proposition~\ref{PropositionRank} is given in Section~\ref{Sec6_B}.
%
%\begin{remark}
%\label{Series}
%\emph{The statement of Proposition~\ref{PropositionRank} remains to hold when one of the interconnection links (i.e., either link with signals $(w_1,v_2)$ or the link $(w_2,v_1)$) does not exist, that is, when the systems are not interconnected in a feedback loop, but in a series interconnection. Note that in that case we need to drop the missing signals in the definition of interconnection neutral supply functions. In that case both matrices $D_1$ and $D_2$ and either the pair ($C_1,B_2$) or the pair ($B_2,C_1$) do not exist.}  
%\end{remark}

%%%%%%%%%%%%%%%%%%%%%%%%%%%%%%%%%%%%%%%%%%%%%%%%%%%%%%%%%%%%%%%%%%
%%%%%%%%%%%%%%%%%%%%%%%%%%%%%%%%%%%%%%%%%%%%%%%%%%%%%%%%%%%%%%%%%
\section{DYNAMICAL NETWORKS}
\label{Sec4}

In this section we first extend the results from Section~\ref{Sec3} from two interconnected systems to more general dynamical networks. After introducing necessary definitions, we study acyclic networks, where, due to space limitations, we present results for \emph{autonomous} acyclic dynamical networks only (i.e., generalization of Corollary~\ref{Cor1}). The corresponding results for \emph{open} acyclic networks which are dissipative with respect to exogenous supply functions with additive structure (i.e., generalization of Theorem~\ref{OpenSeparationProposition_53}) straightforwardly follow along the same lines. 
At the end of the section, we summarize several results regarding neutral supply rates in networks with cycles. 

\subsection{Dynamical networks and additive Lyapunov functions}
\label{Sec4_A}
We define a \emph{dynamical network} as a finite set of dynamical systems interconnected over some graph. More precisely, we use a directed graph $\Gamma:=(\Omega,E)$ in which each vertex $G_i \in \Omega$ is identified with a dynamical system, while a directed edge $(G_i, G_j) \in E$ means that there is an output signal of $G_i$ that is input to $G_j$. 
With an edge $(G_i, G_j) \in E$ we make the following input/output definitions: $w_{ij}$ is an output from the system $G_i$ and $v_{ji}$ is the input to the system $G_j$. The interconnection constraint related to the edge $(G_i, G_j)$ is given by $v_{ji}=w_{ij}$. We use $x_i$ to denote the state vector of system $G_i$.

%Note that we do not require $(G_i, G_j) \in E \implies (G_j, G_i) \in E$, i.e., series interconnections between the neighbouring systems are allowed (similarly as in Remark~\ref{Series}). 

From the directed interconnection graph $\Gamma$, we define an undirected graph $\hat{\Gamma}=(\Omega,\hat{E})$, as follows
\begin{equation}
 \Big((G_i,G_j) \in E\Big)\,\, \text{or}\,\, \Big( (G_j,G_i) \in E \Big) \Longleftrightarrow (G_i,G_j) \in \hat{E}. \nonumber
\end{equation}
Graph $\hat{\Gamma}$ will be used below in this section to define what we mean by the term acyclic network.

Let $\cN_i$ denote the set of indices of systems adjacent to $G_i$ in the graph $\hat{\Gamma}$, that is, $\cN_i:=\{j \,\, : \,\, (G_i,G_j) \in \hat{E}\}$. Consider an arbitrary system $G_i$ from the dynamical network and let $\cN_i=\{N_1^i,N_2^i,\ldots,N_{r(i)}^i\}$ where $r(i)=|\cN_i|$. We assume the system $G_i$ has a state space realization of the following form
 
\begin{equation}
\label{GiNetwork}
G_i: \,\, \begin{pmatrix}\dot{x}_i \\ w_{iN_1^i} \\ w_{iN_2^i} \\ \vdots \\ w_{iN_{r(i)}^i} \end{pmatrix}=
\begin{pmatrix} A_i & B_{iN_1^i} & B_{iN_2^i} & \ldots & B_{iN_{r(i)}^i} \\
C_{iN_1^i} & D_{iN_1^i} & 0 & \ldots & 0 \\
C_{iN_2^i} & 0 & D_{iN_2^i} & \ldots & 0\\
\vdots & \vdots & \ddots & \vdots & 0\\
C_{iN_{r(i)}^i} & 0 & 0 & \ldots & D_{iN_{r(i)}^i} \end{pmatrix}
\begin{pmatrix} x_i\\ v_{iN_1^i} \\ v_{iN_2^i} \\ \vdots \\ v_{iN_{r(i)}^i} \end{pmatrix}. 
\end{equation}

A distinguishing feature of \eqref{GiNetwork} is that there are no direct feed-through paths from $v_{ik}$ to $w_{il}$ when $k \neq l$.

Let $N$ denote the number of systems in a dynamical network, i.e., $N=|\Omega|$, and suppose the overall dynamics of the network is described by $\dot{x}=\cA x$, where $x=\col(x_1,\ldots,x_N)$, $x_i \in \Rset^{n_i}$. We say that the dynamical network admits an additive quadratic Lyapunov function if there exists a block diagonal matrix $P=\diag(P_1,\ldots,P_N)$ with symmetric matrices $P_i\in\Rset^{n_i \times n_i}$, $P_i \succ 0$, such that $\cA^\top P + P \cA \prec 0$. Indeed, the term \emph{additive} is used since the Lyapunov function defined with $P$ is given by
\begin{equation}
\label{AdditiveLyap}
V(x)=\underbrace{x_1^\top P_1 x_1}_{V_1(x_1)} + \ldots \underbrace{x_N^\top P_N x_N}_{V_N(x_N)},
\end{equation}
that is, $V(x)$ is a sum of local functions $V_i(x_i)$, where each $V_i$ is \emph{local} to the system $i$ in a sense that it depends only on the states of that system. 

\subsection{Acyclic dynamical networks}
\label{SecIV_B}
%Next we define what we mean by the term acyclic network.

We say that a dynamical network defined with a directed graph $\Gamma$ is \emph{acyclic} dynamical network if the corresponding undirected graph $\hat{\Gamma}=(\Omega,\hat{E})$ is acyclic. For definition of an acyclic graph we refer to, e.g., \cite{Bollobas}. 

%\textcolor{blue}{In the remainder, with a slight abuse of terminology and to simplify the presentation, we will sometimes refer to a dynamical network by referring to its graph $\hat{\Gamma}$, e.g., we will use the term ``system $\hat{\Gamma}$'' to address the underlying dynamical network rather then the graph $\hat{\Gamma}$.}  

\begin{theorem}
\label{TH1}
Suppose the graph $\hat{\Gamma}$ is acyclic, each $G_i$ is represented in the form of \eqref{GiNetwork} and either one of the following two cases holds:  
\begin{itemize}
\item[1)] the matrices $D_{ij}$ in \eqref{GiNetwork} are in general non-zero matrices while each matrix $C_{ij}$ in \eqref{GiNetwork} has full row rank; 
\item[2)] $D_{ij}=0$ for all $(i,j)$ such that $(G_i,G_j) \in E$.
\end{itemize}
Let the overall interconnected network be well-posed~\footnote{In the case (2) this is not an assumption as then the network is necessarily well-posed.}.  
Then the following two statements are equivalent:
\begin{itemize}
\item[i)] The dynamical network admits an additive quadratic Lyapunov function of the form \eqref{AdditiveLyap}.
\item[ii)] For each $G_i \in \Omega$ and each $j \in \cN_i$ there exists a quadratic function $s_{ij}(v_{ij},w_{ij})$ such that 
\begin{itemize}
\item[a)] $G_i$ is strictly dissipative with respect to the supply function $\sum_{j \in \cN_i} s_{ij}(v_{ij},w_{ij})$ with local function $V_i(x_i)=x_i^\top P_i x_i$, with $P_i \succ 0$, as a storage function; 
\item[b)] $s_{ij}(v_{ij},w_{ij})+s_{ji}(v_{ji},w_{ji})=0$ for each $(i,j)$ such that $(G_i,G_j) \in \hat{E}$.
\end{itemize}
\end{itemize}
\end{theorem}
Note that the condition in part (\emph{b}) of the statement (\emph{ii}) means that the supplies $s_{ij}$ and $s_{ji}$ are interconnection neutral supply functions for the interconnections between the systems $G_i$ and $G_j$. 
We also emphasize that in the above theorem each local function $V_i(x_i)$ from (\emph{i}) is indeed the same $V_i(x_i)$ as in (\emph{ii}). 

\medskip
\begin{proof}[Proof of Theorem~\ref{TH1}]The implication (\emph{ii}) $\implies$ (\emph{i}) is trivial and is a straightforward generalization of Proposition~\ref{Prop2}. It remains to show (\emph{i}) $\implies$ (\emph{ii}). The proof is based on a repeated applications of Theorem~\ref{OpenSeparationProposition_53} and Corollary~\ref{Cor1} for the case (1), or Proposition~\ref{PropositionRank} for the case (2). 

Suppose first that $\hat{\Gamma}$ is a connected graph.
Consider an arbitrary system $G_i$ from the network and let $\mathcal{N}_i=\{N_1^i,N_2^i,\ldots,N_{r(i)}^i\}$ and $r(i)$ be defined as above in Section~\ref{Sec4_A}. Let $\mathcal{E}_i=\{e_1^i,e_2^i,\ldots,e_{r(i)}^i\}:=\{(G_i,G_{N_1^i}),\ldots,(G_i,G_{N_{r(i)}^i})\}$, that is, $\mathcal{E}_i$ is the set of all undirected edges which have $G_i$ as an end vertex. For each $G_i \in \Omega$ and $j \in \{1, \ldots,r(i)\}$ we define the following two subgraphs of $\hat{\Gamma}$:
\begin{enumerate}
\item[i)] $\hat{\Gamma}_+(\hat{\Gamma},e^i_j)$ is the connected component of the graph $\hat\Gamma - e_j^i$ (the graph obtained by removing the edge $e_j^i$ from the graph $\hat\Gamma$) which contains $G_i$;
% induced subgraph\footnote{A graph $\hat{\Gamma}=(\hat{\Omega},\hat{E})$ is an \emph{induced subgraph} of a graph $\hat{\Gamma}=(\Omega,E)$ if $\hat{\Omega} \subset \Omega$, $\hat{E} \subset E$, and $\hat{\Gamma}$ contains all edges of $\Gamma$ that join two vertices in $\hat{E}$. See, e.g, \cite{Bollobas} for examples.} of $\hat{\Gamma}$ containing all vertices which are reachable\footnote{A vertex $G_k$ is defined as reachable from the vertex $G_i$ if there is a path from $G_i$ to $G_k$} from $G_i$ when an edge $e^i_j$ is removed from $\hat{\Gamma}$;
\item[ii)] $\hat{\Gamma}_-(\hat{\Gamma},e^i_j)$ is the connected component of the graph $\hat\Gamma - e_j^i$ which does not contain $G_i$.
% induced subgraph of $\hat{\Gamma}$ containing all vertices which are not reachable from $G_i$ when an edge $e^i_j$ is removed from $\hat{\Gamma}$.
\end{enumerate}
Note that since $\hat{\Gamma}$ is an acyclic graph, for each edge $e^i_j$ we can view the overall dynamical network as an interconnection of two systems: one as the dynamical network corresponding to $\hat{\Gamma}_+(\hat{\Gamma},e^i_j)$ and the other as the dynamical network corresponding to $\hat{\Gamma}_-(\hat{\Gamma},e^i_j)$.   
In the overall network $\hat{\Gamma}$, the two systems are interconnected over the edge $e^i_j$ only. In the remainder we will refer to a dynamical system (dynamical network) simply by referring to its graph, i.e., we will use the term ``system $\hat{\Gamma}_+(\hat{\Gamma},e^i_j)$''. 

Next we infer existence of interconnection neutral supply functions $s_{ij}$, $s_{ji}$ for an arbitrary $G_i \in \Omega$ and for all $j \in \mathcal{N}_i$, which satisfy conditions (\emph{ii}) from the theorem. In other words, we infer existence of neutral supply functions defined on all edges in $\mathcal{E}_i$ for an arbitrary $G_i \in \Omega$. This is an iterative procedure described as follows, with $k$ as the iteration counter:\\
Let $\hat \Gamma_0 := \hat \Gamma$, $k=1$.\\
While $k < r(i) +1$:
\begin{itemize}
	\item if $k=1$ we apply Corollary~\ref{Cor1} (in case (1)) or Proposition~\ref{PropositionRank} (in case (2)), 
	\item if $k > 1$ we apply Theorem~\ref{OpenSeparationProposition_53} (in case (1)), or Proposition~\ref{PropositionRank} (in case (2)),
\end{itemize}
to infer the existence of interconnection neutral supply functions on the edge $e^i_k$ between the systems $\hat{\Gamma}_+(\hat{\Gamma}_{k-1},e^i_k)$ and $\hat{\Gamma}_-(\hat{\Gamma}_{k-1},e^i_k)$. With $j=N_k^i$, denote these interconnection neutral supply functions with $s_{ij}(v_{ij}, w_{ij})$ and $s_{ji}(v_{ji},w_{ji})$ from the statement (ii) of the theorem. \\
Denote $\hat \Gamma_k = \hat \Gamma_+ (\hat \Gamma_{k-1}, e_k^i)$. Increase $k$ by one and continue the procedure.

At each step of the procedure, the system $\hat \Gamma_k$ is strictly dissipative with respect to the external supply function $\sum_{j\in \{N_1^i, \ldots, N_{k}^i\}}s_{ij}(v_{ij},w_{ij})$ with the corresponding storage function being $\sum_{p \in T_{k}} V_p(x_p)$. Here $T_{k}$ denotes the set of indices $p$ of systems $G_p$ which belong to the vertex set of $\hat{\Gamma}_{k}$. 

Finally, if the graph $\hat{\Gamma}$ is not connected, we can apply the presented procedure to each connected component of $\hat{\Gamma}$ separately.  
\end{proof}

The key feature in the procedure described in the above proof is that in each step all the assumptions for applying Corollary~\ref{Cor1} or Theorem~\ref{OpenSeparationProposition_53} are necessarily satisfied at each iteration.

\begin{remark}
\label{RemarkAcyclic}
In the case of a graph $\hat{\Gamma}$ with a cycle, if $e \in E$ is an edge from a cycle in $\hat{\Gamma}$, we cannot use the above stated definitions of $\hat{\Gamma}_+(\hat{\Gamma},e)$ and $\hat{\Gamma}_-(\hat{\Gamma},e)$ to divide the network into two systems and to subsequently apply Corollary~\ref{Cor1} or Theorem~\ref{OpenSeparationProposition_53}.
\end{remark}

\begin{example}
\label{Example1}
In this example we illustrate the iterative procedure (presented in the proof of Theorem~\ref{TH1}) for obtaining interconnection neutral supply rates for all three edges in the graph $\hat{\Gamma}$ in Figure~\ref{ProofNetwork}a. Suppose the considered network belongs to the case (1) from Theorem~\ref{TH1}. First the overall network is seen as the interconnection of $G_1$ and $G_\alpha$, where $G_\alpha$ consists of interconnected $G_2$, $G_3$ and $G_4$, as presented in Figure~\ref{ProofNetwork}a. Corollary~\ref{Cor1} applies and we infer existence of interconnection neutral supply functions $s_{1\alpha}$, $s_{\alpha 1}$. Next, consider interconnection of $G_4$ and $G_\beta$, where $G_\beta$ consists of interconnected $G_2$ and $G_3$, as presented in Figure~\ref{ProofNetwork}b. Neutral supplies follow from Theorem~\ref{OpenSeparationProposition_53}. Finally, consider $G_2$ and $G_3$ and apply again Theorem~\ref{OpenSeparationProposition_53} to obtain $s_{23},s_{32}$, see Figure~\ref{ProofNetwork}c. Note that $s_{1\alpha}=s_{1\beta}=s_{12}$, $s_{\alpha 1}=s_{\beta 1}=s_{21}$, $s_{4 \beta}=s_{42}$, $s_{\beta 4}=s_{24}$.                
\end{example}

%%%%%%%%%%%%%%%%%%%%%%%%%%%%%%%%%%%%%%%%%%%%%%%%%%%%%%%%%%%%%%%%%%%%%%
%%%%%%%%%%%%%%%%%%%%%%%%%%%%%%%%%%%%%%%%%%%%%%%%%%%%%%%%%%%%%%%%%%%%%%
%%%%%%%%%%%%%%%%%%%%%%%%%%%%%%%%%%%%%%%%%%%%%%%%%%%%%%%%%%%%%%%%%%%%%%

\begin{figure}[h!]
\centering
 \begin{tikzpicture}%[scale=0.7]
 
%%%%%%%%%%%%%%%%%%%%%  a  %%%%%%%%%%%%%%%%%%%%%%%%%%%%%%%%%%%%%%% 
\draw [line width=0.7pt, rounded corners] (2.6,0) rectangle (6.8,4.5);
\begin{pgfonlayer}{background}
\filldraw [line width=0.1mm, black!2] (2.6,0) rectangle (6.8,4.5);
\end{pgfonlayer}

\draw [line width=0.7pt ] (1,1) circle (0.5);
\draw [line width=0.7pt ] (3.5,1) circle (0.5);
\draw [line width=0.7pt ] (3.5,3.5) circle (0.5);
\draw [line width=0.7pt ] (6,1) circle (0.5);

\draw [line width=2.5pt] (1.5,1) -- (3,1);
\draw [line width=0.75pt] (4,1) -- (5.5,1);
\draw [line width=0.75pt] (3.5,1.5) -- (3.5,3);

\begin{pgfonlayer}{background}
\filldraw [line width=0.1mm, black!10] (1,1) circle (0.5);
\filldraw [line width=0.1mm, black!10] (3.5,1) circle (0.5);
\filldraw [line width=0.1mm, black!10] (3.5,3.5) circle (0.5);
\filldraw [line width=0.1mm, black!10] (6,1) circle (0.5);
\end{pgfonlayer}

\node at (1,1) {$G_1$};
\node at (3.5,1) {$G_2$};
\node at (6,1) {$G_3$};
\node at (3.5,3.5) {$G_4$};

\node at (6,3.5) {$G_{\alpha}$};

\node at (3.5,-0.7) {a)};

\node at (2,1.4) {$s_{1\alpha}$};
\node at (2,0.6) {$s_{\alpha 1}$};

%%%%%%%%%%%%%%%%%%%%%   b    %%%%%%%%%%%%%%%%%%%%%%%%%%%%%%%%%%%%

\draw [line width=0.7pt, rounded corners] (6+2.6,0) rectangle (6+6.8,2);
\begin{pgfonlayer}{background}
\filldraw [line width=0.1mm, black!2] (6+2.6,0) rectangle (6+6.8,2);;
\end{pgfonlayer} 

\draw [line width=0.7pt ] (6+3.5,1) circle (0.5);
\draw [line width=0.7pt ] (6+3.5,3.5) circle (0.5);
\draw [line width=0.7pt ] (6+6,1) circle (0.5);

\draw [line width=2.5pt] (6+1.5,1) -- (6+3,1);
\draw [line width=0.75pt] (6+4,1) -- (6+5.5,1);
\draw [line width=2.5pt] (6+3.5,1.5) -- (6+3.5,3);

\begin{pgfonlayer}{background}
\filldraw [line width=0.1mm, black!10] (6+3.5,1) circle (0.5);
\filldraw [line width=0.1mm, black!10] (6+3.5,3.5) circle (0.5);
\filldraw [line width=0.1mm, black!10] (6+6,1) circle (0.5);
\end{pgfonlayer}

\node at (6+3.5,1) {$G_2$};
\node at (6+6,1) {$G_3$};
\node at (6+3.5,3.5) {$G_4$};

\node at (6+4.8,1.5) {$G_{\beta}$};

\node at (6+2,1.4) {$s_{1\beta}$};
\node at (6+2,0.6) {$s_{\beta1}$};

\node at (6+3,2.5) {$s_{4\beta}$};
\node at (6+4,2.5) {$s_{\beta4}$};

\node at (6+3.5,-0.7) {b)};

%%%%%%%%%%%%%%%%%%%%%%  c   %%%%%%%%%%%%%%%%%%%%%%%%%%%%%%%%%%%%%

\draw [line width=0.7pt ] (3.5,1-4.5) circle (0.5);
\draw [line width=0.7pt ] (6,1-4.5) circle (0.5);

\draw [line width=2.5pt] (1.5,1-4.5) -- (3,1-4.5);
\draw [line width=2.5pt] (4,1-4.5) -- (5.5,1-4.5);
\draw [line width=2.5pt] (3.5,1.5-4.5) -- (3.5,3-4.5);

\begin{pgfonlayer}{background}
\filldraw [line width=0.1mm, black!10] (3.5,1-4.5) circle (0.5);
\filldraw [line width=0.1mm, black!10] (6,1-4.5) circle (0.5);
\end{pgfonlayer}

\node at (3.5,1-4.5) {$G_2$};
\node at (6,1-4.5) {$G_3$};

\node at (2,1.4-4.5) {$s_{12}$};
\node at (2,0.6-4.5) {$s_{21}$};

\node at (4.75,1.4-4.5) {$s_{23}$};
\node at (4.75,0.6-4.5) {$s_{32}$};

\node at (3,2.5-4.5) {$s_{42}$};
\node at (4,2.5-4.5) {$s_{24}$};

\node at (3.5,-0.3-4.5) {c)};

%%%%%%%%%%%%%%%%%%%%%%%%%%%%%%%%%%%%%%%%%%%%%%%%%%%%%%%%%%%%%%%%%%%

 \end{tikzpicture}
\caption{Illustration of the iterative procedure from the proof of Theorem~\ref{TH1}.}
\label{ProofNetwork}
\end{figure}
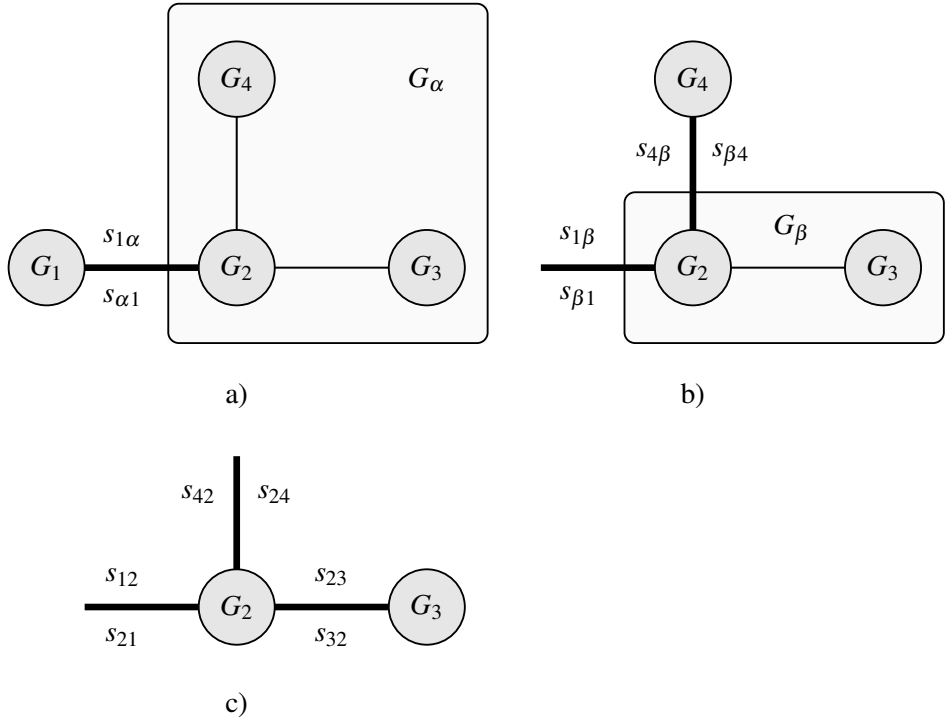

%%%%%%%%%%%%%%%%%%%%%%%%%%%%%%%%%%%%%%%%%%%%%%%%%%%%%%%%%%%%%%%%%%%%%%
%%%%%%%%%%%%%%%%%%%%%%%%%%%%%%%%%%%%%%%%%%%%%%%%%%%%%%%%%%%%%%%%%%%%%%
%%%%%%%%%%%%%%%%%%%%%%%%%%%%%%%%%%%%%%%%%%%%%%%%%%%%%%%%%%%%%%%%%%%%%%

%Note that the results of Section~\ref{Sec3} can be applied as done in the proof of Theorem~\ref{TH1} to infer existence of interconnection neutral supply rates in acyclic networks which do not necessarily belong to the cases (1) or (2) from Theorem~\ref{TH1}. For example, non-zero feed-through matrices which map $v_{ik}$ to $w_{il}$ when $k\neq l$ in some $G_i$ can also be accounted for using Theorem~\ref{OpenSeparationProposition_53}. To illustrate this, recall that Figure~\ref{StructuredOpen1_53} presents the feed-through paths which Theorem~\ref{OpenSeparationProposition_53} can accommodate. Suppose $v_{ik}$ corresponds to $v_2$ in Figure~\ref{StructuredOpen1_53} and $w_{il}$ corresponds to $z_2$ in the same figure. The existing feed-thorough from  $v_{ik}$ to $w_{il}$ might therefore not be an obstacle.  Furthermore, cases which can be interpreted as ``combination of'' cases (1) and (2) from Theorem~\ref{TH1} are also allowed, e.g., a subgraph of $\Gamma$ can belong to case (1), while some other subgraph to the case (2). However, in this paper we will not further pursue the problem of full characterization of general interconnection graphs $\Gamma$ to which Theorem~\ref{TH1} can be generalized to. 

\subsection{Networks with cycles}
We have already in Remark~\ref{RemarkAcyclic} pointed out to the main difficulty of generalizing Theorem~\ref{TH1} to networks where $\hat{\Gamma}$ contains cycles. However, Theorem~\ref{TH1} can be applied to infer certain dissipativiy properties in networks with cycles as well, based on the following simple observation. 

If we group subsets of connected systems from a dynamical network (i.e., from $\hat{\Gamma}$) and consider each of these groups as a single dynamical system, we obtain a graph of a network characterized with new set of vertices (``new systems'') and edges (``new interconnections''). We assume this grouping of systems is such that each system belongs to one and only one group. In this way, from $\hat{\Gamma}$ we obtain new undirected graph $\hat{\Gamma}_{new}$. Indeed, the underlying dynamical network behind  $\Gamma_{new}$ remains the same. This process is illustrated in Figure~\ref{NewNetwork}.  

%%%%%%%%%%%%%%%%%%%%%%%%%%%%%%%%%%%%%%%%%%%%%%%%%%%%%%%%

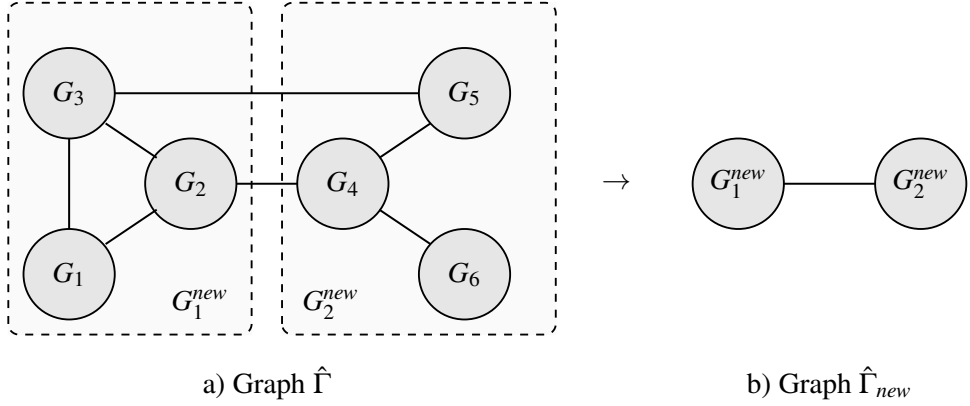
\begin{figure}[h!]
\centering
 \begin{tikzpicture}[scale=0.8]
 
%%%%%%%%%%%%%%%%%%%%%%%%%%
\draw [line width=0.7pt, rounded corners, dashed] (0,0) rectangle (4,5.5);
\begin{pgfonlayer}{background}
\filldraw [line width=0.1mm, black!2] (0,0) rectangle (4,5.5);
\end{pgfonlayer} 

\draw [line width=0.7pt, rounded corners, dashed] (4.5,0) rectangle (9,5.5);
\begin{pgfonlayer}{background}
\filldraw [line width=0.1mm, black!2] (4.5,0) rectangle (9,5.5);
\end{pgfonlayer} 

%%%%%%%%%%%%%%%%%%%%%%%%%%%

\draw [line width=0.7pt ] (1,1) circle (0.75);
\draw [line width=0.7pt ] (1,4) circle (0.75);
\draw [line width=0.7pt ] (3,2.5) circle (0.75);
\draw [line width=0.7pt ] (5.5,2.5) circle (0.75);
\draw [line width=0.7pt ] (7.5,1) circle (0.75);
\draw [line width=0.7pt ] (7.5,4) circle (0.75);

\draw [line width=0.7pt ] (12,2.5) circle (0.75);
\draw [line width=0.7pt ] (15,2.5) circle (0.75);

\begin{pgfonlayer}{background}
\filldraw [line width=0.1mm, black!10] (1,1) circle (0.75);
\filldraw [line width=0.1mm, black!10] (1,4) circle (0.75);
\filldraw [line width=0.1mm, black!10] (3,2.5) circle (0.75);
\filldraw [line width=0.1mm, black!10] (5.5,2.5) circle (0.75);
\filldraw [line width=0.1mm, black!10] (7.5,1) circle (0.75);
\filldraw [line width=0.1mm, black!10] (7.5,4) circle (0.75);

\filldraw [line width=0.1mm, black!10] (12,2.5) circle (0.75);
\filldraw [line width=0.1mm, black!10] (15,2.5) circle (0.75);
\end{pgfonlayer}

\node at (1,1) {$G_1$};
\node at (1,4) {$G_3$};
\node at (3,2.5) {$G_2$};
\node at (5.5,2.5) {$G_4$};
\node at (7.5,1) {$G_6$};
\node at (7.5,4) {$G_5$};

\node at (12,2.5) {$G_1^{new}$};
\node at (15,2.5) {$G_2^{new}$};

\node at (3.15,0.5) {$G_1^{new}$};
\node at (5.3,0.5) {$G_2^{new}$};

\node at (10,2.5) {$\rightarrow$};

\node at (4.25,-0.8) {a) Graph $\hat{\Gamma}$};
\node at (13.5,-0.8) {b) Graph $\hat{\Gamma}_{new}$};
%%%%%%%%%%%%%%%%%%%%%%%%%%%%%%%%%%%%%

\draw [line width=0.75pt] (1.6,1.5) -- (2.44,2.07);
\draw [line width=0.75pt] (1.6,3.5) -- (2.44,3-0.07);

\draw [line width=0.75pt] (6.1,2.07) -- (7.5-0.56,1.5);
\draw [line width=0.75pt] (6.1,3-0.07) -- (7.5-0.56,3.5);

\draw [line width=0.75pt] (1,1.75) -- (1,3.25);
\draw [line width=0.75pt] (3.75,2.5) -- (4.75,2.5);
\draw [line width=0.75pt] (1.75,4) -- (6.75,4);

\draw [line width=0.75pt] (12.75,2.5) -- (14.25,2.5);

%%%%%%%%%%%%%%%%%%%%%%%%%%%%%%%%%%%%%%%%%%%%%%%%%%%%%%%%%%%%%%%%%%%

 \end{tikzpicture}
\caption{Example of grouping of vertices in $\hat{\Gamma}$ to obtain $\hat{\Gamma}_{new}$.}
\label{NewNetwork}
\end{figure}

%%%%%%%%%%%%%%%%%%%%%%%%%%%%%%%%%%%%%%%%%%%%%%%%%%%%%%%%%%

If the system behind the original graph $\hat{\Gamma}$ is stable and admits an additive Lyapunov function, then the system represented with $\hat{\Gamma}_{new}$ admits an additive Laypunov function as well. 
Note that we can always partition networks with cycles in $\hat{\Gamma}$ to obtain $\hat{\Gamma}_{new}$ which is acyclic, and typically this partitioning can be done in several ways. If a network admits an additive Lyapunov function and $\hat{\Gamma}_{new}$ is an acyclic graph, we can use Theorem~\ref{TH1} to infer existence of interconnection neutral supply functions in connection to the edges of $\hat{\Gamma}_{new}$. 
A notable specific case of partitioning and the corresponding dissipativity characterizations are summarised in the following remark.    

\begin{remark}
\label{CorCycles}
Suppose $\hat{\Gamma}$ is an arbitrary graph, possibly containing cycles, and let $G_i$ be an arbitrary system from the corresponding network. Suppose that the considered dynamical network is stable and admits an additive quadratic Lyapunov function. We can view the network as interconnection of two systems $G_1^{new}$ and $G_2^{new}$, where $G_1^{new}:=G_i$ and $G_2^{new}$ is composed of all the remaining systems in the network. The corresponding graph $\hat{\Gamma}_{new}$ is acyclic, contains only two systems and we can apply Theorem~\ref{OpenSeparationProposition_53} and Corollary~\ref{Cor1}. In this way we infer existence of interconnection neutral supply rate defined as quadratic function of all interconneting signals between $G_i$ and the rest of the network.       
\end{remark}

\section{NEUTRAL SUPPLY FUNCTIONS AND ROBUSTNESS}
\label{Sec5}

\subsection{Stability robustness to connection/disconnection of an interconnection link}

Consider an LTI dynamical system $G$ with inputs $(v_A,v_B)$ and outputs $(w_A,w_B)$, as depicted in Figure~\ref{FigureLink}. Suppose that $v_A(t)\in \Rset^p$, $v_B(t)\in \Rset^q$ and $w_A(t)\in \Rset^q$, $w_B(t)\in \Rset^p$, that is, $v_{A}(t)$ and $w_{B}(t)$ are of the same dimension, while $v_{B}(t)$ and $w_{A}(t)$ are of the same dimension. Let $v:=\col(v_{A},v_{B})$, $w:=\col(w_{A},w_{B})$ and let $x$ denote the state vector of $G$.

%%%%%%%%%%%%%%%%%%%%%%%%%%%%%%%%%%%%%%%%%%%%%%%%%%%%%%%%%%%%%%%%%%%%%%%%%%%%%%%%%%%%%%%
%%%%%%%%%%%%%%%%%%%%%%%%%%%%%%%%%%%%%%%%%%%%%%%%%%%%%%%%%%%%%%%%%%%%%%%%%%%%%%%%%%%%%%%
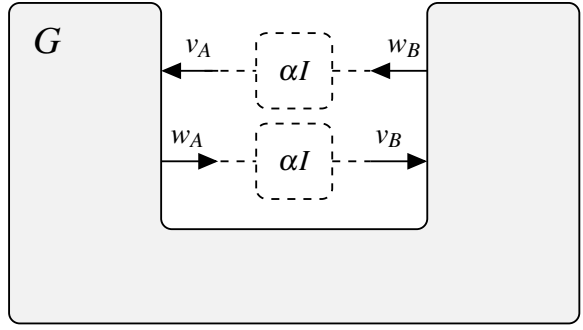
\begin{figure}[h!]
\centering
 \begin{tikzpicture}%[scale=0.8]
%%%%%%%%%%%%%%%%%%%
% The network

%\draw [line width=0.7pt, rounded corners] (0,2)--(0,5)--(2,5)--(2,2)--(5.5,2)--(5.5,5)--(7.5,5)--(7.5,0)--(0,0)--(0,2);

\draw [line width=0.7pt, rounded corners] (0,2)--(0,5)--(2,5)--(2,2)--(5.5,2)--(5.5,5)--(7.5,5)--(7.5,0.75)--(0,0.75)--(0,2);

%%%%%%%%%%%%%%%%%%%
% ISOLATED LINK

\draw [line width=0.7pt, ->, >= triangle 45, rounded corners] (2.75,4.1) -- (2,4.1);
\draw [line width=0.7pt, ->, >= triangle 45, rounded corners] (2,2.9) -- (2.75,2.9);

\draw [line width=0.7pt, ->, >= triangle 45, rounded corners] (5.5,4.1) -- (4.75,4.1);
\draw [line width=0.7pt, ->, >= triangle 45, rounded corners] (4.75,2.9) -- (5.5,2.9);

\draw [line width=0.7pt, dashed, rounded corners] (2.57,4.1) -- (3.25,4.1);
\draw [line width=0.7pt, dashed, rounded corners] (4.25,4.1) -- (4.75,4.1);
\draw [line width=0.7pt, dashed, rounded corners] (3.25,3.6) rectangle (4.25,4.6);

\draw [line width=0.7pt, dashed, rounded corners] (2.57,2.9) -- (3.25,2.9);
\draw [line width=0.7pt, dashed, rounded corners] (4.25,2.9) -- (4.75,2.9);
\draw [line width=0.7pt, dashed, rounded corners] (3.25,2.4) rectangle (4.25,3.4);

%%%%%%%%%%%%%%%%%%%%

%\begin{pgfonlayer}{background}
%\filldraw [line width=0.1mm,rounded corners,black!5] (0,2)--(0,5)--(2,5)--(2,2)--(5.5,2)--(5.5,5)--(7.5,5)--(7.5,0)--(0,0)--(0,2);
%\end{pgfonlayer}

\begin{pgfonlayer}{background}
\filldraw [line width=0.1mm,rounded corners,black!5] (0,2)--(0,5)--(2,5)--(2,2)--(5.5,2)--(5.5,5)--(7.5,5)--(7.5,0.75)--(0,0.75)--(0,2);
\end{pgfonlayer}

%%%%%%%%%%%%%%%%%%%%%%%%%%%%

\node at (3.75,2.9) {$\alpha I$};
\node at (3.75,4.1) {$\alpha I$};
\begin{Large}
\node at (0.5,4.5) {$G$};
\end{Large}
\node at (2.5,4.4) {$v_A$};
\node at (2.35,3.2) {$w_A$};
\node at (5.2,4.4) {$w_B$};
\node at (5,3.2) {$v_B$};

 \end{tikzpicture}
\caption{Dynamical system with an isolated interconnection link.}
\label{FigureLink}
%\caption{An uncertainty $\Delta$}
\end{figure}
%%%%%%%%%%%%%%%%%%%%%%%%%%%%%%%%%%%%%%%%%%%%%%%%%%%%%%%%%%%%%%%%%
%%%%%%%%%%%%%%%%%%%%%%%%%%%%%%%%%%%%%%%%%%%%%%%%%%%%%%%%%%%%%%%%%

\begin{lemma}
\label{LemmaRobust}
Let the system $G$ be strictly dissipative with respect to a quadratic supply function $s(v,w)$ with an additive structure 
\begin{equation}
s(v,w)=s_A(v_A,w_A)+s_B(v_B,w_B)
\end{equation}
and with some corresponding quadratic storage function $V(x)=x^\top P x$. Suppose the following holds
\begin{enumerate}[a)]
\item $s_A(v_A,w_A)+s_B(v_B,w_B)=0$ for all $v_A=w_B$, $v_B=w_A$;
\item $s_A(0,w_A)\leq 0$ for all $w_A \neq 0$, $s_B(0,w_B)\leq 0$ for all $w_B \neq 0$;
\item $P \succ 0.$ 
\end{enumerate}
Then the system $G$ is stable and it remains stable if the following interconnection is made: $v_A=\alpha w_B$, $v_B=\alpha w_A$, for all $\alpha \in [0,1]$.
\end{lemma} 
\begin{proof}The conditions (a) and (b) are satisfied if and only if $R \succeq 0$, $Q \preceq 0$ and the supply $s(v,w)$ is structured as follows 

\begin{equation}
s(v,w)=-\begin{pmatrix} v_{A} \\ v_{B} \\ \hdashline w_{A} \\ w_{B}  \end{pmatrix}^\top 
\underbrace{\left(
\begin{array}{c c : c c}
Q & 0 & S & 0 \\
0 & -R & 0 & -S^\top \\ \hdashline
S^\top & 0 & R & 0 \\
0 & -S & 0 & -Q
\end{array}
\right)}_\Pi
\begin{pmatrix} v_{A} \\ v_{B} \\ \hdashline w_{A} \\ w_{B}  \end{pmatrix}.
\end{equation}
 
Since for $\alpha \in [0,1]$ we have

\begin{equation}
\left(
\begin{array}{c }\star
\end{array}
 \right)^\top
 \Pi
\left(
\begin{array}{c c}
0 & \alpha I \\
\alpha I & 0 \\ \hdashline
I & 0 \\
0 & I 
\end{array}
 \right)=
 \begin{pmatrix}
 (1-\alpha^2)R & \\
 0 & -(1-\alpha^2)Q
 \end{pmatrix}\succeq 0
\end{equation}

the desired result follows directly from Theorem~\ref{RobStability} in Appendix~\ref{Appendix_A} with \[H = \begin{pmatrix} 0 & \alpha I \\ \alpha I & 0 \end{pmatrix} \in \mathbf{H}=\Big\{ \begin{pmatrix} 0 & \alpha I \\ \alpha I & 0 \end{pmatrix} \,\, : \alpha \in [0,1] \Big\}\]
\end{proof}
In Lemma~\ref{LemmaRobust}, the exposed interconnection link defined with an output-input pair $(w, v)$ does not need to be an interconnection link from an acyclic network. Indeed, the system $G$ is an arbitrary LTI system, with no specific structure required. Also note that no specific structure on the Laypunov function (i.e., on the matrix $P$) was imposed. 
In the following subsection, we will use the following convenient interpretation of this result. Suppose that some interconnection link in a system is characterized with an interconnection neutral supply function which satisfies the property (b) from Lemma~\ref{LemmaRobust} and with a positive definite storage function. Then the system is stable irrespective of whether the interconnection is present or not.    

\subsection{Acyclic networks with no algebraic loops}   
The following proposition is proved in Section~\ref{Sec6_C}.
%\begin{proposition}
%\label{Robust}
%\begin{enumerate}[(a)]
%\item Consider LTI systems given by \eqref{Systems_53} and let $D_i=0$, $K_i=0$ and $L_i=0$, $i=1,2$. 
%Then the hypothesis of Theorem~\ref{OpenSeparationProposition_53} holds with full row rank assumption of $C_1$ and $C_2$ omitted and with some interconnection neutral supply functions $s_{1,int}$ and $s_{2,int}$ which satisfy the following condition
%\begin{equation}
%\label{EquNegative}
%s_{1,int}(0,w_1)\leq 0,\,\, s_{2,int}(0,w_2)\leq 0 \,\,\, \text{for all} \, \,\, w_1\neq 0, w_2\neq 0. 
%\end{equation}
%\item Let $D_1=0$, $D_2=0$. Then Corollary~\ref{Cor1}  holds with full row rank assumption of $C_1$ and $C_2$ omitted and with some interconnection neutral supply functions $s_{1,int}$ and $s_{1,int}$ which satisfy \eqref{EquNegative}.
%\end{enumerate}
%\end{proposition}
\begin{proposition}
\label{Robust}
Consider LTI systems given by \eqref{Systems_53}.
\begin{enumerate}[(a)]
\item  Let $D_i=0$, $K_i=0$ and $L_i=0$, $i=1,2$. 
Then the statement of Theorem~\ref{OpenSeparationProposition_53} holds with full row rank assumption on $C_1$ and $C_2$ omitted and with some interconnection neutral supply functions $s_{1,int}$ and $s_{2,int}$ which satisfy the following condition
\begin{equation}
\label{EquNegative}
s_{1,int}(0,w_1)\leq 0,\,\, s_{2,int}(0,w_2)\leq 0 \,\,\, \text{for all} \, \,\, w_1\neq 0, w_2\neq 0. 
\end{equation}
\item Let $D_1=0$, $D_2=0$. Then Corollary~\ref{Cor1}  holds with full row rank assumption on $C_1$ and $C_2$ omitted and with some interconnection neutral supply functions $s_{1,int}$ and $s_{1,int}$ which satisfy \eqref{EquNegative}.
\end{enumerate}
\end{proposition}

The following result follows directly from Proposition~\ref{Robust}, Theorem~\ref{TH1} and Lemma~\ref{LemmaRobust}.

\begin{corollary}
\label{CorROB}
Consider a dynamical network which belongs to the case (2) from Theorem~\ref{TH1} and suppose it admits an additive Lyapunov function of the form  \eqref{AdditiveLyap}. Then the network is robustly stable with respect to removal (disconnection) of an arbitrary edge from $\hat{\Gamma}$.
\end{corollary}

\subsection{Networks with cycles and no algebraic loops}

The following result follows directly from Remark~\ref{CorCycles}, Proposition~\ref{Robust} and Lemma~\ref{LemmaRobust}.
\begin{corollary}
\label{CorROBcycles}
Consider an arbitrary dynamical network defined in Section~\ref{Sec4_A} (hence graph $\hat{\Gamma}$ can contain cycles) in which dynamics of the system $G_i$ is given by \eqref{GiNetwork} with $D_{ij}=0$ for all $(i,j)$ such that $(G_i,G_j)\in E$. Suppose that the network admits an additive Lyapunov function of the form  \eqref{AdditiveLyap}. Then the network is robustly stable with respect to removal (disconnection) of the system $G_i$ from the network\footnote{By this we mean removal of the system $G_i$ together with all edges $(G_i,G_j)\in E$ from the graph $\Gamma$.}.  
\end{corollary}

Simpler, alternative approach to prove Corollary~\ref{CorROB} and Corollary~\ref{CorROBcycles}, which avoids proving existence of interconnection neutral supply rates (Theorem~\ref{TH1}) and their characterizations (Proposition~\ref{Robust}), is presented in \cite{JokicNakic_MTNS}. Still we believe that the characterizations and insights obtained from the proofs based on existence of interconnection neutral supply rates are of independent interest.

%Proposition~\ref{Robust} and Corollary~\ref{} 
%
%ROBUSTNESS
%From the proof of Theorem~\ref{TH1} and using Proposition~\ref{Robust} we can conclude that there exist ... This implies that removal of of any CAN alpha 2 and alpha 2 be introduced istead of alpha. 
%
%\textcolor{blue}{ALLOW for $D_i \neq 0$ only. Then subsytems do not need to be stable, and thre is no robustness to link disconnections.}
%
%\textcolor{red}{
%Suppose there is no cycles. Then we can recursevl apply theorem .... Proof. Any set of connected sybsystems in $\hat{Gamma}$ satisfy Conditions of theorem from previus senction in a sense that here is no direct feedthrough from....\\
%Graph $\hat{\Gamma}$ and $\hat{\hat{\Gamma}}$ are the same (find the appropriate term.)\\
%The propoerty can be assessed via observing the $A$ matrix of the overall system... Note that to state this, it is important that we could not have zero $C$ matrices, as then there would be zero in some part of $A$ while there is a feed-thrugh term passing by.} 

%%%%%%%%%%%%%%%%%%%%%%%%%%%%%%%%%%%%%%%%%%%%%%%%%%%%%%%%%%%%%
\section{EXAMPLE}
\label{Example}

Numerical example presented in this section is based on the example from Section~VI in \cite{Ebihara}. There a DC-grid is presented, which is here illustrated in Figure~\ref{FigExample1}. The grid consists of a DC voltage source $E$, two resistive loads with resistances $R_{L1}$ and $R_{L3}$ and a battery. Each of three transmission lines of the network has resistance $R$ and inductance $L$.  

%%%%%%%%%%%%%%%%%%%%%%%%%%%%%%%%%%%%%%%%%%%%%%%%%%%%%%%%%%%%%%%%%%%%
%%%%%%%%%%%%%%%%%%%%%%%%%%%%%%%%%%%%%%%%%%%%%%%%%%%%%%%%%%%%%%%%%%%%%%

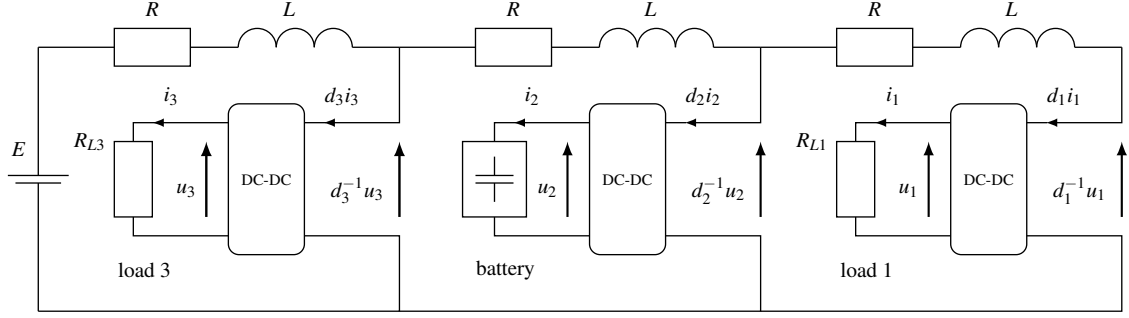
\begin{figure}[h!]
\centering
 \begin{tikzpicture}%[scale=0.58]

%%%%%%%%%%%%%%%%%%%%%%%%%%%%%%%%%%%%%%%%%%%%%%%%%%%%%%%%%%%%%
%A
\draw [line width=0.5pt] (0.5,4.75)--(1.5, 4.75);
%B
\draw [line width=0.5pt] (1.5,4.5) rectangle (2.5,5);
%C
\draw [line width=0.5pt] (3.63,4.75) arc (0:180:0.25);
\draw [line width=0.5pt] (4.13,4.75) arc (0:180:0.25);
\draw [line width=0.5pt] (4.63,4.75) arc (0:180:0.25);
\draw [line width=0.5pt] (2.5,4.75)--(3.13, 4.75);
\draw [line width=0.5pt] (4.63,4.75)--(5.25, 4.75);
%0 
\draw [line width=0.5pt, rounded corners] (3,2) rectangle (4,4);
%1
\draw [line width=0.5pt, ->, >= latex] (5.25,4.75) -- (5.25,3.75) -- (4.25,3.75);
%2
\draw [line width=0.5pt] (4.29,3.75)--(4, 3.75);
%3
\draw [line width=0.5pt, ->, >= latex] (3,3.75) -- (2,3.75);
%4
\draw [line width=0.5pt] (2.1,3.75)--(1.75, 3.75)--(1.75,3.5);
%5
\draw [line width=0.5pt] (1.5,2.5) rectangle (2,3.5);
%6
\draw [line width=0.5pt] (1.75,2.5)--(1.75, 2.25)--(3,2.25);
%7
\draw [line width=0.5pt] (4,2.25)--(5.25, 2.25)--(5.25,1.25)--(0.5,1.25);
%8
\draw [line width=0.9pt, ->, >= latex] (5.25,2.5) -- (5.25,3.5);
%9
\draw [line width=0.9pt, ->, >= latex] (2.71,2.5) -- (2.71,3.5);
%10
\draw [line width=0.5pt] (0.5,1.25)--(0.5, 3-0.05);
\draw [line width=0.5pt] (0.5,3+0.05)--(0.5, 4.75);
\draw [line width=0.5pt] (0.5-0.3,3-0.05)--(0.5+0.3, 3-0.05);
\draw [line width=0.5pt] (0.5-0.4,3+0.05)--(0.5+0.4, 3+0.05);
% NODES
\node at (2,5.25) {\begin{scriptsize}$R$\end{scriptsize}};
\node at (3.8,5.25) {\begin{scriptsize}$L$\end{scriptsize}};
\node at (0.25,3.4) {\begin{scriptsize}$E$\end{scriptsize}};
\node at (2.25,4.1) {\begin{scriptsize}$i_3$\end{scriptsize}};
\node at (4.5,4.1) {\begin{scriptsize}$d_3 i_3$\end{scriptsize}};
\node at (4.7,2.85) {\begin{scriptsize}$d_3^{-1} u_3$\end{scriptsize}};
\node at (2.45,2.85) {\begin{scriptsize}$u_3$\end{scriptsize}};
\node at (1.18,3.5) {\begin{scriptsize}$R_{L3}$\end{scriptsize}};
\node at (3.5,3) {\begin{tiny}DC-DC\end{tiny}};
%\node at (1.8,1.8) {\begin{tiny}load 3\end{tiny}};
\node at (1.9,1.8) {\begin{scriptsize}load 3\end{scriptsize}};
%%%%%%%%%%%%%%%%%%%%%%%%%%%%%%%%%%%%%%%%%%%%%%%%%%%%%%%%%%%%%%%
%%%%%%%%%%%%%%%%%%%%%%%%%%%%%%%%%%%%%%%%%%%%%%%%%%%%%%%%%%%%%
%A
\draw [line width=0.5pt] (0.5+4.75,4.75)--(1.5+4.75, 4.75);
%B
\draw [line width=0.5pt] (1.5+4.75,4.5) rectangle (2.5+4.75,5);
%C
\draw [line width=0.5pt] (3.63+4.75,4.75) arc (0:180:0.25);
\draw [line width=0.5pt] (4.13+4.75,4.75) arc (0:180:0.25);
\draw [line width=0.5pt] (4.63+4.75,4.75) arc (0:180:0.25);
\draw [line width=0.5pt] (2.5+4.75,4.75)--(3.13+4.75, 4.75);
\draw [line width=0.5pt] (4.63+4.75,4.75)--(5.25+4.75, 4.75);
%0 
\draw [line width=0.5pt, rounded corners] (3+4.75,2) rectangle (4+4.75,4);
%1
\draw [line width=0.5pt, ->, >= latex] (5.25+4.75,4.75) -- (5.25+4.75,3.75) -- (4.25+4.75,3.75);
%2
\draw [line width=0.5pt] (4.27+4.75,3.75)--(4+4.75, 3.75);
%3
\draw [line width=0.5pt, ->, >= latex] (3+4.75,3.75) -- (2+4.75,3.75);
%4
\draw [line width=0.5pt] (2.1+4.75,3.75)--(1.75+4.75, 3.75)--(1.75+4.75,3.5);
%5
\draw [line width=0.5pt] (1.33+4.75,2.5) rectangle (2.17+4.75,3.5);
\draw [line width=0.5pt] (6.5-0.25,3-0.05)--(6.5+0.25, 3-0.05);
\draw [line width=0.5pt] (6.5-0.25,3+0.05)--(6.5+0.25, 3+0.05);
\draw [line width=0.5pt] (6.5,3+0.05)--(6.5, 3+0.05+0.25);
\draw [line width=0.5pt] (6.5,3-0.05)--(6.5, 3-0.05-0.25);
%6
\draw [line width=0.5pt] (1.75+4.75,2.5)--(1.75+4.75, 2.25)--(3+4.75,2.25);
%7
\draw [line width=0.5pt] (4+4.75,2.25)--(5.25+4.75, 2.25)--(5.25+4.75,1.25)--(0.5+4.75,1.25);
%8
\draw [line width=0.8pt, ->, >= latex] (5.25+4.75,2.5) -- (5.25+4.75,3.5);
%9
\draw [line width=0.8pt, ->, >= latex] (2.71+4.75,2.5) -- (2.71+4.75,3.5);
%temp
%\draw [line width=0.5pt] (4.75+4.75,4.75)--(1+4.57, 4.75);
% NODES
\node at (2+4.75,5.25) {\begin{scriptsize}$R$\end{scriptsize}};
\node at (3.8+4.75,5.25) {\begin{scriptsize}$L$\end{scriptsize}};
\node at (2.25+4.75,4.1) {\begin{scriptsize}$i_2$\end{scriptsize}};
\node at (4.5+4.75,4.1) {\begin{scriptsize}$d_2 i_2$\end{scriptsize}};
\node at (4.7+4.75,2.85) {\begin{scriptsize}$d_2^{-1} u_2$\end{scriptsize}};
\node at (2.45+4.75,2.85) {\begin{scriptsize}$u_2$\end{scriptsize}};
%\node at (1.15+4.75,2.85) {\begin{scriptsize}$R_{L3}$\end{scriptsize}};
\node at (3.5+4.75,3) {\begin{tiny}DC-DC\end{tiny}};
%\node at (1.8,1.8) {\begin{tiny}load 3\end{tiny}};
\node at (1.9+4.75,1.8) {\begin{scriptsize}battery\end{scriptsize}};
%%%%%%%%%%%%%%%%%%%%%%%%%%%%%%%%%%%%%%%%%%%%%%%%%%%%%%%%%%%%%%%
%%%%%%%%%%%%%%%%%%%%%%%%%%%%%%%%%%%%%%%%%%%%%%%%%%%%%%%%%%%%%%%
%A
\draw [line width=0.5pt] (0.5+9.5,4.75)--(1.5+9.5, 4.75);
%B
\draw [line width=0.5pt] (1.5+9.5,4.5) rectangle (2.5+9.5,5);
%C
\draw [line width=0.5pt] (3.63+9.5,4.75) arc (0:180:0.25);
\draw [line width=0.5pt] (4.13+9.5,4.75) arc (0:180:0.25);
\draw [line width=0.5pt] (4.63+9.5,4.75) arc (0:180:0.25);
\draw [line width=0.5pt] (2.5+9.5,4.75)--(3.13+9.5, 4.75);
\draw [line width=0.5pt] (4.63+9.5,4.75)--(5.25+9.5, 4.75);
%0 
\draw [line width=0.5pt, rounded corners] (3+9.5,2) rectangle (4+9.5,4);
%1
\draw [line width=0.5pt, ->, >= latex] (5.25+9.5,4.75) -- (5.25+9.5,3.75) -- (4.25+9.5,3.75);
%2
\draw [line width=0.5pt] (4.27+9.5,3.75)--(4+9.5, 3.75);
%3
\draw [line width=0.5pt, ->, >= latex] (3+9.5,3.75) -- (2+9.5,3.75);
%4
\draw [line width=0.5pt] (2.1+9.5,3.75)--(1.75+9.5, 3.75)--(1.75+9.5,3.5);
%5
\draw [line width=0.5pt] (1.5+9.5,2.5) rectangle (2+9.5,3.5);
%6
\draw [line width=0.5pt] (1.75+9.5,2.5)--(1.75+9.5, 2.25)--(3+9.5,2.25);
%7
\draw [line width=0.5pt] (4+9.5,2.25)--(5.25+9.5, 2.25)--(5.25+9.5,1.25)--(0.5+9.5,1.25);
%8
\draw [line width=0.8pt, ->, >= latex] (5.25+9.5,2.5) -- (5.25+9.5,3.5);
%9
\draw [line width=0.8pt, ->, >= latex] (2.71+9.5,2.5) -- (2.71+9.5,3.5);
%temp
%\draw [line width=0.5pt] (4.75+9.5,4.75)--(1+9.5, 4.75);
% NODES
\node at (2+9.5,5.25) {\begin{scriptsize}$R$\end{scriptsize}};
\node at (3.8+9.5,5.25) {\begin{scriptsize}$L$\end{scriptsize}};
\node at (2.25+9.5,4.1) {\begin{scriptsize}$i_1$\end{scriptsize}};
\node at (4.5+9.5,4.1) {\begin{scriptsize}$d_1 i_1$\end{scriptsize}};
\node at (4.7+9.5,2.85) {\begin{scriptsize}$d_1^{-1} u_1$\end{scriptsize}};
\node at (2.45+9.5,2.85) {\begin{scriptsize}$u_1$\end{scriptsize}};
\node at (1.18+9.5,3.5) {\begin{scriptsize}$R_{L1}$\end{scriptsize}};
\node at (3.5+9.5,3) {\begin{tiny}DC-DC\end{tiny}};
%\node at (1.8,1.8) {\begin{tiny}load 3\end{tiny}};
\node at (1.9+9.5,1.8) {\begin{scriptsize}load 1\end{scriptsize}};

 \end{tikzpicture}
\caption{An example of a DC-grid, from \cite{Ebihara}.}
\label{FigExample1}
\end{figure}
%%%%%%%%%%%%%%%%%%%%%%%%%%%%%%%%%%%%%%%%%%%%%%%%%%%%%%%%%%%%%%%%%%%%%%
%%%%%%%%%%%%%%%%%%%%%%%%%%%%%%%%%%%%%%%%%%%%%%%%%%
%%%%%%%%%%%%%%%%%%%%
The voltage and the current of load $1$, the battery, and the load $3$ are denoted by $(u_1, i_1)$, $(u_2, i_2)$ and $(u_3, i_3)$, respectively. Furthermore, load $1$, the battery and load $3$ are equipped with ideal DC-DC converters with voltage gains $d_1, d_2$ and $d_3$, respectively. The overall network is seen as an interconnection of $3$ dynamical systems, as presented in Figure~\ref{FigExample2} and as described next.

The first system $G_1$ consists of load $1$, its converter and the transmission line. This is a first order system with a single state $x_1:=i_1$, which is the current over the corresponding transmission inductance. Input to $G_1$ is voltage $u_2$, and output of $G_1$ is voltage $u_1$.

The second system $G_2$ consists of the battery, its converter and a transmission line. The sate vector of this system is $x_2:=\col(s, i_2)$, where $s$ denotes the state of charge of the battery and is assumed to be linear with respect to $u_2$, that is, $s=Ku_2$ where $K$ is a given constant. The dynamics of the battery is modelled as $\dot{s}=i_2$. Input to $G_2$ is $\col(u_1, u_3)$, while the output is $\col(u_2, u_2)$, as presented in Figure~\ref{FigExample2}. 

The third system $G_3$ consists of load $3$, its converter, a transmission line and the voltage source $E$. Here the state vector is $x_3:=i_3$, the input is $u_2$ and the output is $u_3$.

%%%%%%%%%%%%%%%%%%%%%%%%%%%%%%%%%%%%%%%%%%%%%%%%%%%%%%%%%%%%%%%%%%%%%
%%%%%%%%%%%%%%%%%%%%%%%%%%%%%%%%%%%%%%%%%%%%%%%%%%%%%%%%%%%%%%%%%%%%%%

\begin{figure}[h!]
\centering
 \begin{tikzpicture}[scale=1]
 
\draw [line width=0.7pt, rounded corners] (1.5,1) rectangle (2.5,2);
\draw [line width=0.7pt, rounded corners] (4,1) rectangle (5,2);
\draw [line width=0.7pt, rounded corners] (6.5,1) rectangle (7.5,2);

\draw [line width=0.5pt, ->, >= triangle 45, rounded corners] (4,1.75) -- (2.5,1.75);
\draw [line width=0.5pt, ->, >= triangle 45, rounded corners] (2.5,1.25) -- (4,1.25);

\draw [line width=0.5pt, ->, >= triangle 45, rounded corners] (6.5,1.75) -- (5,1.75);
\draw [line width=0.5pt, ->, >= triangle 45, rounded corners] (5,1.25) -- (6.5,1.25);

%\draw [line width=0.5pt, ->, >= triangle 45, rounded corners] (8.5,1.5) -- (7.5,1.5);

\node at (3.25,1.95) {$u_2$};
\node at (3.25,1.05) {$u_1$};
\node at (5.75,1.95) {$u_3$};
\node at (5.75,1.05) {$u_2$};
%\node at (8.2,1.68) {$E$};

\node at (2,1.5) {$G_1$};
\node at (4.5,1.5) {$G_2$};
\node at (7,1.5) {$G_3$};

%%%%%%%%%%%%%%% Direct feed-throughs

%%%%%%%%%%%%%%%%

\begin{pgfonlayer}{background}
\filldraw [line width=0.1mm,rounded corners,black!5] (1.5,1) rectangle (2.5,2);
\end{pgfonlayer}

\begin{pgfonlayer}{background}
\filldraw [line width=0.1mm,rounded corners,black!5] (4,1) rectangle (5,2);
\end{pgfonlayer}

\begin{pgfonlayer}{background}
\filldraw [line width=0.1mm,rounded corners,black!5] (6.5,1) rectangle (7.5,2);
\end{pgfonlayer}

 \end{tikzpicture}
\caption{Interconnection of systems in the DC-grid}
\label{FigExample2}
\end{figure}
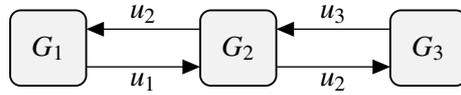
With the above descriptions, the state-space realizations of the three systems are given as follows:

\begin{equation}
G_1=\left[\begin{array}{c | c} -\frac{R+d_1^{-2}R_{L1}}{L} & \frac{d_1^{-1}d_2^{-1}}{L} \\ \hline R_{L1} & 0 \end{array} \right],\, G_3=\left[\begin{array}{c | c} -\frac{R+2d_3^{-2}R_{L3}}{L} & \frac{d_2^{-1}d_3^{-1}}{L} \\ \hline R_{L3} & 0 \end{array} \right], \nonumber
\end{equation}
\begin{equation}
G_2=\left[\begin{array}{c c | c c} 0 & 1 & 0 & 0 \\ 
-\frac{2d_2^{-2}}{KL} & -\frac{R}{L} & \frac{d_1^{-1}d_2^{-1}}{L} & \frac{d_2^{-1}d_3^{-1}}{L}\\ \hline 
\frac{1}{K} & 0 & 0 & 0 \\ 
\frac{1}{K} & 0 & 0 & 0  \end{array} \right]. \nonumber
\end{equation}

Note that in the above models, and in Figure~\ref{FigExample2}, we do not consider $E$ as an exogenous input. Instead we assume $E$ to be a constant signal while the values of the states and the interconnecting signals denote their deviations from  the corresponding equilibrium values. Numerical values of the parameters are as well taken from \cite{Ebihara} and are as follows: $R=1\, \Omega$, $L = 10^{-3}\, \text{H}$, $K=10\, \text{F}$, $R_{L1}=R_{L3}=20\, \Omega$. $(d_1,d_2,d_3)=(0.3953, 0.1634, 0.7785)$.

It can easily be verified that the considered dynamical network is stable and admits an additive quadratic Lyapunov function $V(x_1,x_2,x_3)=V_1(x_1)+V_2(x_2)+V_3(x_3)$, where 

\begin{align}
V_1(x_1)&=3.3282\,\, i_1^2, \,\,
V_2(x_2)=\begin{pmatrix} s \\ i_2 \end{pmatrix}^\top
\begin{pmatrix} 14.3127 & 0.0261 \\ 0.0261 & 0.0069 \end{pmatrix}
\begin{pmatrix} s \\ i_2 \end{pmatrix}, \nonumber \\
V_3(x_3)&=2.3523\,\, i_3^2. \nonumber
\end{align}

Theorem~\ref{TH1} states that there exist interconnection neutral supply functions defined on the interconnection links. We have applied the iterative procedure described in the proof of Theorem~\ref{TH1} to construct a set of such supply functions. We have that $G_1$ is strictly dissipative w.r.t.\ supply function $s_{12}(u_2,u_1)$ with the storage $V_1(x_1)$; the system $G_2$ is strictly dissipative w.r.t. $s_{21}(u_1,u_2)+s_{23}(u_3,u_2)$ with the storage $V_2(x_2)$; and the system $G_3$ is strictly dissipative w.r.t. $s_{32}(u_2,u_3)$ with the storage $V_3(x_3)$, where
%
%\begin{align}
%s_{12}(u_2,u_1)&=\begin{pmatrix} u_2 \\ u_1 \end{pmatrix}^\top
%\begin{pmatrix} -4754.6 & -1543.5 \\ -1543.5 & 1637.6  \end{pmatrix}
%\begin{pmatrix} u_2 \\ u_1 \end{pmatrix},\\
%s_{21}(u_1,u_2)&=\begin{pmatrix} u_1 \\ u_2 \end{pmatrix}^\top
%\begin{pmatrix} -1637.6 & 1543.5 \\ 1543.5 & 4754.6  \end{pmatrix}
%\begin{pmatrix} u_1 \\ u_2 \end{pmatrix}, \\
%s_{23}(u_3,u_2)&= \begin{pmatrix} u_3 \\ u_2 \end{pmatrix}^\top 
%\begin{pmatrix} -608 & 506.8 \\ 506.8 & 2298.6  \end{pmatrix}
%\begin{pmatrix} u_3 \\ u_2 \end{pmatrix}, \\
%s_{32}(u_2,u_3)&=\begin{pmatrix} u_2 \\ u_3 \end{pmatrix}^\top
%\begin{pmatrix} -2298.6 & -506.8 \\ -506.8 & 608  \end{pmatrix}
%\begin{pmatrix} u_2 \\ u_3 \end{pmatrix}.
%\end{align}
%

\begin{align}
s_{12}(u_2,u_1)&=\begin{pmatrix} u_2 \\ u_1 \end{pmatrix}^\top
\begin{pmatrix} 4754.6 & 1543.5 \\ 1543.5 & -1637.6  \end{pmatrix}
\begin{pmatrix} u_2 \\ u_1 \end{pmatrix},\\
s_{21}(u_1,u_2)&=\begin{pmatrix} u_1 \\ u_2 \end{pmatrix}^\top
\begin{pmatrix} 1637.6 & -1543.5 \\ -1543.5 & -4754.6  \end{pmatrix}
\begin{pmatrix} u_1 \\ u_2 \end{pmatrix}, \\
s_{23}(u_3,u_2)&= \begin{pmatrix} u_3 \\ u_2 \end{pmatrix}^\top 
\begin{pmatrix} 608 & -506.8 \\ -506.8 & -2298.6  \end{pmatrix}
\begin{pmatrix} u_3 \\ u_2 \end{pmatrix}, \\
s_{32}(u_2,u_3)&=\begin{pmatrix} u_2 \\ u_3 \end{pmatrix}^\top
\begin{pmatrix} 2298.6 & 506.8 \\ 506.8 & -608  \end{pmatrix}
\begin{pmatrix} u_2 \\ u_3 \end{pmatrix}.
\end{align}

Note that $s_{12}(u_2,u_1)+s_{21}(u_1,u_2)=0$ and $s_{23}(u_3,u_2)+s_{32}(u_2,u_3)=0$. 
We remark that in the calculation of the supply functions according to the constructive procedure from the proof of Theorem~\ref{OpenSeparationProposition_53}, we have chosen $\alpha=0.5$ in \eqref{MultNew}. 

\emph{A posteriori} tests, based on verification of the corresponding strict dissipativity LMIs of the form \eqref{DisLMI}, indeed confirm that all three systems are strictly dissipative w.r.t.\ supply functions and with the storage functions as descried above in the text. Note that $s_{12}(0,u_1)<0$ for all $u_1 \neq 0$; $s_{21}(0,u_2)+s_{23}(0,u_2) < 0$ for all $u_2 \neq 0$ and $s_{32}(0,u_3)<0$ for all $u_3 \neq 0$, which is in accordance with the statements of Proposition~\ref{Robust}. Finally, it is easily verified that the network stability is robust w.r.t.\ removal of the interconnection links, what is in conformity with Corollary~\ref{CorROB}.

%

%
%
%\begin{equation}
%s_{12}(u_2,u_1)=\begin{pmatrix} u_2 \\ u_1 \end{pmatrix}^\top
%\begin{pmatrix} -4754.6 & -1543.5 \\ -1543.5 & 1637.6  \end{pmatrix}
%\begin{pmatrix} u_2 \\ u_1 \end{pmatrix}
%\end{equation}
%%\begin{align}
%%s_{21}(u_1,u_2)+s_{23}(u_3,u_2)=&\begin{pmatrix} u_1 \\ u_2 \end{pmatrix}^\top
%%\begin{pmatrix} -1637.6 & 1543.5 \\ 1543.5 & 4754.6  \end{pmatrix}
%%\begin{pmatrix} u_1 \\ u_2 \end{pmatrix} \nonumber \\
%%+\begin{pmatrix} u_3 \\ u_2 \end{pmatrix}^\top &
%%\begin{pmatrix} -608 & 506.8 \\ 506.8 & 2298.6  \end{pmatrix}
%%\begin{pmatrix} u_3 \\ u_2 \end{pmatrix}
%%\end{align}
%\begin{equation}
%s_{21}(u_1,u_2)=\begin{pmatrix} u_1 \\ u_2 \end{pmatrix}^\top
%\begin{pmatrix} -1637.6 & 1543.5 \\ 1543.5 & 4754.6  \end{pmatrix}
%\begin{pmatrix} u_1 \\ u_2 \end{pmatrix} 
%\end{equation}
%\begin{equation}
%s_{23}(u_3,u_2)= \begin{pmatrix} u_3 \\ u_2 \end{pmatrix}^\top 
%\begin{pmatrix} -608 & 506.8 \\ 506.8 & 2298.6  \end{pmatrix}
%\begin{pmatrix} u_3 \\ u_2 \end{pmatrix}
%\end{equation}
%\begin{equation}
%s_{32}(u_2,u_3)=\begin{pmatrix} u_2 \\ u_3 \end{pmatrix}^\top
%\begin{pmatrix} -2298.6 & -506.8 \\ -506.8 & 608  \end{pmatrix}
%\begin{pmatrix} u_2 \\ u_3 \end{pmatrix}
%\end{equation}
%

%    3.3282         0         0         0
%         0   14.3127    0.0261         0
%         0    0.0261    0.0069         0
%         0         0         0    2.3523

%   -4.7546   -1.5435
%   -1.5435    1.6376

%   -1.6376    1.5435
%    1.5435    4.7546

%   -0.6080    0.5068
%    0.5068    2.2986

%   -2.2986   -0.5068
%   -0.5068    0.6080
%%%%%%%%%%%%%%%%%%%%%%%%%%%%%%%%%%%%%%%%%%%%%%%%%%%%%%%%%%%%%%%
\section{PROOFS}
\label{Sec6}

This section contains proofs of the novel results presented in this paper. Some of the proofs make use of the non-conservative robust stability and robust dissipativity characterizations based on the full block S-procedure \cite{SchererLPV, SchererFromBook}, which are suitably summarized in Appendix~\ref{Appendix_A}.

\subsection{Proof of Theorem~\ref{OpenSeparationProposition_53}}
\label{Sec6_A}
\medskip

\begin{proof}[\textbf{Step 0: \emph{Starting point, main idea and overview}}]
%\textbf{\emph{Step 0}: Starting point, main idea and overview.} \newline
The system $G$ can be presented as system $G_0$ with an interconnection matrix $H=\left(\begin{smallmatrix} 0 & I \\ I & 0 \end{smallmatrix}\right)$ in a feedback loop, as illustrated in Figure~\ref{RobPerf}b in Appendix~\ref{Appendix_A}, where the system $G_0$ is presented in Figure~\ref{G0}. 

%%%%%%%%%%%%%%%%%%%%%%%%%%%%%%%%%%%%%%%%%%%%%%%%%%%%%%%%%%%%%%%%%%%%%%
%%%%%%%%%%%%%%%%%%%%%%%%%%%%%%%%%%%%%%%%%%%%%%%%%%%%%%%%%%%%%%%%%%%%%%
%%%%%%%%%%%%%%%%%%%%%%%%%%%%%%%%%%%%%%%%%%%%%%%%%%%%%%%%%%%%%%%%%%%%%%

\begin{figure}[h!]
\centering
 \begin{tikzpicture}%[scale=0.8]
 
\draw [line width=0.7pt, rounded corners] (1.5,1) rectangle (2.5,2);
\draw [line width=0.7pt, rounded corners] (1.5,2.5) rectangle (2.5,3.5);

\draw [line width=0.7pt, rounded corners] (-0.5,0.5) rectangle (4.5,4.3);

\begin{pgfonlayer}{background}
\filldraw [line width=0.1mm,rounded corners,black!2] (-0.5,0.5) rectangle (4.5,4.3);
\end{pgfonlayer}

\begin{pgfonlayer}{background}
\filldraw [line width=0.1mm,rounded corners,black!10] (1.5,1) rectangle (2.5,2);
\end{pgfonlayer}

\begin{pgfonlayer}{background}
\filldraw [line width=0.1mm,rounded corners,black!10] (1.5,2.5) rectangle (2.5,3.5);
\end{pgfonlayer}

%%%%%%%%%%%%%%%%%%%%%%%%%%%%

%\draw [line width=0.7pt, rounded corners] (4.5,1) rectangle (5.5,2);

\draw [line width=0.5pt, ->, >= triangle 45] (5.5,3.25) -- (2.5,3.25);
\draw [line width=0.5pt, ->, >= triangle 45] (5.5,2.75) -- (4.5-0.5,2.75) -- (3.5-0.5,1.75) -- (2.5,1.75);
\draw [line width=0.5pt, ->, >= triangle 45] (1.5,3.25) -- (-1.5,3.25);
\draw [line width=0.5pt, ->, >= triangle 45] (1.5,2.75) -- (0.5+0.5,2.75) -- (-0.5+0.5,1.75) -- (-1.5,1.75);
\draw [line width=0.5pt, ->, >= triangle 45] (5.5,1.25) -- (2.5,1.25);
\draw [line width=0.5pt, ->, >= triangle 45] (1.5,1.25) -- (-1.5,1.25);

\draw [line width=0.5pt, ->, >= triangle 45] (5.5,1.75) -- (4.5-0.5,1.75) -- (3.5-0.5,2.75) -- (2.5,2.75);
\draw [line width=0.5pt, ->, >= triangle 45] (1.5,1.75) -- (0.5+0.5,1.75) -- (-0.5+0.5,2.75) -- (-1.5,2.75);

%\draw [line width=0.5pt, ->, >= triangle 45, rounded corners] (1.5,1.75) -- (0.5,1.75);
%\draw [line width=0.5pt, ->, >= triangle 45, rounded corners] (0.5,1.25) -- (1.5,1.25);
%
%\draw [line width=0.5pt, ->, >= triangle 45, rounded corners] (3.5,1.75) -- (2.5,1.75);
%\draw [line width=0.5pt, ->, >= triangle 45,   rounded corners] (3.5,1.25) -- (2.5,1.25);

%%%%%%%%%%%%%%%%%%%%%%%%%%%%%%%%%%%%%%%
\node at (2,3) {$G_1$};
\node at (2,1.5) {$G_2$};

\node at (0.1,4) {$G_0$};

\node at (5,3.4) {$v_1$};
\node at (5,2.9) {$v_2$};
\node at (5,3.4-1.4) {$d_1$};
\node at (5,2.9-1.4) {$d_2$};

\node at (-0.9,3.4) {$w_1$};
\node at (-0.9,2.9) {$w_2$};
\node at (-0.9,3.4-1.5) {$z_1$};
\node at (-0.9,2.9-1.5) {$z_2$};

\draw[decoration={brace,raise=0},decorate]
  (5.7,3.5) -- node[right=0.5] {$v$} (5.7,2.5);
  
  \draw[decoration={brace,mirror,raise=0},decorate]
  (-1.7,3.5) -- node[left=0.5] {$w$} (-1.7,2.5);
  
  \draw[decoration={brace,raise=0},decorate]
  (5.7,3.5-1.5) -- node[right=0.5] {$d$} (5.7,2.5-1.5);
  
    \draw[decoration={brace,mirror,raise=0},decorate]
  (-1.7,3.5-1.5) -- node[left=0.5] {$z$} (-1.7,2.5-1.5);

%
%\node at (1,1.95) {$z_1$};
%\node at (1,1.05) {$d_1$};
%
%\node at (3,1.95) {$v_1$};
%\node at (3,1.05) {$w_1$};
%
%\node at (4,1.95) {$w_2$};
%\node at (4,1.05) {$v_2$};
%\node at (5,1.5) {$G_2$};
%\node at (6,1.95) {$d_{2}$};
%\node at (6,1.05) {$z_{2}$};

%%%%%%%%%%%%%%% Direct feed-throughs

%%%%%%%%%%%%%%%%

 \end{tikzpicture}
\caption{The system $G_0$.}
\label{G0}
\end{figure}
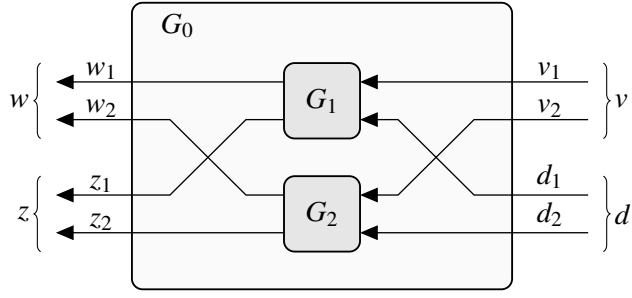

%%%%%%%%%%%%%%%%%%%%%%%%%%%%%%%%%%%%%%%%%%%%%%%%%%%%%%%%%%%%%%%%%%%%%%
%%%%%%%%%%%%%%%%%%%%%%%%%%%%%%%%%%%%%%%%%%%%%%%%%%%%%%%%%%%%%%%%%%%%%%
%%%%%%%%%%%%%%%%%%%%%%%%%%%%%%%%%%%%%%%%%%%%%%%%%%%%%%%%%%%%%%%%%%%%%%

The system $G_0$ from Figure~\ref{G0} is given by \eqref{SysRobPerfA1} with $M=0$, $x=\col(x_1,x_2)$, $w=\col(w_1,w_2)$, $v=\col(v_1,v_2)$, $d=\col(d_1,d_2)$, $z=\col(z_1,z_2)$; with $A=\diag(A_1,A_2)$ and with $B$, $C$, $D$, $E$, $F$, $K$, $L$ defined as block diagonal matrices in the same way as $A$. 

Let the external supply function be given by
\[
s_{ext}(d,z):=-\begin{pmatrix} d \\ z \end{pmatrix}^\top 
\begin{pmatrix} Q^P & S^P \\ (S^P)^\top & R^P \end{pmatrix}\begin{pmatrix} d \\ z \end{pmatrix}.\;
\]
The fact that $s_{ext}$ has an additive structure implies that the matrices in the above definition of $s_{ext}$ are block diagonal, as follows:

\begin{equation}
\label{FullPiP_53}
\Pi^P=\begin{pmatrix} Q^P & S^P \\ (S^P)^\top & R^P \end{pmatrix}:=\left( \begin{array}{c c : c c} 
Q_{1}^P & 0 & S_{1}^P & 0 \\ 
0 & Q_{2}^P & 0  & S_{2}^P \\  \hdashline 
(S_{1}^P)^\top & 0 & R_{1}^P & 0 \\ 
0 & (S_{2}^P)^\top & 0 & R_{2}^P
\end{array} \right),
\end{equation}

since then we have
\begin{equation}
s_{ext}(d,z)=
\sum_{i=1}^2 \underbrace{\begin{pmatrix}d_i \\ z_i \end{pmatrix}^\top
\begin{pmatrix} -Q_{i}^P & -S_{i}^P \\ -(S_{i}^P)^\top & -R_{i}^P \end{pmatrix}
\begin{pmatrix}d_i \\ z_i \end{pmatrix}}_{s_{i,ext}(d_i,z_i)}.
\end{equation}
The dissipativity condition with respect to $s_{ext}(d,z)$ and with a storage function $V(x)=x^\top P x$ is, according to statement (b) of Theorem~\ref{PerformanceSProcedure} from Appendix~\ref{Appendix_A}, equivalent to the existence of a multiplier 
\begin{equation}
\label{MultiplierPi}
\Pi=\begin{pmatrix} Q & S \\ S^\top & R \end{pmatrix}
\end{equation}
such that
\begin{equation}
\label{NN1_53}
\begin{pmatrix} H \\ I \end{pmatrix}^\top 
\underbrace{\begin{pmatrix} Q & S \\ S^\top & R \end{pmatrix}}_{\Pi}
\begin{pmatrix} H \\ I \end{pmatrix} \succeq  0 
\end{equation}
and 

\begin{equation}
\label{NN2_53}
\left(\star\right)^\top
\left(\begin{array}{cc:cc:cc} 0 & P & 0 & 0 & 0 & 0\\ 
P & 0 & 0 & 0 & 0 & 0 \\ \hdashline 
0 & 0 & Q & S & 0 & 0\\ 
0 & 0 & S^\top & R  & 0 & 0 \\ \hdashline
0 & 0 & 0 & 0 & Q^P & S^P \\
0 & 0 & 0 & 0 & (S^P)^\top & R^P 
 \end{array}\right)
\left(\begin{array}{ccc} I & 0 & 0 \\ A & B & E \\ \hdashline 0 & I & 0 \\ C & D & 0 \\ \hdashline 0 & 0 & I \\ F & K & L \end{array}\right) \prec 0.
\end{equation}

Note that $\Pi$ is in general a full-block multiplier (i.e., $\Pi$ is a \emph{full symmetric matrix}).
Furthermore, note that we do not require any condition regarding definiteness of the matrix $R^P$, as opposed to the case in statement (a) of Theorem~\ref{PerformanceSProcedure}. 

Since the statement of Theorem~\ref{OpenSeparationProposition_53} assumes a storage function $V(x)$ with an additive structure, the matrix $P$ in \eqref{NN2_53} is block diagonal matrix, that is, $P=\diag(P_1,P_2)$. 

Next we present the main idea of the proof. 

\emph{\underline{The starting point:}} The LMIs \eqref{NN1_53} and \eqref{NN2_53} are necessarily satisfied for some full block $\Pi$. The LMI \eqref{NN2_53} implies (see Section~\ref{Sec2_B}) that the system $G_0$ is dissipative with respect to the supply function $s(v,w,d,z):=s_H(v,w)+s_{ext}(d,z)$, where 
\begin{equation}
\label{sH}
s_H(v,w)=-\begin{pmatrix}v \\ w\end{pmatrix}^\top \underbrace{\begin{pmatrix}Q & S \\ S^\top & R\end{pmatrix}}_{\Pi} \begin{pmatrix}v \\ w\end{pmatrix}.  
\end{equation}

\emph{\underline{The goal:}} Recall that our goal is to show that for $i=1,2$, the system $G_i$ is strictly dissipative with $s_i(d_i,v_i,z_i,w_i)=s_{i,ext}(d_i,z_i)+s_{i,int}(v_i,w_i)$ with respect to the storage function $V_i(x_i)=x_i^\top P_i x_i$. We will call these conditions \emph{local dissipativity conditions}. Furthermore, we require $s_{1,int}(v_1,w_1)+s_{2,int}(v_2,w_2)=0$, for all $v_1=w_2$, $v_2=w_1$. We will call this condition the \emph{neutrality condition}.
Next we show how we can ``translate'' these desired properties into equivalent conditions imposed on $G_0$ and $\Pi$, via \eqref{NN1_53}, \eqref{NN2_53} and the dissipativity interpretation of \eqref{NN2_53}.

Suppose that the matrices $Q$, $S$ and $R$ from \eqref{NN2_53}, i.e., from \eqref{sH}, are block diagonal, that is, $Q=Q^D:=\diag(Q_1^D,Q_2^D)$, $R=R^D:=\diag(R_1^D,R_2^D)$ and $S=S^D:=\diag(S_1^D,S_2^D)$, with the dimensions of the blocks in conformity with dimensions of $v_1$, $v_2$, $w_1$, $w_2$ that appear in $v$ and $w$ from \eqref{sH}. In that case, and since \textbf{\emph{all}} the matrices (e.g., $P$, $A$, $B$, etc.) that appear in \eqref{NN2_53} are block diagonal (now including also $Q=Q^D$, $R=R^D$, $S=S^D$), the LMI \eqref{NN2_53} can be decomposed into the following two independent LMIs:
\begin{equation}
\label{separateLMIs_53}
\left(\star \right)^\top
\left(\begin{array}{cc:cc:cc} 0 & P_i & 0 & 0 & 0 & 0\\ 
P_i & 0 & 0 & 0 & 0 & 0 \\ \hdashline 
0 & 0 & Q^D_i & S^D_i & 0 & 0\\ 
0 & 0 & (S^D_i)^\top & R^D_i  & 0 & 0 \\ \hdashline
0 & 0 & 0 & 0 & Q^P_i & S^P_i \\
0 & 0 & 0 & 0 & (S^P_i)^\top & R^P_i 
 \end{array}\right)
\left(\begin{array}{ccc} I & 0 & 0 \\ A_i & B_i & E_i \\ \hdashline 0 & I & 0 \\ C_i & D_i & 0 \\ \hdashline 0 & 0 & I \\ F_i & K_i & L_i \end{array}\right) \prec 0,
\end{equation}
$i=1,2$. These two LMIs in turn imply (see Section~\ref{Sec2_B}) that the \emph{local dissipativity conditions} are satisfied with $s_{i,int}:=-\begin{pmatrix}v_i \\ w_i\end{pmatrix}^\top\begin{pmatrix} Q_i^D & S_i^D \\ (S_i^D)^\top & R_i^D \end{pmatrix}\begin{pmatrix}v_i \\ w_i\end{pmatrix}$, for $i=1,2$. Furthermore, with such definitions of $s_{1,int}$ and $s_{2,int}$, it is easy to verify that the \emph{neutrality} condition  is satisfied if and only if
\begin{equation}
\label{Neutrality_53}
\begin{pmatrix} Q_2^D & S_2^D \\ (S_2^D)^\top & R_2^D \end{pmatrix}=
\begin{pmatrix} -R_1^D & -(S_1^D)^\top \\ -S_1^D & -Q_1^D \end{pmatrix}.
\end{equation}
With the abbreviation $\Pi_D:=\begin{pmatrix}Q^D & S^D \\ (S^D)^\top & R^D\end{pmatrix}$, it is crucial to note that for $\Pi_D$ which satisfies the neutrality condition \eqref{Neutrality_53}, the LMI \eqref{NN1_53} is necessarily satisfied with equality sign when $\Pi$ is replaced by $\Pi_D$.

To summarize, the goal we want to reach is equivalent to showing that the dissipativity LMI \eqref{NN2_53} is feasible when the full block multiplier $\Pi$ is replaced with some structured multiplier $\Pi_D$, while $\Pi_D$ also satisfies the condition \eqref{Neutrality_53}. 

\emph{\underline{The main idea and proof overview}}: The main idea is to \emph{construct} a structured multiplier $\Pi_D$ which satisfies all the above described conditions, starting from existence of a full-block multiplier $\Pi$ which satisfies \eqref{NN1_53} and \eqref{NN2_53}. 

The remaining part of the proof is divided in 3 steps. In \emph{Step~1} we derive an alternative dissipativity condition, which is equivalent to feasibility of \eqref{NN2_53}. Note that in this step we exploit the assumption that both $C_1$ and $C_2$ are full row rank matrices. In \emph{Step~2}, this alternative condition is combined with \eqref{NN1_53} to derive an auxiliary result which will be used in \emph{Step~3} to show that the proposed $\Pi_D$ satisfies the alternative dissipativity condition. In \emph{Step~3} we construct $\Pi_D$ which satisfies all required conditions.    

\medskip

\textbf{\emph{Step 1}: Alternative characterization of dissipativity.} \newline
We first consider the case when $n_{w1} < n_1$ and $n_{w2} < n_2$. At the end of the proof we remark on the case when $n_{w1}=n_1$ and $n_{w2}=n_2$, or when we have some combination of the above equalities/inequalities. 

Let $V_1$ span the kernel of $C_1$ and $V_2$ span the kernel of $C_2$. Furthermore, let $W_1$ and $W_2$ be the matrices whose columns span the orthogonal subspaces to $V_1$ and $V_2$, respectively. Let 
%\begin{equation}
%\label{newT_53}
%T:=\begin{pmatrix} V & W & 0 & 0 \\ 0 & 0 & C_W & 0 \\  0 & 0 & 0 & I  \end{pmatrix} 
%\end{equation}
\begin{equation}
\label{newT_53}
T:=\begin{pmatrix} \Xi & 0 & 0 \\ 0  & C_W & 0 \\   0 & 0 & I  \end{pmatrix} ,
\end{equation}
where $\Xi=\begin{pmatrix} V & W \end{pmatrix}$ with $V=\diag(V_1,V_2)$, $W=\diag(W_1,W_2)$, while $C_W=\diag(C_2W_2,C_1W_1)$.
Note that $T$ is a nonsingular square matrix (which would not be the case if $C_1$ and/or $C_2$ would not have full row rank). 
After congruence transformation with $T$ on \eqref{NN2_53} we obtain the inequality \eqref{BigEquation}, 
\begin{equation}
\label{BigEquation}
\begin{pmatrix}
V^\top M V & V^\top M W & V^\top N_A C_W & V^\top N_B \\
W^\top M V & W^\top(M+C^\top R C)W & W^\top(N_A+C^\top S^\top + C^\top R D)C_W & W^\top N_B\\
C_W^\top N_A^\top V & C_W^\top (N_A^\top+SC+D^\top RC)W & C_W^\top (K^\top R^P K + Q+SD+D^\top S^\top + D^\top R D) C_W & C_W^\top N_C \\
N_B^\top V & N_B^\top W & N_C^\top C_W & N_D
\end{pmatrix} \prec 0
\end{equation}
where we have used the abbreviations
\begin{align}
M&:=A^\top P+PA+F^\top R^P F, \label{labelM} \\
N_A&:=PB+F^\top R^P K, \\
N_B&:=PE+F^\top(S^P)^\top+F^\top R^P L, \\
N_C&:=K^\top (S^P)^\top+K^\top R^P L, \\
N_D&:=Q^P+S^PL+L^\top (S^P)^\top +L^\top R^P L.
\end{align}
%%%%%%%%%%%%%%%%%%%%%%%%%%%%%%%%%%%%%%%%%%%%
% ensure that we have normalsize text
\normalsize
% Store the current equation number.
%\setcounter{MYtempeqncnt}{\value{equation}} % !!!!!!!!!!!!!!!!!!
% Set the equation number to one less than the one
% desired for the first equation here.
% The value here will have to changed if equations
% are added or removed prior to the place these
% equations are referenced in the main text.
%\setcounter{equation}{33}
%\begin{footnotesize}

%\end{footnotesize}
% Restore the current equation number.
%\setcounter{equation}{\value{MYtempeqncnt}} % !!!!!!!!!!!!!!!!!!!
% The IEEE uses as a separator
% The spacer can be tweaked to stop underfull vboxes.
%%%%%%%%%%%%%%%%%%%%%%%%%%%%%%%%%%%%%%%%%%%
After applying Schur complement rule twice on the inequality \eqref{BigEquation}, first time with the diagonal block $V^\top M V$ (upper left block in \eqref{BigEquation}) to be inverted, and the second time with the diagonal block $N_D-N_B^\top V (V^\top M V)^{-1} V^\top N_B $ (which appears on diagonal in one of the matrices after the first Schur complement has been applied) to be inverted, we obtain the following inequalities which are equivalent to \eqref{BigEquation}, hence also to \eqref{NN2_53}:
%\begin{scriptsize}
%\begin{subequations}
%\begin{align}
%&V^\top M V \prec 0  \nonumber \\
%&\begin{pmatrix} 
%W^\top C^\top R C W & W^\top (C^\top S^\top+C^\top RD) C_W & 0  \\
%C_W^\top (SC+D^\top R C)W & C_W^\top ( Q +SD+D^\top S^\top + D^\top R D ) C_W & 0  \\
%0 & 0 & 0    \end{pmatrix} + \nonumber \\
%& \begin{pmatrix} 
%W^\top (M-  M V (V^\top M V)^{-1}V^\top M) W & W^\top (N_A  -  M V (V^\top M V)^{-1} V^\top N_A) C_W & W^\top (N_B - M V (V^\top M V)^{-1} V^\top N_B) \\
%\star  & 
%C_W^\top (K^\top R^P K - N_A^\top V (V^\top M V)^{-1} V^\top N_A) C_W & C_W^\top(N_C- N_A^\top V (V^\top M V)^{-1} V^\top N_B) \\
%\star & \star & N_D-N_B^\top V (V^\top M V)^{-1} V^\top N_B 
% \end{pmatrix} \prec 0  \nonumber
%\end{align}
%\end{subequations}
%\end{scriptsize}
%After applying the Schur complement rule on the second inequality above, with the lower right block ($N_D-N_B^\top V (V^\top M V)^{-1} V^\top N_B $) to be inverted, we obtain the following inequalities
\begin{subequations}
\label{DisCompact}
\begin{align}
&V^\top M V \prec 0  \label{DisCompact0} \\
&N_D-N_B^\top V (V^\top M V)^{-1} V^\top N_B  \prec 0, \label{DisCompactA}  \\
&\begin{pmatrix} \star\end{pmatrix}^\top
\begin{pmatrix} \star\end{pmatrix}^\top
\begin{pmatrix} Q &  S \\ S^\top & R \end{pmatrix}
\begin{pmatrix} 0 & I \\ I & D \end{pmatrix}
\begin{pmatrix} CW & 0 \\ 0 & C_W \end{pmatrix} + \nonumber \\
&+\underbrace{
\begin{pmatrix} \star \end{pmatrix}^\top
\begin{pmatrix} \breve{R}_I+\breve{R}_{II} &  \breve{S}_I^\top+\breve{S}_{II}^\top \\ \breve{S}_I+\breve{S}_{II} & \breve{Q}_I+\breve{Q}_{II} \end{pmatrix}
\begin{pmatrix} W & 0 \\ 0 & C_W \end{pmatrix}}_Y\prec 0. \label{DisCompactB}
\end{align}
\end{subequations}

In the above inequalities where we have used the abbreviations
\begin{subequations}
\label{NewLabel_1}
\begin{align}
& \breve{R}_I=M-MV(V^\top M V)^{-1}V^\top M,  \\
& \breve{S}_I=N_A^\top - N_A^\top V (V^\top M V)^{-1}V^\top M,  \\
& \breve{Q}_I=K^\top R^P K-N_A^\top V (V^\top M V)^{-1} V^\top N_A,  \\
& \breve{R}_{II}=-\Big(\star \Big)^\top Z^{-1}\Big(N_B^\top-N_B^\top V (V^\top M V)^{-1} V^\top M \Big),  \\
& \breve{S}_{II}=-\tilde{Z}\,Z^{-1}\, \tilde{\tilde{Z}},  \\
& \breve{Q}_{II}=-\Big(\star \Big)^\top Z^{-1}\Big(N_C^\top- N_B^\top V (V^\top M V)^{-1} V^\top N_A \Big). 
\end{align}
\end{subequations}
with $Z:=N_D-N_B^\top V (V^\top M V)^{-1} V^\top N_B$, $\tilde{Z}:=N_C-N_A^\top V (V^\top M V)^{-1}V^\top N_B$, $\tilde{\tilde{Z}}:= N_B^\top-N_B^\top V (V^\top M V)^{-1} V^\top M$.
Let 
\begin{equation}
\label{NewLabel_2}
\hat{R}:=\breve{R}_{I}+\breve{R}_{II}, \,\,\, \hat{S}:=\breve{S}_{I}+\breve{S}_{II} \,\,\, \text{and} \,\,\, \hat{Q}:=\breve{Q}_{I}+\breve{Q}_{II}.
\end{equation}
%$\hat{R}:=\breve{R}_{I}+\breve{R}_{II}, \hat{S}:=\breve{S}_{I}+\breve{S}_{II}$ and $\hat{Q}:=\breve{Q}_{I}+\breve{Q}_{II}$.
Note that $\hat{R}, \hat{S}$ and $\hat{Q}$ are by construction block diagonal matrices, i.e., we can write
\[\hat{R}=\begin{pmatrix} \hat{R}_1 & 0 \\ 0 & \hat{R}_2 \end{pmatrix}, \quad
\hat{S}=\begin{pmatrix} \hat{S}_1 & 0 \\ 0 & \hat{S}_2 \end{pmatrix}, \quad
\hat{Q}=\begin{pmatrix} \hat{Q}_1 & 0 \\ 0 & \hat{Q}_2 \end{pmatrix},
\]
where $\hat{R}_i \in \Rset^{n_i \times n_i}$ for $i=1,2$, $\hat{Q}_1 \in \Rset^{n_w \times n_w}$, $\hat{Q}_2 \in \Rset^{n_z \times n_z}$, $\hat{S}_1 \in \Rset^{n_w \times n_1}$ and $\hat{S}_2 \in \Rset^{n_z \times n_2}$.

Let us define $L_1=\left(\begin{smallmatrix} C_1 \\ V_1^\top \end{smallmatrix}\right)$, $L_2=\left(\begin{smallmatrix} C_2 \\ V_2^\top \end{smallmatrix}\right)$. Note that $L_1 \in \Rset^{n_1\times n_1}$ and $L_2\in \Rset^{n_2\times n_2}$ are nonsingular square matrices and that
\[L_1 W_1 = \begin{pmatrix} C_1 W_1 \\ 0 \end{pmatrix}, \quad \quad L_2 W_2 = \begin{pmatrix} C_2 W_2 \\ 0 \end{pmatrix}. \]
With $L=\diag(L_1,L_2)$ the matrix $Y$ from \eqref{DisCompactB} can be presented as
%\begin{equation}
%Y=\begin{pmatrix} \star \end{pmatrix}^\top
%\begin{pmatrix} \star \end{pmatrix}^\top 
%\begin{pmatrix} \hat{R} &  \hat{S}^\top \\ \hat{S} & \hat{Q} \end{pmatrix}
%\begin{pmatrix} L^{-1}L & 0 \\ 0 & I \end{pmatrix}
%\begin{pmatrix} W & 0 \\ 0 & C_W \end{pmatrix},
%\end{equation}
%or 
\begin{equation}
\label{tempYa_NEW}
Y=\begin{pmatrix} \star \end{pmatrix}^\top
\begin{pmatrix} L^{-\top}\hat{R} L^{-1} &  L^{-\top} \hat{S}^\top \\ \hat{S} L^{-1} & \hat{Q} \end{pmatrix}
\begin{pmatrix} LW & 0 \\ 0 & C_W \end{pmatrix}.
\end{equation}
Note that
\begin{equation}
LW=\left( \begin{array} {c : c}
C_1 W_1  & 0 \\ 
0 & 0 \\ \hdashline 
0 & C_2 W_2 \\
0 & 0 
\end{array} \right)
\end{equation}
and that we can, in conformity with the above partition of $LW$, partition $L^{-\top}\hat{R} L^{-1}$ and $\hat{S} L^{-1}$ into blocks
\begin{equation}
L^{-\top}\hat{R} L^{-1}= \left( 
\begin{array}{c c : c c}
\tilde{R}^1_{11} & \tilde{R}^1_{12} & 0 & 0 \\
(\tilde{R}^1_{12})^\top & \tilde{R}^1_{22} & 0 & 0 \\ \hdashline
0 & 0 & \tilde{R}^2_{11} & \tilde{R}^2_{12} \\
0 & 0 & (\tilde{R}^2_{12})^\top & \tilde{R}^2_{22}  
\end{array} 
\right), \label{LRL}
\end{equation}
\begin{equation}
\hat{S}L^{-1}=\left( \begin{array}{c c : c c}
\tilde{S}^1_{11} & \tilde{S}^1_{12} & 0 & 0 \\ \hdashline
0 & 0 & \tilde{S}^2_{11} & \tilde{S}^2_{12}
\end{array} \right). \nonumber
\end{equation}
%
%\[L^{-\top}\hat{R} L^{-1}= \left( 
%\begin{array}{c c : c c}
%\tilde{R}^1_{11} & \tilde{R}^1_{12} & 0 & 0 \\
%(\tilde{R}^1_{12})^\top & \tilde{R}^1_{22} & 0 & 0 \\ \hdashline
%0 & 0 & \tilde{R}^2_{11} & \tilde{R}^2_{12} \\
%0 & 0 & (\tilde{R}^2_{12})^\top & \tilde{R}^2_{22}  
%\end{array} 
%\right), \quad \hat{S}L^{-1}=\left( \begin{array}{c c : c c}
%\tilde{S}^1_{11} & \tilde{S}^1_{12} & 0 & 0 \\ \hdashline
%0 & 0 & \tilde{S}^2_{11} & \tilde{S}^2_{12}
%\end{array} \right).   \]
The matrix $Y$ from \eqref{tempYa_NEW}, after multiplications and with $C_W=\diag(C_2 W_2, C_1 W_1)$, becomes
%\begin{equation}
%Y=
%\left( \begin{array} {c c : c c}
%W_1^\top C_1^\top \tilde{R}^1_{11} C_1 W_1 & 0 & W_1^\top C_1^\top (\tilde{S}^1_{11})^\top C_2 W_2 & 0\\
%0 & W_2^\top C_2^\top \tilde{R}^2_{11} C_2 W_2 & 0 & W_2^\top C_2^\top (\tilde{S}^2_{11})^\top C_1 W_1 \\ \hdashline
%W_2^\top C_2^\top \tilde{S}^{1}_{11} C_1 W_1 & 0 & W_2^\top C_2^\top \hat{Q}_1 C_2 W_2 & 0 \\
%0 & W_1^\top C_1^\top \tilde{S}^2_{11}C_2 W_2 & 0 & W_1^\top C_1^\top \hat{Q}_2 C_1 W_1 
%\end{array} \right)
%\end{equation}
%or in more compact form
\begin{equation}
Y=\begin{pmatrix} C W & 0 \\ 0 & C_W \end{pmatrix}^\top
\begin{pmatrix} \cR & \cS^\top \\ \cS & \cQ \end{pmatrix}
\begin{pmatrix} C W & 0 \\ 0 & C_W \end{pmatrix}, 
\end{equation}
where
\begin{subequations}
\label{cc}
\begin{align}
\cR&=
\begin{pmatrix} \cR_1 & 0 \\ 0 & \cR_2 \end{pmatrix}:=
\begin{pmatrix} \tilde{R}^1_{11} & 0 \\ 0 & \tilde{R}^2_{11} \end{pmatrix}, \label{ccA} \\
\cS&=
\begin{pmatrix} \cS_1 & 0 \\ 0 & \cS_2 \end{pmatrix}:=
\begin{pmatrix} \tilde{S}^1_{11} & 0 \\ 0 & \tilde{S}^2_{11} \end{pmatrix}, \label{ccB}\\
\cQ&=
\begin{pmatrix} \cQ_1 & 0 \\ 0 & \cQ_2 \end{pmatrix}:=
\begin{pmatrix} \hat{Q}_1 & 0 \\ 0 & \hat{Q}_2 \end{pmatrix}=\hat{Q}. \label{ccC}
\end{align}
\end{subequations}
The inequality \eqref{DisCompactB} now reads as

\begin{equation}
\begin{pmatrix} \star\end{pmatrix}^\top
\left(
\begin{pmatrix} \star \end{pmatrix}^\top 
\begin{pmatrix} Q & S \\ S^\top & R \end{pmatrix}
\begin{pmatrix} 0 & I \\ I & D \end{pmatrix}
+
\begin{pmatrix} \cR & \cS^\top \\ \cS & \cQ \end{pmatrix}
\right)
\begin{pmatrix}C W & 0 \\ 0 & C_W \end{pmatrix}\prec 0. 
\end{equation}

Since $CW$ and $C_W$ are nonsigular square matrices, \eqref{DisCompactB} is equivalent to
%\begin{equation}
%\label{QSRinequalityTemp_53}
%\begin{pmatrix} \star \end{pmatrix}^\top 
%\begin{pmatrix} Q & S \\ S^\top & R \end{pmatrix}
%\begin{pmatrix} 0 & I \\ I & D \end{pmatrix}
%+
%\begin{pmatrix} \star \end{pmatrix}^\top
%\begin{pmatrix} \cQ & \cS \\ \cS^\top & \cR \end{pmatrix}
%\begin{pmatrix} 0 & I \\ I & 0 \end{pmatrix}
%\prec 0,
%\end{equation}
%or equivalently 
\begin{equation}
\label{QSRinequality_53}
\begin{pmatrix} I & 0 \\ D & I \end{pmatrix}^\top
\begin{pmatrix} Q & S \\ S^\top & R \end{pmatrix}
\begin{pmatrix} I & 0 \\ D & I \end{pmatrix}
+
\begin{pmatrix} \cQ & \cS \\ \cS^\top & \cR \end{pmatrix}
\prec 0. 
\end{equation}
Recall that $Q$, $S$ and $R$ are full matrices, while $\cQ$, $\cS$ and $\cR$ are block diagonal matrices, derived from the parameters of the systems ($A_i, B_i, C_i), i=1,2$ and the Laypunov matrices $P_1, P_2$. 
The derived results up to now can be summarized in the following equivalence
\begin{equation}
\label{Equivalence_53}
\text{\eqref{DisCompact0}, \eqref{DisCompactA}, \eqref{QSRinequality_53}} \quad \iff \quad \text{\eqref{NN2_53}}.
\end{equation}
The inequalities \eqref{DisCompact0}, \eqref{DisCompactA} do not depend on $Q$, $S$ and $R$, and since our goal is to devise new multipliers $Q$, $S$ and $R$, which define the interconnection neutral supply rates, in the remainder we focus on the inequality \eqref{QSRinequality_53}.

\medskip

\textbf{\emph{Step 2}: Auxiliary result.} \newline
After congruence transformation of \eqref{QSRinequality_53} with $\left(\begin{smallmatrix} I & 0 \\ -D & I \end{smallmatrix}\right)$ we have 
\begin{equation}
\label{QSRinequality1_53}
\begin{pmatrix} Q & S \\ S^\top & R \end{pmatrix}+
\begin{pmatrix} I & 0 \\ -D & I \end{pmatrix}^\top
\begin{pmatrix} \cQ & \cS \\ \cS^\top & \cR \end{pmatrix}
\begin{pmatrix} I & 0 \\ -D & I \end{pmatrix} \prec 0.
\end{equation}
Pre-multiplying \eqref{QSRinequality1_53} with $\left(\begin{smallmatrix} H \\ I \end{smallmatrix}\right)^\top$ and post-multiplying with $\left(\begin{smallmatrix} H \\ I \end{smallmatrix}\right)$, and using \eqref{NN1_53}, we obtain
%\begin{footnotesize}
\begin{align}
&\cM = \begin{pmatrix} \cM_1 & \cM_2 \\ \cM_2^\top & \cM_3 \end{pmatrix} \prec 0, \label{StructuredQSR_53} \\
& \cM_1=\cR_1+\cQ_2-\cS_2 D_2 - D_2^\top \cS_2^\top + D_2^\top \cR_2 D_2, \nonumber \\
& \cM_2=\cS_1^\top + \cS_2 - \cR_1 D_1 - D_2^\top \cR_2, \nonumber \\
&\cM_3= \cQ_1+\cR_2-\cS_1 D_1 - D_1^\top \cS_1^\top + D_1^\top \cR_1 D_1. \nonumber
\end{align}

\medskip 

\textbf{\emph{Step 3}: Construction of structured multipliers.} \newline
Our goal is now to find a new multiplier 
\begin{equation}
\label{MultiplierPiNew}
\Pi_D=\begin{pmatrix} Q^D & S^D \\ (S^D)^\top & R^D \end{pmatrix}
\end{equation}
 with block diagonal matrices $Q^D=\diag(Q^D_1,Q^D_2)$, $S^D=\diag(S^D_1,S^D_2)$ and $R^D=\diag(R^D_1,R^D_2)$, which can replace $\Pi$ from \eqref{MultiplierPi}  in the inequality \eqref{NN2_53} so that this inequality remains to hold, while in addition the multiplier $\Pi_D$ satisfies the neutrality condition \eqref{Neutrality_53}.
Due to \eqref{Equivalence_53}, the multiplier $\Pi_D$ satisfies the inequality \eqref{NN2_53} if and only if

\begin{equation}
\label{Goal1a}
\begin{pmatrix} I & 0 \\ D & I \end{pmatrix}^\top
\begin{pmatrix} Q^D & S^D \\ (S^D)^\top & R^D \end{pmatrix}
\begin{pmatrix} I & 0 \\ D & I \end{pmatrix}
+
\begin{pmatrix} \cQ & \cS \\ \cS^\top & \cR \end{pmatrix}
\prec 0.
\end{equation}

%
%\begin{align}
%&\begin{pmatrix} I & 0 \\ D & I \end{pmatrix}^\top
%\begin{pmatrix} Q^D & S^D \\ (S^D)^\top & R^D \end{pmatrix}
%\begin{pmatrix} I & 0 \\ D & I \end{pmatrix}
%+
%\begin{pmatrix} \cQ & \cS \\ \cS^\top & \cR \end{pmatrix}
%\prec 0, \label{Goal1a}\\
%&\begin{pmatrix} H \\ I \end{pmatrix}^\top 
%\begin{pmatrix} Q^D & S^D \\ (S^D)^\top & R^D \end{pmatrix}
%\begin{pmatrix} H \\ I \end{pmatrix} \succeq 0. \label{Goal1b}
%\end{align}
%
Consider the multiplier $\Pi_D$ defined with the following matrices

\begin{subequations}
\label{MultNew}
\begin{align}
Q^D_1&= \alpha ( -\cQ_1+D_1^\top \cS_1^\top+\cS_1 D_1-D_1^\top \cR_1 D_1 )+(1-\alpha)\cR_2, \label{MultA} \\
S^D_1&= \alpha(-\cS_1+D_1^\top \cR_1)+(1-\alpha)(\cS_2^\top-\cR_2 D_2), \label{MultB}\\
R^D_1&= \alpha (-\cR_1)+(1-\alpha)(\cQ_2-\cS_2 D_2-D_2^\top \cS_2^\top + D_2^\top \cR_2 D_2), \label{MultC}\\
Q^D_2&=-R^D_1, \quad S^D_2=-(S^D_1)^\top, \quad R^D_2=-Q^D_1,  \label{MultD}
\end{align}
\end{subequations}

where $\alpha$ is a real number from the interval $(0,1)$. We have 
\begin{align}
&\begin{pmatrix} I & 0 \\ D & I \end{pmatrix}^\top
\begin{pmatrix} Q^D & S^D \\ (S^D)^\top & R^D \end{pmatrix}
\begin{pmatrix} I & 0 \\ D & I \end{pmatrix}
+
\begin{pmatrix} \cQ & \cS \\ \cS^\top & \cR \end{pmatrix} = \nonumber\\
&\left( \star \right)^\top
\left( \begin{array}{c c : c c} 
(1-\alpha)\cM_3 & 0 & (1-\alpha)\cM_2^\top & 0 \\ 
0 & \alpha\cM_1 & 0 & \alpha\cM_2 \\ \hdashline
(1-\alpha)\cM_2 & 0 & (1-\alpha)\cM_1 & 0 \\
0 & \alpha\cM_2^\top & 0 & \alpha\cM_3 
\end{array} \right)
\begin{pmatrix} I & 0 \\ D & I \end{pmatrix} \prec 0, \label{Condition1_53}
\end{align}
that is, the condition \eqref{Goal1a} is satisfied. The inequality in \eqref{Condition1_53} follows from \eqref{StructuredQSR_53}. 
Finally, note that due to \eqref{MultD}, $\Pi_D$ by construction also satisfies the neutrality condition \eqref{Neutrality_53}. 

This concludes the proof for the case when $n_{w1} < n_1$ and $n_{w2} < n_2$. In the case when either $n_{w1}=n_1$ or $n_{w2}=n_2$ (or both), the proof follows along the similar lines, except that the congruence transformation with the matrix $T$, defined in \eqref{newT_53}, is either completely omitted (when $n_{w1}=n_1$, $n_{w2}=n_2$) or $T$ is suitably modified. 
\end{proof}

\subsection{Proof of Proposition~\ref{PropositionRank}}
\label{Sec6_B}
We present proof of the extension of Theorem~\ref{OpenSeparationProposition_53} in the case when $C_1$ and $C_2$ are not full row rank. From there, proof of the analogous extension of Corollary~\ref{Cor1} follows directly by omitting the matrices related to exogenous inputs/outputs $d_1, d_2, z_1, z_2$, i.e., the matrices  $E_i, F_i, K_i, L_i$, $i=1,2$, and indeed by omitting the terms related to the external supply functions $s_{1,ext}$ and $s_{2,ext}$. 

\medskip

\begin{proof}[\textbf{Step 0: \textit{The main idea and overview}}]
%\textbf{\emph{Step 0}: The main idea and overview.} \newline
Suppose that $C_1 \in \Rset^{n_{w1}\times n_1}$ and $C_2 \in \Rset^{n_{w2}\times n_2}$ are row rank deficient. Let $\tilde{n}_{w1}$ ($\tilde{n}_{w1} < n_{w1}$) and $\tilde{n}_{w2}$ $(\tilde{n}_{w2} < n_{w2})$ denote respectively rank of $C_1$ and rank of $C_2$. Then, without loss of generality, we can take
\begin{equation}
\label{Cmatrices}
C_1=\begin{pmatrix} J_1 \\ I_{\tilde{n}_{w1}} \end{pmatrix} \tilde{C}_1, \quad \quad C_2=\begin{pmatrix} J_2 \\ I_{\tilde{n}_{w2}} \end{pmatrix} \tilde{C}_2,
\end{equation}  
%\begin{subequations}
%\label{Cmatrices}
%\begin{align}
%C_1&=\begin{pmatrix} J_1 \\ I_{\tilde{n}_{w1}} \end{pmatrix} \tilde{C}_1, \label{CmatricesA}\\
%C_2&=\begin{pmatrix} J_2 \\ I_{\tilde{n}_{w2}} \end{pmatrix} \tilde{C}_2 \label{CmatricesB}
%\end{align}
%\end{subequations} 
where $\tilde{C}_1 \in \Rset^{\tilde{n}_{w1} \times n_1}$ and $\tilde{C}_2 \in \Rset^{\tilde{n}_{w2} \times n_2}$ are full row rank matrices. The matrix $J_1 $ defines the first $(n_{w1}-\tilde{n}_{w1})$ rows of $C_1$ as linear combinations of rows of $\tilde{C}_1$, while the matrix $J_2$ does the same for $C_2$.   

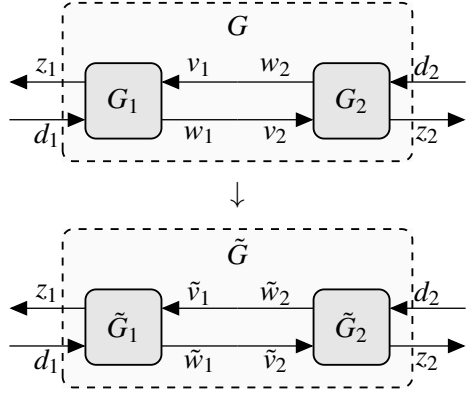
\begin{figure}[h!]
\centering
 \begin{tikzpicture}[scale=1]

\draw [line width=0.7pt, rounded corners, dashed] (1.2,0.7) rectangle (5.8,2.8);
\begin{pgfonlayer}{background}
\filldraw [line width=0.1mm, black!2] (1.2,0.7) rectangle (5.8,2.8);
\end{pgfonlayer}

\node at (3.5,2.5) {$G$};
 
%%%%% 
 
\draw [line width=0.7pt, rounded corners] (1.5,1) rectangle (2.5,2);
\draw [line width=0.7pt, rounded corners] (4.5,1) rectangle (5.5,2);

\draw [line width=0.5pt, ->, >= triangle 45, rounded corners] (1.5,1.75) -- (0.5,1.75);
\draw [line width=0.5pt, ->, >= triangle 45, rounded corners] (0.5,1.25) -- (1.5,1.25);

\draw [line width=0.5pt, ->, >= triangle 45, rounded corners] (3.5,1.75) -- (2.5,1.75);
\draw [line width=0.5pt,  rounded corners] (2.5,1.25) -- (3.5,1.25);

\draw [line width=0.5pt,  rounded corners] (4.5,1.75) -- (3.5,1.75);
\draw [line width=0.5pt, ->, >= triangle 45, rounded corners] (3.5,1.25) -- (4.5,1.25);

\draw [line width=0.5pt, ->, >= triangle 45, rounded corners] (6.5,1.75) -- (5.5,1.75);
\draw [line width=0.5pt, ->, >= triangle 45, rounded corners] (5.5,1.25) -- (6.5,1.25);

%\draw [red, line width=2.5pt, dashed] (3.5,0.7) -- (3.5,2.3);

\node at (1,1.95) {$z_1$};
\node at (1,1.05) {$d_1$};
\node at (2,1.5) {$G_1$};
\node at (3,1.95) {$v_1$};
\node at (3,1.05) {$w_1$};

\node at (4,1.95) {$w_2$};
\node at (4,1.05) {$v_2$};
\node at (5,1.5) {$G_2$};
\node at (6,1.95) {$d_{2}$};
\node at (6,1.05) {$z_{2}$};
%%%%%%%%%%%%%%
\begin{pgfonlayer}{background}
\filldraw [line width=0.1mm,rounded corners,black!10] (1.5,1) rectangle (2.5,2);
\end{pgfonlayer}

\begin{pgfonlayer}{background}
\filldraw [line width=0.1mm,rounded corners,black!10] (4.5,1) rectangle (5.5,2);
\end{pgfonlayer}

%%%%%%%%%%%%%%%%%%%%%%%%%%%%%%%%%%%%%%%%%%%%%%%%%%
%%%%%%%%%%%%%%%%%%%%%%%%%%%%%%%%%%%%%%%%%%%%%%%%%%

\node at (3.5,0.25) {$\downarrow$};

%%%%%%%%%%%%%%%%%%%%%%%%%%%%%%%%%%%%%%%%%%%%%%%%%
%%%%%%%%%%%%%%%%%%%%%%%%%%%%%%%%%%%%%%%%%%%%%%%%
 
\draw [line width=0.7pt, rounded corners, dashed] (1.2,0.7-3) rectangle (5.8,2.8-3);
\begin{pgfonlayer}{background}
\filldraw [line width=0.1mm, black!2] (1.2,0.7-3) rectangle (5.8,2.8-3);
\end{pgfonlayer}

\node at (3.5,2.5-3) {$\tilde{G}$};
 
%%%%% 
 
\draw [line width=0.7pt, rounded corners] (1.5,1-3) rectangle (2.5,2-3);
\draw [line width=0.7pt, rounded corners] (4.5,1-3) rectangle (5.5,2-3);

\draw [line width=0.5pt, ->, >= triangle 45, rounded corners] (1.5,1.75-3) -- (0.5,1.75-3);
\draw [line width=0.5pt, ->, >= triangle 45, rounded corners] (0.5,1.25-3) -- (1.5,1.25-3);

\draw [line width=0.5pt, ->, >= triangle 45, rounded corners] (3.5,1.75-3) -- (2.5,1.75-3);
\draw [line width=0.5pt,  rounded corners] (2.5,1.25-3) -- (3.5,1.25-3);

\draw [line width=0.5pt,  rounded corners] (4.5,1.75-3) -- (3.5,1.75-3);
\draw [line width=0.5pt, ->, >= triangle 45, rounded corners] (3.5,1.25-3) -- (4.5,1.25-3);

\draw [line width=0.5pt, ->, >= triangle 45, rounded corners] (6.5,1.75-3) -- (5.5,1.75-3);
\draw [line width=0.5pt, ->, >= triangle 45, rounded corners] (5.5,1.25-3) -- (6.5,1.25-3);

%\draw [red, line width=2.5pt, dashed] (3.5,0.7) -- (3.5,2.3);

\node at (1,1.95-3) {$z_1$};
\node at (1,1.05-3) {$d_1$};
\node at (2,1.5-3) {$\tilde{G}_1$};
\node at (3,1.95-3) {$\tilde{v}_1$};
\node at (3,1.05-3) {$\tilde{w}_1$};

\node at (4,1.95-3) {$\tilde{w}_2$};
\node at (4,1.05-3) {$\tilde{v}_2$};
\node at (5,1.5-3) {$\tilde{G}_2$};
\node at (6,1.95-3) {$d_{2}$};
\node at (6,1.05-3) {$z_{2}$};
%%%%%%%%%%%%%%
\begin{pgfonlayer}{background}
\filldraw [line width=0.1mm,rounded corners,black!10] (1.5,1-3) rectangle (2.5,2-3);
\end{pgfonlayer}

\begin{pgfonlayer}{background}
\filldraw [line width=0.1mm,rounded corners,black!10] (4.5,1-3) rectangle (5.5,2-3);
\end{pgfonlayer}

 \end{tikzpicture}
\caption{Redefinition of local systems and interconnecting signals.}
\label{GandG}
\end{figure}

Instead of considering interconnection of systems $G_1$ and $G_2$ given by \eqref{Systems_53} with $D_1=0$ and $D_2=0$ (let us denote this interconnection with $G$, as done in Figure~\ref{GandG}), we consider dissipativity properties of the system obtained by interconnecting $\tilde{G}_1$ and $\tilde{G}_2$ with

\begin{equation}
\tilde{G}_i: \,
\begin{pmatrix} \dot{x}_i \\ \tilde{w}_i \\ z_i \end{pmatrix}=
\begin{pmatrix} A_i & \tilde{B}_i & E_i \\
\tilde{C}_i & 0 & 0 \\
F_i & \tilde{K}_i & L_i \end{pmatrix}
\begin{pmatrix} x_i \\ \tilde{v}_i \\ d_i \end{pmatrix}, \quad i=1,2, 
\end{equation}

where 
%$\tilde{B}_1=B_1\begin{pmatrix} J_2 \\ I \end{pmatrix}$, $\tilde{K}_1=K_1\begin{pmatrix} J_2 \\ I \end{pmatrix}$ and $\tilde{B}_2=B_2\begin{pmatrix} J_1 \\ I \end{pmatrix}$, $\tilde{K}_2=K_2\begin{pmatrix} J_1 \\ I \end{pmatrix}$. 

\begin{equation}
\label{BKmatrices}
\tilde{B}_1=B_1\begin{pmatrix} J_2 \\ I \end{pmatrix}, \,\, \tilde{K}_1=K_1\begin{pmatrix} J_2 \\ I \end{pmatrix}, \,\, \tilde{B}_2=B_2\begin{pmatrix} J_1 \\ I \end{pmatrix}, \,\, \tilde{K}_2=K_2\begin{pmatrix} J_1 \\ I \end{pmatrix}.
\end{equation}

Let $\tilde{G}$ denote the interconnection of $\tilde{G}_1$ and $\tilde{G}_2$ obtained with constraints $\tilde{v}_1=\tilde{w}_2$ and $\tilde{v}_2=\tilde{w}_1$. 
Then $\tilde{G}$ and $G$ are the same system when observed from the input-output behaviour of the exogenous signals $d:=\col(d_1,d_2)\rightarrow z=\col(z_1,z_2)$. The only difference is in the definitions of internal ``subsystems'' ($G_1$ and $G_2$ in $G$; $\tilde{G}_1$ and $\tilde{G}_2$ in $\tilde{G}$) and the corresponding interconnection signals:  $w_1$, $v_1$ in $G$ and $\tilde{w}_1$, $\tilde{v}_1$ in $\tilde{G}$, as presented in Figure~\ref{GandG}.

%%%%%%%%%%%%%%%%%%%%%%%%%%%%%%%%%%%%%%%%%%%%%%%%%%%%%%%%%%%%%%%%%%%%%%
%%%%%%%%%%%%%%%%%%%%%%%%%%%%%%%%%%%%%%%%%%%%%%%%%%%%%%%%%%%%%%%%%%%%%%
%%%%%%%%%%%%%%%%%%%%%%%%%%%%%%%%%%%%%%%%%%%%%%%%%%%%%%%%%%%%%%%%%%%%%%

\begin{figure}[h!]
\centering
 \begin{tikzpicture}%[scale=0.6]
 
%%%%%%%%%%%%%%%%%%%%%%%%%%%%%%%%%%%%%%%%%%%%%%%%%%%%%%%%

\draw [line width=0.7pt, rounded corners, dashed] (0.5,-0.5) rectangle (7.3,3);
\begin{pgfonlayer}{background}
\filldraw [line width=0.1mm, black!2] (0.5,-0.5) rectangle (7.3,3);
\end{pgfonlayer}

\draw [line width=0.7pt, rounded corners, dashed] (4.5,0) rectangle (11,3.5);
\begin{pgfonlayer}{background}
\filldraw [line width=0.1mm, black!2] (4.5,0) rectangle (11,3.5);
\end{pgfonlayer}

%%%%%%%%%%%%%%%%%%%%%%%%%%%%%%%%%%%%%%%%%%%%%%%%%%%%%%%%%%%% 
\draw [line width=0.7pt, rounded corners] (1,1) rectangle (2,2);
\begin{pgfonlayer}{background}
\filldraw [line width=0.1mm, black!10,rounded corners] (1,1) rectangle (2,2);
\end{pgfonlayer} 

\draw [line width=0.7pt, rounded corners] (5,0.5) rectangle (6.5,2.5);
\begin{pgfonlayer}{background}
\filldraw [line width=0.1mm, black!10, rounded corners] (5,0.5) rectangle (6.5,2.5);
\end{pgfonlayer}

\draw [line width=0.7pt, rounded corners] (9.5,1) rectangle (10.5,2);
\begin{pgfonlayer}{background}
\filldraw [line width=0.1mm, black!10,rounded corners] (9.5,1) rectangle (10.5,2);
\end{pgfonlayer} 

%%%%%%%

\node at (1.5,1.5) {$B_1$};
\node at (5.75,1.5) {$\begin{pmatrix} J_2 \\ I \end{pmatrix}$};
\node at (10,1.5) {$\tilde{C}_2$};

%%%%%%%

\draw [line width=0.5pt, ->, >= triangle 45, rounded corners] (5,1.5) -- (2,1.5);
\draw [line width=0.5pt, ->, >= triangle 45, rounded corners] (1,1.5) -- (-0.75,1.5);
\draw [line width=0.5pt, ->, >= triangle 45, rounded corners] (9.5,1.5) -- (6.5,1.5);
\draw [line width=0.5pt, ->, >= triangle 45, rounded corners] (12,1.5) -- (10.5,1.5);

%%%%%%%

\node at (1.2,0) {$\tilde{B}_1$};
\node at (10.3,2.9) {$C_2$};

\node at (11.6,1.75) {$x_2$};
\node at (9,1.75) {$\tilde{w}_2$};
\node at (7.75,1.75) {$\tilde{v}_1$};
\node at (4,1.75) {$w_2$};
\node at (2.8,1.75) {$v_1$};

 \end{tikzpicture}
\caption{Interconnection between $G_1$ and $G_2$, versus interconnection between $\tilde{G}_1$ and $\tilde{G}_2$.}
\label{Substitution}
\end{figure}
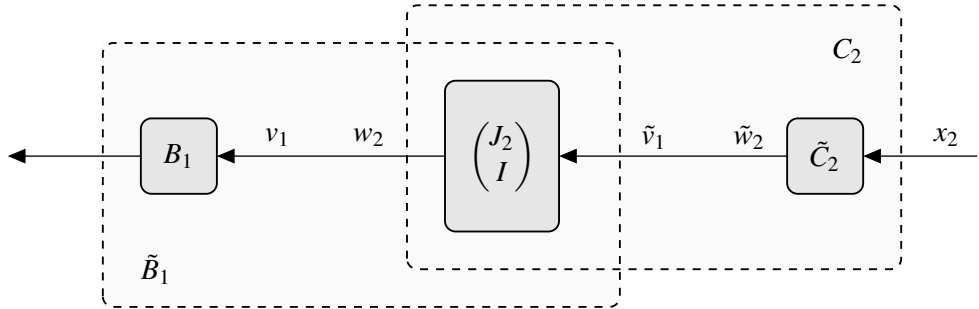

%%%%%%%%%%%%%%%%%%%%%%%%%%%%%%%%%%%%%%%%%%%%%%%%%%%%%%%%%%%%%%%%%%%%%%
%%%%%%%%%%%%%%%%%%%%%%%%%%%%%%%%%%%%%%%%%%%%%%%%%%%%%%%%%%%%%%%%%%%%%%
%%%%%%%%%%%%%%%%%%%%%%%%%%%%%%%%%%%%%%%%%%%%%%%%%%%%%%%%%%%%%%%%%%%%%%
The main principle behind the above described redefinition of signals and system matrices is illustrated in Figure~\ref{Substitution}, which presents relation of the interconnection signal between $G_1$ and $G_2$ (signal $v_1=w_2$), and the signal between $\tilde{G_1}$ and $\tilde{G}_2$ (signal $\tilde{v}_1=\tilde{w}_2$). The same figure illustrates the relations between $\tilde{B}_1$ and $B_1$ as well as $\tilde{C}_2$ and $C_2$. By altering all the indices in the figure from $1$ to $2$ and from $2$ to $1$, we obtain illustration of the relations between $\tilde{B}_2$ and $B_2$, between $\tilde{C}_1$ and $C_1$, and between signals pars $(v_2,w_1)$ and $(\tilde{v}_2,\tilde{w}_1)$.

We emphasize that the above redefinition of signals, matrices and subsystems cannot be done in case when $D_1 \neq 0$ and/or $D_2 \neq 0$.  

\underline{\emph{The main idea}}: Note that in $\tilde{G}$, both $\tilde{C}_1$ and $\tilde{C}_2$ are full row rank matrices and therefore we can apply Theorem~\ref{OpenSeparationProposition_53} to infer existence of interconnection neutral supply functions $\tilde{s}_{1,int}(\tilde{v}_1,\tilde{w}_1)$ and $\tilde{s}_{2,int}(\tilde{v}_2,\tilde{w}_2)$ on the interconnecting signals of $\tilde{G}$. 
The main idea of the proof is to devise interconnection neutral supply functions $s_{1,int}(v_1,w_1)$ and $s_{2,int}(v_2,w_2)$ on the interconnecting signals of $G$, starting from existence of $\tilde{s}_{1,int}$ and $\tilde{s}_{2,int}$.

\underline{\emph{Overview}}:The remainder of the proof is divided in 4 steps. In \emph{Step~1} we define $\tilde{s}_{1,int}$ and $\tilde{s}_{2,int}$ using a triplet of matrices $(\tilde{Q},\tilde{R},\tilde{S})$. The goal is to find a new triplet $(Q,R,S)$ which defines $s_{1,int}(v_1,w_1)$ and $s_{2,int}(v_2,w_2)$. Since the interconnection signals $\tilde{v}_1$, $\tilde{w_1}$, $\tilde{v}_2$, $\tilde{w_2}$ have lower spatial dimension than the corresponding signals $v_1$, $w_1$, $v_2$, $w_2$, i.e., $\tilde{v}_1(t)\in \Rset^{\tilde{n}_{w1}}$, $v_1(t)\in \Rset^{n_{w1}}$ with $\tilde{n}_{w1}<n_{w1}$, etc., the matrices $(Q,R,S)$ are characterized with larger number of rows/columns than the corresponding matrices $(\tilde{Q},\tilde{R},\tilde{S})$. We say that $(Q,R,S)$ are constructed by an appropriate \emph{extension} of $(\tilde{Q},\tilde{R},\tilde{S})$. In \emph{Step~1} we present a set of algebraic conditions for this extension. In \emph{Step~2} we present a procedure how to construct $Q$ by an appropriate extension from $\tilde{Q}$. In \emph{Step~3} and \emph{Step~4} we present similar extensions to construct $R$ and $S$, respectively.     

\medskip

\textbf{\emph{Step 1}: Conditions for extension.} \newline
Let 
\begin{equation}
\label{s1int}
\tilde{s}_{1,int}(\tilde{v}_1,\tilde{w}_1)=\begin{pmatrix} \tilde{v}_1 \\ \tilde{w}_1 \end{pmatrix}^\top \begin{pmatrix} \tilde{Q} & \tilde{S} \\ \tilde{S}^\top & \tilde{R} \end{pmatrix}\begin{pmatrix} \tilde{v}_1 \\ \tilde{w}_1 \end{pmatrix}. 
\end{equation}
Then, with the neutrality condition $(\tilde{v}_1=\tilde{w}_2, \tilde{v}_2=\tilde{w}_1) \implies \tilde{s}_{1,int}(\tilde{v}_1,\tilde{w}_1)+\tilde{s}_{2,int}(\tilde{v}_2,\tilde{w}_2)=0$ accounted for, we have 
\begin{equation}
\label{s2int}
\tilde{s}_{2,int}(\tilde{v}_2,\tilde{w}_2)=\begin{pmatrix} \tilde{v}_2 \\ \tilde{w}_2 \end{pmatrix}^\top \begin{pmatrix} -\tilde{R} & -\tilde{S}^\top \\ -\tilde{S} & -\tilde{Q} \end{pmatrix}\begin{pmatrix} \tilde{v}_2 \\ \tilde{w}_2 \end{pmatrix}.
\end{equation} 

%Let $\tilde{s}_{1,int}(\tilde{v}_1,\tilde{w}_1)=\begin{pmatrix} \tilde{v}_1 \\ \tilde{w}_1 \end{pmatrix}^\top \begin{pmatrix} \tilde{Q} & \tilde{S} \\ \tilde{S}^\top & \tilde{R} \end{pmatrix}\begin{pmatrix} \tilde{v}_1 \\ \tilde{w}_1 \end{pmatrix}$. Then, with the neutrality condition $(\tilde{v}_1=\tilde{w}_2, \tilde{v}_2=\tilde{w}_1) \implies \tilde{s}_{1,int}(\tilde{v}_1,\tilde{w}_1)+\tilde{s}_{2,int}(\tilde{v}_2,\tilde{w}_2)=0$ accounted for, we have $\tilde{s}_{2,int}(\tilde{v}_2,\tilde{w}_2)=\begin{pmatrix} \tilde{v}_2 \\ \tilde{w}_2 \end{pmatrix}^\top \begin{pmatrix} -\tilde{R} & -\tilde{S}^\top \\ -\tilde{S} & -\tilde{Q} \end{pmatrix}\begin{pmatrix} \tilde{v}_2 \\ \tilde{w}_2 \end{pmatrix}$. 
Furthermore, let 
\[s_{i,ext}(d_i,z_i)=\left(\begin{smallmatrix} d_i \\ z_i \end{smallmatrix}\right)^\top \left(\begin{smallmatrix} Q^P_i & S^P_i \\ (S^P_i)^\top & R^P_i \end{smallmatrix}\right)\left(\begin{smallmatrix} d_i \\ z_i \end{smallmatrix}\right),\,\,\,\, i=1,2.\] 
With LMI characterization from Section~\ref{Sec2_B}, strict dissipativity of $\tilde{G}_1$ with respect to $s_{1,ext}(d_1,z_1)+\tilde{s}_{1,int}(\tilde{v}_1,\tilde{w}_1)$ and strict dissipativity of $\tilde{G}_2$ with respect to $s_{2,ext}(d_2,z_2)+\tilde{s}_{2,int}(\tilde{v}_2,\tilde{w}_2)$ are respectively given by the following two inequalities
\begin{multline}
\label{D1}
\begin{pmatrix} \star  \end{pmatrix}^\top
\begin{pmatrix} 0 & P_1 \\ P_1 & 0 \end{pmatrix} 
\begin{pmatrix} I & 0  & 0 \\ A_1 & \tilde{B}_1 & E_1  \end{pmatrix}
\prec
\begin{pmatrix} \star  \end{pmatrix}^\top
\begin{pmatrix} \tilde{Q} & \tilde{S} \\ \tilde{S}^\top & \tilde{R} \end{pmatrix}
\begin{pmatrix} 0 & I & 0 \\ \tilde{C}_1 & 0 & 0  \end{pmatrix} \\
+\begin{pmatrix} \star  \end{pmatrix}^\top
\begin{pmatrix} Q^P_1 & S^P_1 \\ (S^P_1)^\top & R^P_1 \end{pmatrix}
\begin{pmatrix} 0 & 0 & I \\ F_1 & \tilde{K}_1 & L_1  \end{pmatrix},
\end{multline}
\begin{multline}
\label{D2}
\begin{pmatrix} \star  \end{pmatrix}^\top
\begin{pmatrix} 0 & P_2 \\ P_2 & 0 \end{pmatrix} 
\begin{pmatrix} I & 0  & 0 \\ A_2 & \tilde{B}_2 & E_2  \end{pmatrix}
\prec
\begin{pmatrix} \star  \end{pmatrix}^\top
\begin{pmatrix} -\tilde{R} & -\tilde{S}^\top \\ -\tilde{S} & -\tilde{Q} \end{pmatrix}
\begin{pmatrix} 0 & I & 0 \\ \tilde{C}_2 & 0 & 0  \end{pmatrix} \\
+\begin{pmatrix} \star  \end{pmatrix}^\top
\begin{pmatrix} Q^P_2 & S^P_2 \\ (S^P_2)^\top & R^P_2 \end{pmatrix}
\begin{pmatrix} 0 & 0 & I \\ F_2 & \tilde{K}_2 & L_2  \end{pmatrix}.
\end{multline}  

Consider the equality

\begin{equation}
\label{S1}
\begin{pmatrix} \star \end{pmatrix}^\top
\begin{pmatrix} Q & S \\ S^\top & R \end{pmatrix}
\begin{pmatrix} \begin{pmatrix} J_2 \\ I \end{pmatrix} & 0 \\ 
0 & \begin{pmatrix} J_1 \\ I \end{pmatrix} \end{pmatrix}
 = \begin{pmatrix} \tilde{Q} & \tilde{S} \\ \tilde{S}^\top & \tilde{R} \end{pmatrix},
\end{equation}

%
%\begin{equation}
%\label{S1}
%\left(\begin{array}{c:c} J_2 & 0 \\ I & 0 \\ \hdashline 0 & J_1 \\ 0 & I \end{array}\right)^\top
%\begin{pmatrix} Q & S \\ S^\top & R \end{pmatrix}
%\left(\begin{array}{c:c} J_2 & 0 \\ I & 0 \\ \hdashline 0 & J_1 \\ 0 & I \end{array}\right)
% = \begin{pmatrix} \tilde{Q} & \tilde{S} \\ \tilde{S}^\top & \tilde{R} \end{pmatrix},
%\end{equation}
%
which is a linear equation in $Q, S, R$ for some known $\tilde{Q}, \tilde{S}, \tilde{R}$.
Substituting \eqref{S1} into \eqref{D1}, together with \eqref{Cmatrices} and \eqref{BKmatrices}, we obtain
%
%\begin{equation}
%\label{D1A}
%\begin{pmatrix} \star \end{pmatrix}^\top
%X_1
%\left( \begin{array}{c:c} I & 0 \\ \hdashline 0 & J_2 \\ 0 & I \end{array} \right)\prec
%\left( \begin{array}{c:c} I & 0 \\ \hdashline 0 & J_2 \\ 0 & I \end{array} \right)^\top
%(Y_1^{int}+Y_1^{ext})
%\left( \begin{array}{c:c} I & 0 \\ \hdashline 0 & J_2 \\ 0 & I \end{array} \right) 
%\end{equation}
%

\begin{equation}
\label{D1A}
\tilde{W}_1^\top
X_1
\underbrace{\begin{pmatrix} I & 0 & 0 \\ 0 & \begin{pmatrix} J_2 \\ I \end{pmatrix} & 0 \\ 0 & 0 & I \end{pmatrix}}_{\tilde{W}_1}\prec
\tilde{W}_1^\top
(Y_1^{int}+Y_1^{ext})
\underbrace{\begin{pmatrix} I & 0 & 0 \\ 0 & \begin{pmatrix} J_2 \\ I \end{pmatrix} & 0 \\ 0 & 0 & I \end{pmatrix}}_{\tilde{W}_1}, 
\end{equation}
where
\begin{align}
&X_1= \begin{pmatrix} I & 0  & 0 \\ A_1 & B_1 & E_1  \end{pmatrix}^\top
\begin{pmatrix} 0 & P_1 \\ P_1 & 0 \end{pmatrix} 
\begin{pmatrix} I & 0  & 0 \\ A_1 & B_1 & E_1  \end{pmatrix},\\
&Y_1^{int}=\begin{pmatrix} 0 & I & 0 \\ C_1 & 0 & 0  \end{pmatrix}^\top
\begin{pmatrix} Q & S \\ S^\top & R \end{pmatrix}
\begin{pmatrix} 0 & I & 0 \\ C_1 & 0 & 0  \end{pmatrix}, \\
&Y_1^{ext}=\begin{pmatrix} 0 & 0 & I \\ F_1 & K_1 & L_1  \end{pmatrix}^\top
\begin{pmatrix} Q^P_1 & S^P_1 \\ (S^P_1)^\top & R^P_1 \end{pmatrix}
\begin{pmatrix} 0 & 0 & I \\ F_1 & K_1 & L_1  \end{pmatrix}.
\end{align}

Similarly, substituting \eqref{S1} into \eqref{D2}, together with \eqref{Cmatrices} and \eqref{BKmatrices}, we obtain

\begin{equation}
\label{D2A}
\tilde{W}_2^\top
X_2
\underbrace{\begin{pmatrix} I & 0 & 0 \\ 0 & \begin{pmatrix} J_1 \\ I \end{pmatrix} & 0 \\ 0 & 0 & I \end{pmatrix}}_{\tilde{W}_2}\prec
\tilde{W}_2^\top
(Y_2^{int}+Y_2^{ext})
\underbrace{\begin{pmatrix} I & 0 & 0 \\ 0 & \begin{pmatrix} J_1 \\ I \end{pmatrix} & 0 \\ 0 & 0 & I \end{pmatrix}}_{\tilde{W}_2}, 
\end{equation}
where
\begin{align}
&X_2= \begin{pmatrix} I & 0  & 0 \\ A_2 & B_2 & E_2  \end{pmatrix}^\top
\begin{pmatrix} 0 & P_2 \\ P_2 & 0 \end{pmatrix} 
\begin{pmatrix} I & 0  & 0 \\ A_2 & B_2 & E_2  \end{pmatrix},\\
&Y_2^{int}=\begin{pmatrix} 0 & I & 0 \\ C_2 & 0 & 0  \end{pmatrix}^\top
\begin{pmatrix} -R & -S^\top \\ -S & -Q \end{pmatrix}
\begin{pmatrix} 0 & I & 0 \\ C_2 & 0 & 0  \end{pmatrix}, \\
&Y_2^{ext}=\begin{pmatrix} 0 & 0 & I \\ F_2 & K_2 & L_2  \end{pmatrix}^\top
\begin{pmatrix} Q^P_2 & S^P_2 \\ (S^P_2)^\top & R^P_2 \end{pmatrix}
\begin{pmatrix} 0 & 0 & I \\ F_2 & K_2 & L_2  \end{pmatrix}.
\end{align}

Our aim is to show that we can always select $Q, S$ and $R$ with \eqref{S1} such that $X_1 \prec Y_1^{int}+Y_1^{ext}$ and $X_2 \prec Y_2^{int}+Y_2^{ext}$. Indeed, the latter two inequalities are precisely the dissipation inequalities stating that  $Q, S$ and $R$ define the interconnection neutral supply functions $s_{1,int}(v_1,w_1)$ and $s_{2,int}(v_2,w_2)$ on the interconnecting signals of $G$.

\medskip

\textbf{\emph{Step 2}: Constructing $Q$.} \newline
Consider first the inequality $X_1 \prec Y_1^{int}+Y_1^{ext}$. The inequality \eqref{D1A} implies that $X_1 \prec Y_1^{int}+Y_1^{ext}$ on $\image \tilde{W}_1$, that is, $x^\top X_1 x < x^\top (Y_1^{int}+Y_1^{ext}) x$ for all $x \in \image \tilde{W}_1$, $x \neq 0$, but not necessarily also for an arbitrary $x\neq 0$. Note that  
$\kernel \tilde{W}_1^\top = \image \tilde{\tilde{W}}_1$ where 
$\tilde{\tilde{W}}_1=\left(\begin{smallmatrix} 0 & \left(\begin{smallmatrix} I & -J_2 \end{smallmatrix}\right) & 0 \end{smallmatrix}\right)^\top$, and
$W_1:=\left(\begin{smallmatrix} \tilde{W}_1 & \tilde{\tilde{W}}_1 \end{smallmatrix}\right)$ is a nonsingular square matrix.
%$\small{\kernel \begin{pmatrix} I & 0 & 0 \\ 0 & J_2^\top & I \end{pmatrix}= \image\begin{pmatrix} 0 \\ I \\ -J_2^\top \end{pmatrix}}$, and
%$W:=\left(
%\begin{array}{c: c c} I & 0 & 0 \\ \hdashline 0 & J_2 & I \\ 0 & I & -J_2^\top \end{array}
% \right)$
%The dashed lines in the above definition of $W$ indicate partition into $2 \times 2$ matrix blocks whose dimensions are in conformity with matrix blocks in $X_1, Y_1^{int}$ and $Y_1^{ext}$, allowing for direct block-wise multiplications in expressions $W^\top X_1 W$, $W^\top Y_1^{int} W$, $W^\top Y_1^{ext} W$.

Next, we show that we can always select $Q, S$ and $R$ in \eqref{S1} so that 
\begin{equation}
\label{DesiredInequality1}
W_1^\top X_1 W_1 \prec W_1^\top (Y_1^{int}+Y_1^{ext}) W_1.
\end{equation} 
Since $W_1$ is a nonsingular square matrix, the latter inequality indeed implies $X_1 \prec Y_1^{int}+Y_1^{ext}$.

In addition to \eqref{S1}, we further constrain $Q$ by adding the following relation between $Q$ and $\tilde{Q}$

\begin{equation}
\label{NewConQ}
\begin{pmatrix} \star\end{pmatrix}^\top 
Q
\begin{pmatrix} J_2 & I \\ I & -J_2^\top \end{pmatrix}=
\begin{pmatrix}
\tilde{Q} & 0 \\ 0 & \gamma_Q I -\begin{pmatrix} I \\ -J_2^\top \end{pmatrix}^\top K_1^\top R_1^P K_1 \begin{pmatrix} I \\ -J_2^\top \end{pmatrix} 
\end{pmatrix}
\end{equation}

for some fixed real number $\gamma_Q$. For given $\tilde{Q}$ and $\gamma_Q$, the above equation uniquely defines $Q$. Note that the only constraint on $Q$ from \eqref{S1} is given by 
$\left(\begin{smallmatrix}J_2 \\ I \end{smallmatrix}\right)^\top Q \left(\begin{smallmatrix}J_2 \\ I \end{smallmatrix}\right) =\tilde{Q}$
and is also present in \eqref{NewConQ}. In that sense, uniquely defined $Q$ from \eqref{NewConQ} necessarily satisfies the constraint on $Q$ imposed by \eqref{S1}. 

With $Z_1:=X_1-Y_1^{int}-Y_1^{ext}$, the inequality \eqref{D1A} reads as $\tilde{W}_1^\top Z_1 \tilde{W}_1 \prec 0$, while the inequality \eqref{DesiredInequality1}, which we want to obtain, reads as 
\begin{equation}
\label{DesiredInequality2}
W_1^\top Z_1 W_1= \begin{pmatrix} \tilde{W}_1^\top Z_1 \tilde{W}_1 & \tilde{W}_1^\top Z_1 \tilde{\tilde{W}}_1 \\ \tilde{\tilde{W}}_1^\top Z_1 \tilde{W}_1 & \tilde{\tilde{W}}_1^\top Z_1 \tilde{\tilde{W}}_1 \end{pmatrix} \prec 0.
\end{equation}
Note that $\tilde{\tilde{W}}_1^\top Z_1 \tilde{\tilde{W}}_1=\left(\begin{smallmatrix} I \\ -J_2^\top \end{smallmatrix}\right)^\top (-Q-K_1^\top R_1^P K_1)\left(\begin{smallmatrix} I \\ -J_2^\top \end{smallmatrix}\right)$, which with \eqref{NewConQ} implies $\tilde{\tilde{W}}_1^\top Z_1 \tilde{\tilde{W}}_1=-\gamma_Q I$. Due to this fact, and since  $\tilde{W}_1^\top Z_1 \tilde{W}_1\prec 0$, the inequality \eqref{DesiredInequality2} can always be rendered feasible by taking sufficiently large positive real number $\gamma_Q$. To summarize, with sufficiently large $\gamma_Q$, the equation \eqref{NewConQ} gives us the parameter matrix $Q$ for the neutral supply rate within $G$, starting from the parameter matrix $\tilde{Q}$ of the system $\tilde{G}$ with modified interconnections.

\medskip

\textbf{\emph{Step 3}: Constructing $R$.} \newline
The inequality $X_2 \prec Y_2^{int}+Y_2^{ext}$ follows by symmetry and as a result gives us the following conditions which relate $R$ with $\tilde{R}$: 

\begin{equation}
\label{NewConR}
\begin{pmatrix} \star \end{pmatrix}^\top 
R
\begin{pmatrix} J_1 & I \\ I & -J_1^\top \end{pmatrix}=
\begin{pmatrix}
\tilde{R} & 0 \\ 0 & -\gamma_R I+\begin{pmatrix} I \\ -J_1^\top \end{pmatrix}^\top K_2^\top R_2^P K_2 \begin{pmatrix} I \\ -J_1^\top \end{pmatrix}
\end{pmatrix}
\end{equation}

for some sufficiently large positive real number $\gamma_R$.
%Due to space limitations we will not present the detailed proof of why \eqref{NewConR} satisfies. 
The procedure is completely analogous to the one for $Q$, i.e., with definitions $W_2:=\left(\begin{smallmatrix} \tilde{W}_2 & \tilde{\tilde{W}}_2 \end{smallmatrix}\right)$,$\tilde{\tilde{W}}_2=\left(\begin{smallmatrix} 0 & \left(\begin{smallmatrix} I & -J_1 \end{smallmatrix}\right) & 0 \end{smallmatrix}\right)^\top$, we can always select $\gamma_R$ in \eqref{NewConR} to impose the inequality $W_2^\top X_2 W_2 \prec W_2^\top (Y_2^{int}+Y_2^{ext}) W_2$. 

\medskip

\textbf{\emph{Step 4}: Constructing $S$.} \newline
Finally, to complete the proof, note that conditions $W_1^\top X_1 W_1 \prec W_1^\top (Y_1^{int}+Y_1^{ext}) W_1$ and $W_2^\top X_2 W_2 \prec W_2^\top (Y_2^{int}+Y_2^{ext}) W_2$, with \eqref{NewConQ} and \eqref{NewConR}, do not impose any additional constraints on $S$, that is, the only constraints on $S$ that we consider are the ones imposed by \eqref{S1}, and it is easy to see that they always have a solution. More precisely, \eqref{S1} gives $\left(\begin{smallmatrix} J_2^\top & I \end{smallmatrix}\right)
S
\left(\begin{smallmatrix} J_1 \\ I \end{smallmatrix}\right) = \tilde{S} $ as the relation between $S$ and $\tilde{S}$, which always has a solution in $S$ for any given $\tilde{S}$.
\end{proof}

\subsection{Proof of Proposition~\ref{Robust}}
\label{Sec6_C}
\begin{proof}[Proof of the part (a)]

\textbf{\emph{Step 0}: The main idea and overview.} \newline
\underline{\emph{The main idea}}: The presented proof is based on the proof of Theorem~\ref{OpenSeparationProposition_53}. The proof of Theorem~\ref{OpenSeparationProposition_53} is constructive in a sense that starting from existence of a general (full-block) multiplier $\Pi$ in \eqref{MultiplierPi}, we are able to construct a structured multiplier $\Pi_D =\left(\begin{smallmatrix} Q^D & S^D \\ (S^D)^\top & R^D \end{smallmatrix}\right)$ with $Q^D=\diag(Q^D_1,Q^D_2)$, $S^D=\diag(S^D_1,S^D_2)$, $R^D=\diag(R^D_1,R^D_2)$ (see \eqref{MultiplierPiNew}), which defines the interconnection neutral supply rates as $Q^D_2=-R^D_1$, $S^D_2=-(S^D_1)^\top$, $R^D_2=-Q^D_1$. Our goal here is to show that with $D=\diag(D_1,D_2)=0$, $K=\diag(K_1,K_2)=0$ and $L=\diag(L_1,L_2)=0$, there exists a multiplier $\Pi_D$ which, in addition to the above properties, also satisfies the condition
\begin{equation}
 \label{C} 
R^D = \begin{pmatrix} R^D_1 & 0 \\ 0& R^D_2 \end{pmatrix}\succeq 0. 
\end{equation}
Indeed, then the supply functions
\begin{equation} 
\label{sIint}
s_{i,int}:=-\left(\begin{smallmatrix} v_i \\ w_i \end{smallmatrix}\right)^\top\left(\begin{smallmatrix} Q^D_i & S^D_i \\ (S^D_i)^\top & R^D_i \end{smallmatrix}\right)\left(\begin{smallmatrix} v_i \\ w_i \end{smallmatrix}\right), \quad  i=1,2,
\end{equation}
are interconnection neutral supply functions, while \eqref{C} implies \eqref{EquNegative}. 
%The conditions \eqref{B} are in the proof of Theorem~\ref{OpenSeparationProposition_53} satisfied by construction. It therefore remains to shown that \eqref{C} is satisfied. 

\underline{\emph{Starting point and overview}}:
Note that with $D=0$, from \eqref{MultNew} we have
\begin{equation}
\label{ReducedCon}
R^D_1=-\alpha \cR_1 + (1-\alpha) \cQ_2, \quad R^D_2=\alpha \cQ_1 - (1-\alpha) \cR_2,
\end{equation}
where $0<\alpha<1$ and $\cR_1, \cR_2, \cQ_1, \cQ_2$ are defined in \eqref{cc}.
In the remainder of the proof we show that
\begin{equation}
\cR:=\begin{pmatrix} \cR_1 & 0 \\ 0 & \cR_2 \end{pmatrix} \preceq 0, \quad \cQ:=\begin{pmatrix} \cQ_1 & 0 \\ 0 & \cQ_2 \end{pmatrix} \succeq 0, 
\end{equation}
and therefore by \eqref{ReducedCon} we have \eqref{C}. This is done in 3 steps. For the first two steps, we assume that $C_1$ and $C_2$ are full row rank matrices. In \emph{Step~1} we prove the inequality $\cQ\succeq 0$, while in \emph{Step~2} we prove the inequality $\cR\preceq 0$. In \emph{Step~3} we relax the rank assumptions on $C_1$ and $C_2$.

\medskip

\textbf{\emph{Step 1}: Proving $\cQ\succeq 0$.} \newline  
From \eqref{cc} we have $\cQ=\hat{Q}$, while from \eqref{NewLabel_2} we have $\hat{Q}=\breve{Q}_{I} +\breve{Q}_{II}$. With $D=0, K=0, L=0$, from \eqref{NewLabel_1} we get 
\begin{subequations}
\label{Qdef}
\begin{align}
& \breve{Q}_I=-N_A^\top V (V^\top M V)^{-1} V^\top N_A, \label{QI} \\
& \breve{Q}_{II}=-\Big(\star \Big)^\top Z^{-1}\Big(N_B^\top V (V^\top M V)^{-1} V^\top N_A \Big), \label{QII}
\end{align}
\end{subequations}
where $Z =N_D-N_B^\top V (V^\top M V)^{-1} V^\top N_B $, $M:=A^\top P+PA+F^\top R^P F$, $N_A:=PB$, $N_B:=PE+F^\top(S^P)^\top$ and $N_D:=Q^P$.

Next we show that $M \prec 0$ and $Z \prec 0$. From there we conclude that $\breve{Q}_I \succeq 0$ (from  \eqref{QI} with $M \prec 0$) and $\breve{Q}_{II} \succeq 0$ (from \eqref{QII} with $Z \preceq 0$), and therefore $\cQ=\hat{Q} \succeq 0$.

The inequality $M \prec 0$ follows directly from \eqref{NN2_53}, since after the multiplication of the matrices in \eqref{NN2_53}, the matrix $M$ appears as a block on diagonal of the matrix on the left hand side of the inequality \eqref{NN2_53}.
The inequality $Z \prec 0$ has already been inferred in the proof of Theorem~\ref{OpenSeparationProposition_53}, see \eqref{DisCompactA}.  
   
\medskip

\textbf{\emph{Step 2}: Proving $\cR\preceq 0$.} \newline 
This part of the proof is somewhat less straightforward and relies on Lemma~\ref{Lemma1} presented in Appendix~\ref{Appendix_B}. This lemma implies that, with $D=0$, $K=0$, $L=0$, there exists a multiplier $\Pi$ from \eqref{MultiplierPi} such that \eqref{NN1_53} and \eqref{NN2_53} hold not only for $H=\left(\begin{smallmatrix} 0 & I \\ I & 0 \end{smallmatrix}\right)$, but also for all 
\[H \in \mathbf{H}:=\Big\{ \begin{pmatrix} 0 & \alpha I \\ \alpha I & 0  \end{pmatrix} \,\, : \,\, \alpha \in [0,1] \Big\}. \]
In that case from \eqref{NN1_53} we have $R \succeq 0$. To see this, take the element from $\mathbf{H}$ with $\alpha =0$ and apply \eqref{NN1_53}. We conclude that there necessarily exists a multiplier $\Pi$ from \eqref{MultiplierPi} such that \eqref{NN1_53} and \eqref{NN2_53} hold, while $R\succeq 0$.

With $D=0$ the condition \eqref{QSRinequality_53} becomes
\[\begin{pmatrix} Q & S \\ S^\top & R \end{pmatrix} + \begin{pmatrix}  \cQ & \cS \\ \cS^\top & \cR  \end{pmatrix} \prec 0.\]
Since $R \succeq 0$, we have $\cR \preceq 0$.

\medskip

\textbf{\emph{Step 3}: Relaxing the full row rank assumption for $C_1$ and $C_2$.} \newline 
We make use of the proof of Proposition~\ref{PropositionRank}. Recall that in the proof of Proposition~\ref{PropositionRank} we have first concluded existence of interconnection supply rates on the modified system which had full row rank matrices $\tilde{C}_1$ and $\tilde{C}_2$. The interconnection signals in the modified system had lower dimension than the corresponding interconnection signals of the original systems. Then we showed that starting from the matrices $\tilde{Q}$, $\tilde{R}$ and $\tilde{S}$, which define the interconnection neutral supply functions of the modified system, as presented in \eqref{s1int} and \eqref{s2int}, we can construct interconnection neutral supply functions for the original systems. These supply functions are defined through matrices $Q$, $R$ and $S$, which are respectively ``extensions'' of the matrices $\tilde{Q}$, $\tilde{R}$ and $\tilde{S}$. The extensions for $Q$ and $R$ are given by \eqref{NewConQ}, and \eqref{NewConR}, respectively. Comparing \eqref{s1int} and \eqref{s2int} with \eqref{sIint}, we have that $R_1^D$ and $R_2^D$ from \eqref{C} correspond respectively to $-\tilde{R}$ and $\tilde{Q}$ from \eqref{s1int} and \eqref{s2int}. To finalize the proof for relaxing the rank conditions of $C_i$ it is therefore sufficient to show that \emph{i}) using extension \eqref{NewConQ} on $\tilde{Q} \succeq 0$ we can obtain $Q$ from \eqref{NewConQ} such that $Q \succeq 0$; and \emph{ii}) using extension \eqref{NewConR} on $\tilde{R} \preceq 0$ we can obtain $R$ from \eqref{NewConR} such that $R \preceq 0$. Indeed, both (i) and (ii) can be achieved when selecting sufficiently large $\gamma_Q$ and $\gamma_R$, since the matrices in \eqref{NewConQ} and \eqref{NewConR} are related with congruence transformations. \\ 
%\end{proof}
%\begin{proof}[Proof of the part (b)]\newline
\newline

\emph{Proof of the part (b):}

\textbf{\emph{Step 0}: The main idea and overview.} \newline
We present proof in the case of full row rank matrices $C_1$ and $C_2$. The proof for relaxation of this assumption is fully analogous to the proof presented for the part (a) above.

\emph{\underline{The main idea and overview}}: Following the same path as in the proof for the part (a) presented above, the proof again boils down to proving inequality \eqref{C}. Note that \eqref{C} is satisfied if $\cR \preceq 0$ and $\cQ \succeq 0$, due to \eqref{ReducedCon}.

The remaining part of the proof is divided in 3 steps. In \emph{Step~1} we give an auxiliary result which will be instrumental in both \emph{Step~2} and \emph{Step~3}, in which the inequalities $\cR \preceq 0$ and $\cQ \succeq 0$ are proven, respectively.

\medskip

\textbf{\emph{Step 1}: Preliminaries and auxiliary results} \newline
Recall that the proof of Corollary~\ref{Cor1} follows directly from the proof of Theorem~\ref{OpenSeparationProposition_53} when we omit all the terms related to exogenous inputs/outputs $d_1, d_2, z_1, z_2$ and the external supply functions $s_{1,ext}(\cdot,\cdot)$ and $s_{2,ext}(\cdot,\cdot)$, that is, by taking 
\begin{equation}
\label{ZeroMatrices}
E=0,\, F=0,\, K=0,\, L=0,\, Q^P=0,\, S^P=0,\, R^P=0
\end{equation}
in \eqref{NN2_53} and in the remainder of the proof of Theorem~\ref{OpenSeparationProposition_53}. 
Additionally, the statement of Corollary~\ref{Cor1} requires $P=\diag(P_1,P_2) \succ 0$, where $P_1$ and $P_2$ are the matrices defining $V_1(\cdot)$ and $V_2(\cdot)$, i.e., $V_1(x_1)=x_1^\top P_1 x_1$, $V_2(x_2)=x_2^\top P_2 x_2$.
Finally, we consider the case when $D=\diag(D_1,D_2)=0$.  
  
%Following the same path as in the proof for the part (a) presented above, the proof again boils down to proving inequality \eqref{C}. Note that \eqref{C} is satisfied if $\cR \preceq 0$ and $\cQ \succeq 0$, due to \eqref{ReducedCon}.  

From \eqref{labelM}, with \eqref{ZeroMatrices},  we have

\begin{equation}
M=A^\top P +P A=\begin{pmatrix} A_1^\top P_1 + P_1 A_1 & 0 \\ 0 & A_2^\top P_2 + P_2 A_2\end{pmatrix}=:\begin{pmatrix} M_1 & 0 \\ 0 & M_2 \end{pmatrix}. 
\end{equation} 
  
%From \eqref{ZeroMatrices} we have $M=A^\top P + P A$. With the abbreviations $M_1:=A_1^\top P_1 + P_1 A_1$, $M_2:=A_2^\top P_2 + P_2 A_2$ we have $M=\diag(M_1,M_2)$. 
Next we show that $M\prec 0$. Let $x=\col(x_1,x_2)$ denote the state vector of the interconnected system, which can be represented as $\dot{x}=\cA x$ with $\cA=\left(\begin{smallmatrix} A_1 & B_1 C_2 \\ B_2 C_1 & A_2 \end{smallmatrix}\right)$. Since $V(x)=x^\top P x$ is a Lyapunov function for the interconnected system, we have $\cA^\top P + P \cA \prec 0$. It remains to note that $\cA^\top P + P \cA=\left(\begin{smallmatrix} M_1 & * \\ * & M_2 \\  \end{smallmatrix}\right)$ (here $*$ denote submatrices which are not of interest) and therefore we conclude that $M_1 \prec 0$, $M_2 \prec 0$, which implies $M \prec 0$.

\medskip

\textbf{\emph{Step 2}: Proving $\cQ\succeq 0$.} \newline
From \eqref{cc}, \eqref{NewLabel_2}, \eqref{NewLabel_1} we have $\cQ=\breve{Q}_I+\breve{Q}_{II}$, which with \eqref{ZeroMatrices} reduces to
\begin{equation}
\cQ=-N_A^\top V(V^\top M V)^{-1} V^\top N_A
\end{equation}
with $N_A=PB$. We have $M\prec 0 \,\, \implies (V^\top M V)^{-1} \prec 0 \implies \cQ \succeq 0$.

\medskip

\textbf{\emph{Step 3}: Proving $\cR\preceq 0$.} \newline
It is possible to prove this inequality following a similar path as done in the proof for the part (a). Here we present an alternative and somewhat more direct proof. From \eqref{ccA} and \eqref{LRL} we have that $\cR\preceq 0$ is implied if $\hat{R} \preceq 0$. Note that $\hat{R}$ is defined in \eqref{NewLabel_2}. With \eqref{ZeroMatrices}, from \eqref{NewLabel_2} and \eqref{NewLabel_1} we have
\begin{equation}
\label{hatR1}
\hat{R}=M-MV(V^\top M V)^{-1}V^\top M,
\end{equation}  
where $V$ is a tall matrix\footnote{The number of rows in larger than the number of columns.} with full column rank. Indeed, recall from \emph{Step~1} in the proof of Theorem~\ref{OpenSeparationProposition_53} that $V:=\diag(V_1,V_2)$, where $V_1$ and $V_2$ are defined as full column rank matrices whose columns span the kernel spaces of $C_1$ and $C_2$, respectively. 
%The desired inequality $\hat{R} \preceq 0$ follows from \eqref{hatR1} and Lemma~\ref{Lemma2} presented in Appendix~\ref{Appendix_C}. 

Since $M\prec 0$, we have
\begin{equation}
X:=\begin{pmatrix}M & MV \\ V^\top M & V^\top M V\end{pmatrix}
=\begin{pmatrix} \star \end{pmatrix}^\top
\begin{pmatrix} M & 0 \\ 0 & M \end{pmatrix}
\begin{pmatrix} I & 0 \\ 0 & V \end{pmatrix} \prec 0.
\end{equation}  
The above inequality implies, via the Schur complement rule, that $M-MV(V^\top M V)^{-1}V^\top M \prec 0$, that is, $\hat{R} \prec 0$. Therefore, we also have the desired (weaker) inequality $\hat{R} \preceq 0$.
\end{proof} 


%%%%%%%%%%%%%%%%%%%%%%%%%%%%%%%%%%%%%%%%%%%%%%%%%%%%%%%%%%%%%%%%%%%%%%%%%%%%%%%%

\section{CONCLUSIONS}
\label{Sec7}
The main results of this paper give insights into interplay between structured storage/Lyapunov functions for a class of interconnected systems and dissipativity properties of the individual systems. These results complement some of the results from the seminal papers \cite{Willems_1, Willems_2} of Jan Willems, making suitable converse statements to Theorem~5 in \cite{Willems_1}. More precisely, we have proven that if a dynamical network, composed as a set of linear time invariant systems interconnected over an acyclic graph, admits an additive Lyapunov function, then the individual systems in the network are dissipative with respect to a (nonempty) set of interconnection neutral supply functions. Each supply function from this set is defined on a single interconnection link in the network. 

From a more practical/application oriented point of view, the presented results can be used to relate certain analysis or controller synthesis approaches for uncertain and/or large scale systems, precisely by relating structural properties of Lyapunov functions with dissipativity properties. Indeed, while some control synthesis approaches build on employment of additive Lyapunov functions (see, e.g., \cite{Siljak, Zheng}), others  directly use interconnection neutral supply functions, e.g., \cite{Langbort}. Furthermore, the relation between interconnection neutral supply functions and robustness properties, which have also be presented in this paper, has a  potential to become a constructive element in both analysis and control synthesis. 

%Assumptions which define the class of systems considered in this paper are imposed on either direct feed-through matrices (matrices ``$D$'' in state space representations) or output matrices (matrices ``$C$''). We might argue that they are mild in a sense that most of the systems in practice satisfy these assumptions, as in practice very often there are no direct feed-through paths between systems in a network. However, the most restrictive assumption of the paper is the assumption that the interconnection graph of a network is acyclic. Whether the results can be generalized to a more general networks with cycles, remains an open question. 

%%%%%%%%%%%%%%%%%%%%%%%%%%%%%%%%%%%%%%%%%%%%%%%%%%%%%%%%%%%%%%%%%%%%%%%%%%%%%%%%

%\appendices

\section*{Appendix A: Robust stability and robust dissipativity}
\label{Appendix_A}
Let the system from Figure~\ref{RobPerf}~a) be defined by
\begin{equation}
\label{Sys_closedA}
\tilde{G}_0: \quad \begin{pmatrix}\dot{x} \\ w  \end{pmatrix} 
=
\begin{pmatrix} A & B  \\ C & D  \end{pmatrix}
\begin{pmatrix} x \\ v \end{pmatrix},
\end{equation}
and
\begin{equation}
\label{Sys_closedB}
v  = H w, \quad H\in \mathbf{H},
\end{equation}
where $A \in \Rset^{n\times n}, D\in \Rset^{n_w \times n_v}$, and $\mathbf{H}$ is a set of matrices in $\Rset^{n_v \times n_w}$. Let $\tilde{G}$ denote the overall interconnected system given by \eqref{Sys_closedA}, \eqref{Sys_closedB}. Furthermore, let the system from Figure~\ref{RobPerf}~b) be defined by
\begin{equation}
\label{SysRobPerfA1}
G_0: \quad \begin{pmatrix}\dot{x} \\ w \\ z \end{pmatrix} 
=
\begin{pmatrix} A & B & E \\ C & D & M \\ K & F & L\end{pmatrix}
\begin{pmatrix} x \\ v \\ d \end{pmatrix},
\end{equation}
and
\begin{equation}
\label{SysRobPerfB}
v  = H w, \quad H\in \mathbf{H},
\end{equation}
where $A \in \Rset^{n\times n}, D\in \Rset^{n_w \times n_v}, L \in \Rset^{n_z \times n_d}$ and $\mathbf{H}$ is a set of matrices in $\Rset^{n_v \times n_w}$. Let $G$ denote the overall interconnected system given by \eqref{SysRobPerfA1}, \eqref{SysRobPerfB}, with $d$ and $z$ as input and output, respectively.

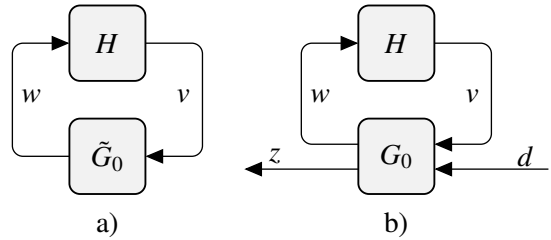
\begin{figure}[h!]
\centering
 \begin{tikzpicture}%[scale=0.9]

%%%%%%%%%%%%%%%%%%%%%%%%%%%%%%%%%%%%%%%%%%%%%%%%%%%%%%%%%%%%%%%%%%%%%%%%%%%%%%% 

\draw [line width=0.7pt, rounded corners] (7-3.8,2) rectangle (8-3.8,3);
\draw [line width=0.7pt, rounded corners] (7-3.8,3.5) rectangle (8-3.8,4.5);
\draw [line width=0.5pt, ->, >= triangle 45, rounded corners] (7-3.8,2.5) -- (6.25-3.8,2.5) -- (1.25-3.8+5,4)--(7-3.8,4);
\draw [line width=0.5pt, ->, >= triangle 45, rounded corners] (8-3.8,4)-- (8.75-3.8,4) -- (8.75-3.8,2.5)-- (8-3.8,2.5) ;
\node at (7.5-3.8,2.5) {$\tilde{G}_0$};
\node at (7.5-3.8,4) {$H$};
%\node at (9.2,2.5) {$d$};
%\node at (5.9,2.5) {$z$};
\node at (6.5-3.8,3.3) {$w$};
\node at (8.5-3.8,3.3) {$v$};
\begin{pgfonlayer}{background}
\filldraw [line width=0.1mm,rounded corners,black!5] (7-3.8,2) rectangle (8-3.8,3);
\end{pgfonlayer}
\begin{pgfonlayer}{background}
\filldraw [line width=0.1mm,rounded corners,black!5] (7-3.8,3.5) rectangle (8-3.8,4.5);
\end{pgfonlayer} 
 \node at (7.5-3.8,1.6) {a)};
%%%%%%%%%%%%%%%%%%%%%%%%%%%%%%%%%%%%%%%%%%%%%%%%%%%%%%%%%%%%%%%%%%%%%%%%%%%%%%% 
\draw [line width=0.7pt, rounded corners] (7,2) rectangle (8,3);
\draw [line width=0.7pt, rounded corners] (7,3.5) rectangle (8,4.5);
\draw [line width=0.5pt, ->, >= triangle 45, rounded corners] (7,2.33) -- (5.5,2.33);
\draw [line width=0.5pt, ->, >= triangle 45, rounded corners] (9.5,2.33) -- (8,2.33);
\draw [line width=0.5pt, ->, >= triangle 45, rounded corners] (7,2.66) -- (6.25,2.66) -- (1.25+5,4)--(7,4);
\draw [line width=0.5pt, ->, >= triangle 45, rounded corners] (8,4)-- (8.75,4) -- (8.75,2.66)-- (8,2.66) ;
\node at (7.5,2.5) {$G_0$};
\node at (7.5,4) {$H$};
\node at (9.2,2.5) {$d$};
\node at (5.9,2.5) {$z$};
\node at (6.5,3.3) {$w$};
\node at (8.5,3.3) {$v$};
\node at (7.5,1.6) {b)};
\begin{pgfonlayer}{background}
\filldraw [line width=0.1mm,rounded corners,black!5] (7,2) rectangle (3+5,3);
\end{pgfonlayer}
\begin{pgfonlayer}{background}
\filldraw [line width=0.1mm,rounded corners,black!5] (7,3.5) rectangle (3+5,4.5);
\end{pgfonlayer}
 \end{tikzpicture}
\caption{a) Autonomous uncertain system; b) Uncertain open system.}
\label{RobPerf}
\end{figure}
%%%%%%%%%%%%%%%%%%%%%%%%%%%%%%%%%%%%%%%%%%%%%%%%%%%%%%%%%%%%%%%%%%%%%%%%%%%%%%%%%%
%%%%%%%%%%%%%%%%%%%%%%%%%%%%%%%%%%%%%%%%%%%%%%%%%%%%%%%%%%%%%%%%%%%%%%%%%%%%%%%%%
The system $\tilde{G}$ is robustly exponentially stable if it is well-posed and if there exists a quadratic function $V(x)=x^\top P x$ with $P\succ 0$ such that $\dot{V}(x(t))<0$ for all $x(t)\neq 0$ and for all $H \in \mathbf{H}$.
The system $G$ is robustly strictly dissipative with respect to the supply function 
\begin{equation}
\label{SupplyPerformance}
s(d,z)=\begin{pmatrix}
  d\\
  z\\
\end{pmatrix}^{\top}
\begin{pmatrix}
  Q_P & S_P\\
  S_P^{\top} & R_P\\
\end{pmatrix}
\begin{pmatrix}
  d\\
  z\\
\end{pmatrix}
\end{equation}
if it is well-posed and if there exists a quadratic storage function $V(x)=x^\top P x$, such that $\dot{V}(x(t)) < s(d(t),z(t))$ at each time $t$, for all $H \in \mathbf{H}$ and for all $\col(x(t),d(t))\neq 0$.
The system $G$ is robustly exponentially stable if it is well-posed and if there exists a quadratic function $V(x)=x^\top P x$ with $P\succ 0$ such that $\dot{V}(x(t))<0$ for all $x(t)\neq 0$, for $d(t)=0$ and for all $H \in \mathbf{H}$.

\begin{theorem}
\label{RobStability}
Suppose that $\mathbf{H}$ is a compact set. Then the system $\tilde{G}$ is robustly exponentially stable if and only if there exist symmetric matrices $P\succ 0$, $Q$, $R$ and a real matrix $S$ such that the following inequalities are satisfied

\begin{subequations}
\label{StabLMI}
\begin{align}
 & \quad \quad \quad \,\,\, \begin{pmatrix} H \\ I \end{pmatrix}^\top
 \begin{pmatrix} Q & S \\ S^\top & R \end{pmatrix} 
 \begin{pmatrix} H \\ I \end{pmatrix} \succeq 0, , \quad \text{for all}\,\, H \in \mathbf{H},\label{StabLMIa} \\
&\left(\begin{array}{c} 
\star 
\end{array}\right)^\top
\left(\begin{array}{cc:cc} 
0 & P & 0 & 0  \\ 
P & 0 & 0 & 0  \\ \hdashline 
0 & 0 & Q & S  \\
0 & 0 & S^\top & R
\end{array}\right)
\left(\begin{array}{cc} 
I & 0 \\ 
A & B  \\ \hdashline 
0 & I  \\
C & D 
\end{array}\right) \prec 0.  \label{StabLMIb}
\end{align}
\end{subequations}
 
\end{theorem}

\begin{theorem}
\label{PerformanceSProcedure}
Let $Q_P, S_P, R_P$ be given real matrices, where $Q_P$ and $R_P$ are symmetric, and suppose $\mathbf{H}$ is a compact set. Consider the following inequalities

\begin{subequations}
\label{PerfLMI2}
\begin{align}
 & \quad \quad \quad \,\,\, \begin{pmatrix} H \\ I \end{pmatrix}^\top
 \begin{pmatrix} Q & S \\ S^\top & R \end{pmatrix} 
 \begin{pmatrix} H \\ I \end{pmatrix} \succeq 0, , \quad \text{for all}\,\, H \in \mathbf{H},\label{PerfLMI2a} \\
&\left(\begin{array}{ccc} 
\star 
\end{array}\right)^\top
\left(\begin{array}{cc:cc:cc} 
0 & P & 0 & 0 & 0 & 0 \\ 
P & 0 & 0 & 0 & 0 & 0 \\ \hdashline 
0 & 0 & Q & S & 0 & 0 \\
0 & 0 & S^\top & R & 0 & 0 \\ \hdashline 
0 & 0 & 0 & 0 & -Q_P & -S_P \\
0 & 0 & 0 & 0 & -S_P^\top & -R_P
\end{array}\right)
\left(\begin{array}{ccc} 
I & 0 & 0 \\ 
A & B & E \\ \hdashline 
0 & I & 0 \\
C & D & M \\ \hdashline 
0 & 0 & I\\
K & F & L 
\end{array}\right) \prec 0.  \label{PerfLMI2b}
\end{align}
\end{subequations}
 
\begin{enumerate}[(a)]
\item Suppose $-R_P \succeq 0$. Then the system $G$ is well-posed, robustly exponentially stable and robustly strictly dissipative with respect to supply \eqref{SupplyPerformance}, if and only if there exist symmetric matrices $P\succ 0$, $Q$ and $R$ and a real matrix $S$ such that the inequalities \eqref{PerfLMI2} hold.
\item Suppose $G$ is well-posed. Then the system $G$ is robustly strictly dissipative with respect to supply \eqref{SupplyPerformance} if and only if there exist symmetric matrices $P$, $Q$ and $R$ and a real matrix $S$ such that the inequalities \eqref{PerfLMI2} hold.
\end{enumerate}
\end{theorem}

We will refer to the matrix $\begin{pmatrix} Q & S \\ S^\top & R \end{pmatrix}$ in Theorem~\ref{RobStability} and Theorem~\ref{PerformanceSProcedure} as the \emph{multiplier}, as commonly done in the literature. The following remarks are in order.

When \eqref{PerfLMI2} holds, the matrix $P$ defines the function $V(x)=x^\top P x$ as a storage function with respect to the supply \eqref{SupplyPerformance}.
The assumption $-R_P \succeq 0$ in (\emph{a}) is instrumental to infer well-posedness from \eqref{PerfLMI2}. Both assumptions $-R_P \succeq 0$ and $P \succ 0$ in (\emph{a}) are instrumental to infer robust stability from \eqref{PerfLMI2}. Also note that adding the condition $P\succ 0$ alone to (\emph{b}) does not necessarily imply robust stability.

Proofs of Theorem~\ref{RobStability} and Theorem~\ref{PerformanceSProcedure} are based on the full block S-procedure and can be directly found or deduced (statement (b) from Theorem~\ref{PerformanceSProcedure}) from \cite{SchererLPV, SchererFromBook}. 
The condition on compactness of $\mathbf{H}$ can be suitably relaxed, as shown recently in \cite{JokicNakic_Automatica}, but this is not used in this paper.

%%%%%%%%%%%%%%%%%%%%%%%%%%%%%%%%%%%%%%%%%%%%%%

%%%%%%%%%%%%%%%%%%%%%%%%%%%%%%%%%%%%%%%%%%%%%%%%%%%%%%%%%%%%%%%%%%%%%%%

\section*{Appendix B}
\label{Appendix_B}

Consider interconnected system $G$ given by \eqref{Systems_53}-\eqref{Interconnection_53}  with $D_i=0$, $K_i=0$, $L_i=0$, $i=1,2$, and as presented in Figure~\ref{NeutralSupply}. With $x:=\col(x_1,x_2)$, $d:=\col(d_1,d_2)$, $z=\col(z_1,z_2)$, we have that $G$ is given by
\begin{equation}
\label{Gapp}
\dot{x}=Ax+Ed, \quad z=Fx,
\end{equation} 
and $E=\diag(E_1,E_2)$, $F=\diag(F_1,F_2)$, $A=\left(\begin{smallmatrix} A_1 & B_1 C_2 \\ B_2 C_1 & A_2 \end{smallmatrix}\right)$.
\begin{lemma}
\label{Lemma1}
Consider the interconnected system $G$ given by \eqref{Gapp} and suppose it is strictly dissipative with an additive quadratic supply function 
$
s_{\text{ext}}(d_1,d_2,z_1,z_2)=s_{1,ext}(d_1,z_1)+s_{2,ext}(d_2,z_2)
$
with an additive quadratic storage function $V(x)=x_1^\top P_1 x_1+x_2^\top P_2 x_2$. Let the system $G_{\alpha}$ be defined from $G$ when the matrix $A$ from \eqref{Gapp} is replaced by $A(\alpha):=\left(\begin{smallmatrix}A_1 & \alpha B_1 C_2 \\ \alpha B_2 C_1 & A_2\end{smallmatrix}\right)$
for $\alpha$ an arbitrary real constant such that $0 \leq \alpha \leq 1$. Then $G_{\alpha}$ is strictly dissipative with respect to the same additive supply function $s_{\text{ext}}(d_1,d_2,z_1,z_2)$ and with the same additive storage function $V(x)$.
\end{lemma}

\begin{proof}Let $s_{ext}(d,z):=\left(\begin{smallmatrix} d \\ z \end{smallmatrix}\right)^\top 
\left(\begin{smallmatrix} -Q^P & -S^P \\ -(S^P)^\top & -R^P \end{smallmatrix}\right)\left(\begin{smallmatrix} d \\ z \end{smallmatrix}\right)$. Strict dissipativity of $G$ with respect to $s_{ext}(d,z)$ and with storage function $V(x)=x^\top P x$ implies that the following matrix inequality holds

\begin{equation}
\label{DissLemma}
\begin{pmatrix}\star\end{pmatrix}^\top
\begin{pmatrix}0 & P \\ P & 0\end{pmatrix}
\begin{pmatrix}I & 0 \\ A & E\end{pmatrix}+
\begin{pmatrix}\star\end{pmatrix}^\top
\begin{pmatrix}Q^P & S^P \\ (S^P)^\top & R^P\end{pmatrix}
\begin{pmatrix}0 & I \\ F & 0\end{pmatrix} \prec 0.
\end{equation}

%\begin{equation}
%\label{DissLemma}
%\left(\begin{array}{cc} I & 0 \\ A & E \\ \hdashline 0 & I \\ F & 0 \end{array}\right)^\top
%\left(\begin{array}{cc:cc} 0 & P & 0 & 0 \\ 
%P & 0 & 0 & 0 \\ \hdashline 
%0 & 0 & Q^P & S^P \\ 
%0 & 0 & (S^P)^\top & R^P \end{array}\right)
%\left(\begin{array}{cc} I & 0 \\ A & E \\ \hdashline 0 & I \\ F & 0 \end{array}\right) \prec 0.
%\end{equation}
Note that the hypothesis that both supply and storage functions are additive implies that the matrices $P$, $Q^P$, $S^P$ and $R^P$ are block diagonal ($P:=\diag(P_1,P_2)$). The dissipativity condition \eqref{DissLemma}, after multiplications, reads as
$\left(\begin{smallmatrix}
A^\top P + P A + F^\top R^P F & P E + F^\top (S^P)^\top \\
E^\top P + S^P F & Q^P 
\end{smallmatrix}\right) \prec 0.$
Applying the Schur complement rule on the above inequality we obtain $A^\top P + P A + Z \prec 0$, where $Z=F^\top R^P F - (P E + F^\top (S^P)^\top) (Q^P)^{-1}(E^\top P + S^P F)$.
%\begin{equation}
%A^\top P + P A + \underbrace{F^\top R^P F - (P E + F^\top (S^P)^\top) (Q^P)^{-1}(E^\top P + S^P F)}_Z   \prec 0,
%\end{equation}
Note that $Z$ is by construction block diagonal matrix, i.e., we have $Z=\diag(Z_1,Z_2)$.
%\begin{equation}
%Z=\begin{pmatrix} Z_1 & 0 \\ 0 & Z_2 \end{pmatrix}.
%\end{equation}
We have
\[
A(\alpha)^\top P + P A(\alpha)=\begin{pmatrix} A_1^\top P_1 + P_1 A_1 & \alpha (C_1^\top B_2^\top P_2 + P_1 B_1 C_2) 
\\ \star & A_2^\top P_2 + P_2 A_2 \end{pmatrix}.
\] 
Next we show that $A(\alpha)^\top P + P A(\alpha) + Z \prec 0$ for $\alpha =1$ implies that this inequality holds also for all $\alpha \in [0,1)$. 
After applying the Schur complement rule on the inequality $A(\alpha)^\top P + P A(\alpha) + Z \prec 0$, we have 
$A_1^\top P_1 + P_1 A_1 + Z_1 - \alpha^2 \Lambda \prec 0$ with
$\Lambda=(\star)^\top(A_2^\top P_2 + P_2 A_2 + Z_2)^{-1}(C_2^\top B_1^\top P_1 + P_2 B_2 C_1)$. Therefore $A_1^\top P_1 + P_1 A_1 + Z_1 \prec \alpha^2 \Lambda \preceq 0$.
Since $\Lambda \prec 0$ and the above inequality holds for $\alpha =1$, it also holds for any  $\alpha \in [0,1)$.
\end{proof}
\bibliography{references}
\bibliographystyle{plain}

%%%%%%%%%%%%%%%%%%%%%%%%%%%%%%%%%%%%%%%%%%%%%%%%%%%%%%%%%%%%%%%%%
%%%%%%%%%%%%%%%%%%%%%%%%%%%%%%%%%%%%%%%%%%%%%%%%%%%%%%%%%%%%%%%%%

\end{document}